\documentclass[12pt]{amsart}
\usepackage{amssymb,amsmath,tabularx,mathrsfs,mathbbol,yfonts,upgreek}
\usepackage{amsthm,verbatim,comment}
\usepackage[bookmarks=true]{hyperref}
\usepackage{pstricks,pst-node,pst-plot}
\usepackage{geometry}
\usepackage{stmaryrd}
\usepackage{paralist}
\usepackage[all]{xy}
\usepackage{array}
\usepackage{url}
\usepackage{longtable}
\numberwithin{equation}{section}

\DeclareSymbolFontAlphabet{\mathbb}{AMSb}
\DeclareSymbolFontAlphabet{\mathbbl}{bbold}

\newtheorem{thm}{Theorem}[section]

\newtheorem{lem}[thm]{Lemma}

\newtheorem{cor}[thm]{Corollary}

\theoremstyle{definition}
\newtheorem{defn}[thm]{Definition}
\newtheorem{nota}[thm]{Notation}
\newtheorem{eg}[thm]{Example}

\newtheorem*{rem*}{Remarks}

\newtheoremstyle{case}{}{}{}{}{}{:}{ }{}
\theoremstyle{case}

\newcommand{\F}{\mathbb{F}}

\title[On Terwilliger $\F$-algebras of factorial association schemes]{On Terwilliger $\F$-algebras of factorial association schemes}
\begin{document}
\author{Yu Jiang}
\address[Y. Jiang]{School of Mathematical Sciences, Anhui University (Qingyuan Campus), No. 111, Jiulong Road, Hefei, 230601, China}
\email[Y. Jiang]{jiangyu@ahu.edu.cn}
%\author{Kay Jin Lim}
%\address[K. J. Lim]{Division of Mathematical Sciences, School of Physical and Mathematical Sciences, Nanyang Technological University, 21 Nanyang Link, Singapore 637371.}
%\email[K. J. Lim]{limkj@ntu.edu.sg}
\begin{abstract}
The Terwilliger algebras of association schemes over an arbitrary field $\F$ were called the Terwilliger $\F$-algebras of association schemes in \cite{J2}. In this paper, we study the Terwilliger $\F$-algebras of factorial association schemes. We determine the $\F$-dimensions, the centers, the semisimplicity, the Jacobson radicals, and the algebraic structures of the Terwilliger $\F$-algebras of factorial association schemes.
\end{abstract}
\maketitle
\vspace{-1.5em}
\section{Introduction}
Association schemes on nonempty finite sets, briefly called schemes, have already been intensively studied as important research objects in algebraic combinatorics. In particular, many different tools have been introduced to study the scheme theory.

The subconstituent algebras of commutative schemes, introduced by Terwilliger in \cite{T1}, are new algebraic tools of investigating schemes. They are finite-dimensional semisimple associative $\mathbb{C}$-algebras and are widely known as the Terwilliger algebras of commutative schemes. In general, the Terwilliger algebras can also be defined for an arbitrary scheme and an arbitrary commutative unital ring (see \cite{Han}). Following \cite{J2}, the Terwilliger algebras of schemes over an arbitrary field $\F$ shall be called the Terwilliger $\F$-algebras of schemes. Therefore the Terwilliger algebras of commutative schemes are precisely the Terwillliger $\mathbb{C}$-algebras of these commutative schemes.

The Terwilliger $\mathbb{C}$-algebras of many commutative schemes have been extensively studied (for example, see \cite{BST}, \cite{CD}, \cite{LMP}, \cite{LM}, \cite{LMW}, \cite{M}, \cite{T1}, \cite{T2}, \cite{T3}, \cite{TY}). However, the investigation of the Terwilliger $\F$-algebras of schemes is almost completely open (see \cite{Her}). In this paper, we investigate the Terwilliger $\F$-algebras of factorial schemes. In particular, we determine the $\F$-dimensions, the centers, the semisimplicity, the Jacobson radicals, the algebraic structures of the Terwilliger $\F$-algebras of factorial schemes (see Theorems \ref{T;Dimension}, \ref{T;Center}, \ref{T;Semisimplicity}, \ref{T;Jacobson}, \ref{T;Structure}, respectively). Since the factorial schemes are precisely the direct products of one-class schemes (see \cite[Page 344]{B}), these main results contribute to studying the question in \cite[Conclude remarks (2)]{BST}.

The organization of this paper is as follows: In Section 2, we introduce the basic notation and the required preliminaries. In Section 3, we determine all closed subsets and strongly normal closed subsets of the factorial schemes. In Section 4, we prove Theorem \ref{T;Dimension} and give two $\F$-bases to the Terwilliger $\F$-algebras of factorial schemes. In Sections 5, 6, 7, we finish the proofs of Theorems \ref{T;Center}, \ref{T;Semisimplicity}, \ref{T;Jacobson}, respectively. In Sections 8 and 9, we deduce some equalities and use them to prove Theorem \ref{T;Structure}.
\section{Basic notation and preliminaries}
For a general theory on association schemes, the reader may refer to \cite{B} or \cite{Z}.
\subsection{Conventions}
Throughout this paper, fix a field $\F$ of characteristic $p$. Let $\mathbb{N}_{0}$ be the set of all nonnegative integers. If $g, h\in \mathbb{N}_{0}$, set $[g,h]=\{a: g\leq a\leq h\}\subseteq \mathbb{N}_0$. Fix a nonempty finite set $\mathbb{X}$. Every association scheme on $\mathbb{X}$ mentioned in this paper is briefly called a scheme. The addition, multiplication, and scalar multiplication of matrices displayed in this paper are the usual matrix operations. For a subset $\mathbb{U}$ of an $\F$-linear space $\mathbb{V}$, let $\langle \mathbb{U}\rangle_\F$ denote the $\F$-linear subspace of $\mathbb{V}$ spanned by $\mathbb{U}$. Every algebra mentioned in this paper is a finite-dimensional associative unital $\F$-algebra. Every module mentioned in this paper is a finitely generated left module.
\subsection{Schemes}
Let $\mathbb{S}=\{R_0, R_1, \ldots, R_d\}$ denote a partition of the cartesian product $\mathbb{X}\times \mathbb{X}$. Then $\mathbb{S}$ is called a $d$-class scheme if the following  conditions hold together:
\begin{enumerate}[(i)]
\item $R_0=\{(a,a): a\in \mathbb{X}\}$;
\item For any $g\in [0,d]$, there is $g'\in [0,d]$ such that $R_{g'}=\{(a,b): (b,a)\in R_g\}\in \mathbb{S}$;
\item For any $g,h,i\in [0,d]$ and $(x,y), (\widetilde{x}, \widetilde{y})\in R_i$, there exists $p_{gh}^i\in \mathbb{N}_0$ such that
$$p_{gh}^i=|\{a: (x,a)\in R_g,\ (a,y)\in R_h\}|=|\{a: (\widetilde{x}, a)\in R_g,\ (a, \widetilde{y})\in R_h\}|.$$
\end{enumerate}

Throughout this paper, $\mathbb{S}=\{R_0, R_1,\dots, R_d\}$ denotes a fixed $d$-class scheme. The scheme $\mathbb{S}$ is symmetric if $g=g'$ for any $g\in [0,d]$. The scheme $\mathbb{S}$ is commutative if $p_{gh}^i=p_{hg}^i$ for any $g, h, i\in [0,d]$. Every symmetric scheme is commutative. For any $x\in \mathbb{X}$ and $g\in [0,d]$, set $k_g=p_{gg'}^0$ and $xR_g\!=\!\{a: (x, a)\in R_g\}$. Call $k_g$ the valency of $R_g$ and notice that $|xR_g|\!=\!k_g$. As $x$ is chosen from $\mathbb{X}$ arbitrarily and $R_g\!\neq\! \varnothing$, $k_g>0$. For any $g, h, i\in [0,d]$ and $y\in xR_i$, $|xR_g\cap yR_h|\!=\!p_{gh'}^i$. The scheme $\mathbb{S}$ is triply regular if, for any $g, h, i, j, k, \ell\in [0,d]$, $y\in xR_j$, and $z\in xR_k\cap yR_\ell$, $|xR_g\cap yR_h\cap zR_i|$ only depends on $g, h, i, j, k, \ell$ and is independent of the choices of elements in $R_j$, $R_k$, $R_\ell$.

For any nonempty subsets $\mathbb{U}, \mathbb{V}$ of $\mathbb{S}$, set $\mathbb{UV}\!=\!\{R_a: \exists\ R_b\in \mathbb{U}, \exists\ R_c\in \mathbb{V}, p_{bc}^a\!>\!0\}$. Notice that  $\mathbb{UV}\!=\!\mathbb{VU}$ if $\mathbb{S}$ is commutative. The operation between $\mathbb{U}$ and $\mathbb{V}$ is called complex multiplication. According to \cite[Lemma 1.3.1]{Z}, the complex multiplication is an associative operation on the set of all nonempty subsets of $\mathbb{S}$. For any nonempty subset $\mathbb{U}$ of $\mathbb{S}$, define $\mathbb{U}'=\{R_{a'}: R_a\in \mathbb{U}\}$. Set $R_g\mathbb{U}=\{R_g\}\mathbb{U}$, $\mathbb{U}R_g=\mathbb{U}\{R_g\}$, and $R_gR_h=\{R_g\}\{R_h\}$ for any $g, h\in [0,d]$. The following lemmas are necessary for us.
\begin{lem}\cite[Lemmas 1.1.3 (ii) and 1.1.4 (i)]{Z}\label{L;Lemma2.1}
Assume that $g, h, i\in [0,d]$. Then $k_gp_{hi}^g=k_hp_{gi'}^h=k_ip_{h'g}^i$. Moreover, $p_{hi}^g\neq 0$, $p_{gi'}^h\neq 0$, $p_{h'g}^i\neq 0$ are pairwise equivalent.
\end{lem}
\begin{lem}\cite[Lemma 1.3.4]{Z}\label{L;Lemma2.2}
Assume that $\mathbb{T}$, $\mathbb{U}$, $\mathbb{V}$, $\mathbb{W}$ are nonempty subsets of $\mathbb{S}$. Then $\mathbb{TU}\cap\mathbb{VW}\neq \varnothing$ if and only if $\mathbb{T}'\mathbb{V}\cap\mathbb{U}\mathbb{W}'\neq\varnothing$.
\end{lem}
\begin{lem}\cite[Lemma 1.5.2]{Z}\label{L;Lemma2.3}
Assume that $g, h\in [0,d]$. Then $|R_{g'}R_h|$ is less than or equal to the great common divisor of $k_g$ and $k_h$.%(see \cite[Lemma 2.5.9 (iii)]{Z}).
\end{lem}
The nonempty subset $\mathbb{U}$ of $\mathbb{S}$ is closed in $\mathbb{S}$ if $\mathbb{U}\mathbb{U}'\subseteq\mathbb{U}$. Write $\mathbb{U}\leq \mathbb{S}$ if $\mathbb{U}$ is closed in $\mathbb{S}$. If $\mathbb{U}\leq \mathbb{S}$, then $R_0\in \mathbb{U}$, $\mathbb{U}'=\mathbb{U}$, and $\mathbb{UU}\subseteq\mathbb{U}$. If $\mathbb{S}$ is commutative, $\mathbb{U}\leq \mathbb{S}$, and $\mathbb{V}\leq \mathbb{S}$, then $\mathbb{U}\mathbb{V}\leq \mathbb{S}$. If $\mathbb{U}\subseteq\mathbb{S}$, then the thin radical $\mathrm{O}_\vartheta(\mathbb{U})$ of $\mathbb{U}$ is defined to be $\{R_a: R_a\in \mathbb{U},\ k_a=1\}$. If $\mathbb{U}\leq \mathbb{S}$, notice that $\mathrm{O}_\vartheta(\mathbb{U})\leq \mathbb{S}$. Define $d_1=|\mathrm{O}_\vartheta(\mathbb{S})|$.

The intersection of the closed subsets in $\mathbb{S}$ is also closed in $\mathbb{S}$. For any nonempty subset $\mathbb{U}$ of $\mathbb{S}$, let $\langle \mathbb{U}\rangle$ be the intersection of all closed subsets in $\mathbb{S}$ containing $\mathbb{U}$. So $\langle \mathbb{U}\rangle\leq \mathbb{S}$. If $g\in [0,d]$, set $\langle R_g\rangle\!=\!\langle \{R_g\}\rangle$. If $\mathbb{U}\leq \mathbb{S}, \mathbb{V}\leq \mathbb{S}$, and $R_{g'}\mathbb{U}R_g\subseteq \mathbb{V}$ for any $R_g\in \mathbb{V}$, then $\mathbb{U}$ is called strongly normal in $\mathbb{V}$. Write $\mathbb{U}\unlhd\mathbb{V}$ if $\mathbb{U}$ is strongly normal in $\mathbb{V}$. So $\mathbb{S}\unlhd\mathbb{S}$. If $\mathbb{U}\leq\mathbb{S}$, then the intersection of the strongly normal closed subsets in $\mathbb{U}$ is also strongly normal in $\mathbb{U}$. For any $\mathbb{U}\leq \mathbb{S}$, the thin residue $\mathrm{O}^\vartheta(\mathbb{U})$ of $\mathbb{U}$ is the intersection of all strongly normal closed subsets in $\mathbb{U}$. We present two lemmas.
\begin{lem}\cite[Theorem 3.2.1 (ii)]{Z}\label{L;Lemma2.4}
If $\mathbb{U}\leq\mathbb{S}$, then $\mathrm{O}^\vartheta(\mathbb{U})=\langle\bigcup_{R_g\in\mathbb{U}}(R_{g'}R_g)\rangle$.
\end{lem}
\begin{lem}\label{L;Lemma2.5}
If $\mathbb{S}$ is commutative and $\mathbb{U}\leq \mathbb{S}$, then $\mathbb{U}\unlhd\mathbb{S}$ if and only if $\mathrm{O}^\vartheta(\mathbb{S})\subseteq\mathbb{U}$.
\end{lem}
\begin{proof}
The desired lemma follows from the definition of $\mathrm{O}^\vartheta(\mathbb{S})$ and Lemma \ref{L;Lemma2.4}.
\end{proof}

Fix $n\in \mathbb{N}_0\setminus\{0\}$. Let $g\!\in\! [0,2^n-1]$ and $\sum_{h=1}^{n}g_{(h)}2^{h-1}$ be the $2$-adic decomposition of $g$. Define $\nu(g)=(g_{(1)}, g_{(2)},\ldots, g_{(n)})$. Hence $\nu$ induces a bijective correspondence between the elements in $[0,2^n-1]$ and all $n$-dimensional $\{0,1\}$-vectors. Write $\mathbb{P}(g)$ for $\{a: g_{(a)}=1\}$. Let
$h\in[0,2^n-1]$. Write $g\leq_2 h$ if and only if $\mathbb{P}(g)\subseteq\mathbb{P}(h)$. For any $i\in [1, n]$, let $\mathbb{U}_i$ denote a fixed set and $|\mathbb{U}_i|=u_i\geq 2$. Set $n_2=|\{a: u_a>2\}|$ and $\mathbb{P}_2(g)=\{a: a\in \mathbb{P}(g),\ u_a>2\}$. Use $\prod_{i=1}^n\mathbb{U}_i$ to denote the cartesian product $\mathbb{U}_1\times \mathbb{U}_2\times\cdots\times \mathbb{U}_n$. For any $\mathbf{u}=(u(1), u(2),\ldots, u(n))\in\prod_{i=1}^n\mathbb{U}_i$, define $\mathbf{u}_i=u(i)$ for any $i\in [1,n]$. For any $\mathbf{u}, \mathbf{v}\in \prod_{i=1}^n\mathbb{U}_i$ and $j\in [0, 2^n-1]$, write $\mathbf{u}=_j\mathbf{v}$ to indicate that $\mathbf{u}_k\neq \mathbf{v}_k$ if and only if $k\in \mathbb{P}(j)$. By \cite[Page 344]{B}, $\mathbb{S}$ is called a factorial scheme with the parameters $u_1, u_2, \ldots, u_n$ if $\mathbb{X}=\prod_{i=1}^n\mathbb{U}_i$, $d=2^n-1$, $R_j=\{(\mathbf{a}, \mathbf{b}): \mathbf{a}=_j\mathbf{b}\}$ for any $j\in [0, d]$. If $\mathbb{S}$ is a factorial scheme with the parameters $u_1, u_2, \ldots, u_n$, notice that $\mathbb{S}$ is symmetric. Moreover, recall that $\mathbb{S}$ is triply regular (see \cite[Theorem 10]{W}).
\subsection{Algebras}
Let $\mathbb{Z}$ be the integer ring. Let $\F_p$ be the prime subfield of $\F$. Given $g\in \mathbb{Z}$, let $\overline{g}$ be the image of $g$ under the unital ring homomorphism from $\mathbb{Z}$ to $\F_p$.

Let $\mathbb{A}$ be an algebra with the identity element $1_{\mathbb{A}}$. The center $\mathrm{Z}(\mathbb{A})$ of $\mathbb{A}$ is defined to be $\{a: a\in \mathbb{A},\ ab=ba\ \forall\ b\in \mathbb{A}\}$. It is a subalgebra of $\mathbb{A}$ with the identity element $1_{\mathbb{A}}$. Let $\mathbb{I}$ be a two-sided ideal of $\mathbb{A}$. Write $\mathbb{A}/\mathbb{I}$ for the quotient algebra of $\mathbb{A}$ with respect to $\mathbb{I}$. Call $\mathbb{I}$ a nilpotent two-sided ideal of $\mathbb{A}$ if there is $h\!\in\!\mathbb{N}_0\setminus\{0\}$ such that the product of any $h$ elements in $\mathbb{I}$ is zero. If $\mathbb{I}$ is a nilpotent two-sided ideal of $\mathbb{A}$, the nilpotency $n(\mathbb{I})$ of $\mathbb{I}$ is the smallest choice of $h$. The Jacobson radical $\mathrm{Rad}(\mathbb{A})$ of $\mathbb{A}$ is the sum of all nilpotent two-sided ideals of $\mathbb{A}$. Recall that $\mathrm{Rad}(\mathbb{A})$ is also a nilpotent two-sided ideal of $\mathbb{A}$. Call $\mathbb{A}$ a semisimple algebra if $\mathbb{A}$ is a direct sum of its minimal two-sided ideals. Recall that $\mathbb{A}$ is semisimple if and only if $\mathrm{Rad}(\mathbb{A})$ is the zero space. If $e\in \mathbb{A}$ and $e^2=e$, let $e\mathbb{I}e=\{eae: a\in \mathbb{I}\}$. Notice that $e\mathbb{A}e$ is a subalgebra of $\mathbb{A}$ with the identity element $e$. An $\mathbb{A}$-module is called an irreducible $\mathbb{A}$-module if it does not have nonzero proper $\mathbb{A}$-submodule. We present two lemmas.
\begin{lem}\label{L;Lemma2.6}\cite[Proposition 3.2.4]{DK}
If $e\in \mathbb{A}$ and $e^2=e$, then $\mathrm{Rad}(e\mathbb{A}e)=e\mathrm{Rad}(\mathbb{A})e$. In particular, $\mathrm{Rad}(e\mathbb{A}e)\subseteq\mathrm{Rad}(\mathbb{A})$.
\end{lem}
\begin{lem}\label{L;Lemma2.7}\cite[Corollary 3.1.7]{DK}
If $g\in\mathbb{N}_0\setminus\{0\}$ and $\mathbb{A}/\mathrm{Rad}(\mathbb{A})$ is a direct sum of $g$ minimal two-sided ideals of $\mathbb{A}/\mathrm{Rad}(\mathbb{A})$, then the number of all isomorphic classes of irreducible $\mathbb{A}$-modules equals $g$.
\end{lem}
\subsection{Terwilliger $\F$-algebras of schemes}
For any $g\in\mathbb{N}_0\setminus\{0\}$, let $\mathrm{M}_g(\F)$ denote the full matrix algebra of $(g\times g)$-matrices whose entries are in $\F$. For any nonempty finite set $\mathbb{U}$, let $\mathrm{M}_{\mathbb{U}}(\F)$ be the full matrix algebra of square $\F$-matrices whose rows and columns are labeled by the elements in $\mathbb{U}$. Then   $\mathrm{M}_\mathbb{U}(\F)\cong \mathrm{M}_{|\mathbb{U}|}(\F)$ as algebras. Let $I, J, O$ denote the identity matrix, the all-one matrix, and the all-zero matrix in $\mathrm{M}_\mathbb{X}(\F)$, respectively. If $x, y\in \mathbb{X}$ and $M\in \mathrm{M}_\mathbb{X}(\F)$, the $(x,y)$-entry of $M$ is denoted by $M(x,y)$. The transpose of $M$ is denoted by $M^T$.

For any $g, h\in \mathbb{N}_0$, let $\delta_{gh}$ denote the Kronecker delta of $g$ and $h$ whose values are from $\F$. For any $x\in \mathbb{X}$ and $g\in [0,d]$, the adjacency $\F$-matrix $A_g$ with respect to $R_g$ is the $\{\overline{0},\overline{1}\}$-matrix in $\mathrm{M}_\mathbb{X}(\F)$, where $A_g(y,z)=\overline{1}$ if and only if $z\in yR_g$. The dual $\F$-idempotent $E_g^*(x)$ with respect to $x$ and $R_g$ is the diagonal $\{\overline{0},\overline{1}\}$-matrix in $\mathrm{M}_\mathbb{X}(\F)$, where $E_g^*(x)(y,y)\!=\!\overline{1}$ if and only if $y\in xR_g$. For any $g, h\in [0,d]$ and $x\in \mathbb{X}$,
\begin{align}\label{Eq;1}
A_g^T=A_{g'}\ \text{and}\ {E_g^*(x)}^T=E_g^*(x) \ (A_g^T=A_g\ \text{if $\mathbb{S}$ is symmetric}),
\end{align}
\begin{align}\label{Eq;2}
E_g^*(x)E_h^*(x)=\delta_{gh}E_g^*(x),
\end{align}
\begin{align}\label{Eq;3}
A_0=I=\sum_{i=0}^dE_i^*(x).
\end{align}

Pick $x\in \mathbb{X}$. The Terwilliger $\F$-algebra of $\mathbb{S}$ with respect to $x$, denoted by $\mathbb{T}(x)$, is the subalgebra of $\mathrm{M}_\mathbb{X}(\F)$ generated by $E_0^*(x), E_1^*(x),\ldots, E_d^*(x), A_0, A_1,\ldots, A_d$. By the definition of $\mathbb{T}(x)$ and \eqref{Eq;1}, notice that $M\in\mathbb{T}(x)$ if and only if $M^T\in\mathbb{T}(x)$. Moreover, notice that $M\in \mathrm{Rad}(\mathbb{T}(x))$ if and only if $M^T\in\mathrm{Rad}(\mathbb{T}(x))$. It is known that the algebraic structures of $\mathbb{T}(x)$ and $\mathrm{Rad}(\mathbb{T}(x))$ may depend on the choices of the fixed field $\F$ and $x$ (\cite[5.1]{Han}). For some progress on the algebraic structure of $\mathbb{T}(x)$, one may refer to \cite{Han}, \cite{J1}, and \cite{J2}. The following lemmas are necessary for us.
\begin{lem}\label{L;Lemma2.8}\cite[Theorem 3.4]{Han}
Assume that $x\in \mathbb{X}$. Then $\mathbb{T}(x)$ is semisimple only if $p\nmid k_g$ for any $g\in [0,d]$.
\end{lem}
\begin{lem}\label{L;Lemma2.9}\cite[Lemma 2.4 (i) and (ii)]{J2}
Assume that $x\in \mathbb{X}$ and $g, h, i\in [0,d]$. Then $p_{gh}^i\neq 0$ if and only if $E_g^*(x)A_hE_i^*(x)\neq O$. Moreover, $\mathbb{T}(x)$ has an $\F$-linearly independent subset $\{E_a^*(x)A_bE_c^*(x): p_{ab}^c\neq0\}$ with cardinality $|\{(a, b, c): p_{ab}^c\neq0\}|$.
\end{lem}
\begin{lem}\label{L;Lemma2.10}
Assume that $x\in \mathbb{X}$ and $\mathbb{S}$ is triply regular. Then $\mathbb{T}(x)$ has an $\F$-basis $\{E_a^*(x)A_bE_c^*(x): p_{ab}^c\neq0\}$ with cardinality $|\{(a, b, c): p_{ab}^c\neq0\}|$. Moreover, if $\mathbb{S}$ is symmetric and $g\in [0,d]$, the subalgebra $E_g^*(x)\mathbb{T}(x)E_g^*(x)$ of $\mathbb{T}(x)$ is commutative.
\end{lem}
\begin{proof}
The desired lemma is from combining \cite[Lemma 4]{MA}, Lemma \ref{L;Lemma2.9}, \eqref{Eq;1}.
\end{proof}
We close this section by simplifying the notation. Recall that $n$ is a fixed number in $\mathbb{N}_0\setminus\{0\}$ and the fixed set $\mathbb{U}_g$ has cardinality $u_g$ for any $g\in [1,n]$. From now on, assume that $\mathbb{X}=\prod_{g=1}^n\mathbb{U}_g$, $d=2^n-1$, $R_h=\{(\mathbf{a}, \mathbf{b}): \mathbf{a}=_h\mathbf{b}\}$ for any $h\in [0,d]$. Therefore $\mathbb{S}$ is a factorial scheme with the parameters $u_1, u_2,\ldots, u_n$. From now on, we shall quote the fact that $\mathbb{S}$ is both symmetric and commutative without reference. Furthermore, we shall also quote the fact that $\mathbb{S}$ is triply regular without reference. Fix $\mathbf{x}\in \mathbb{X}$. For convenience, we write $\mathbb{T}=\mathbb{T}(\mathbf{x})$ and $E_h^*=E_h^*(\mathbf{x})$ for any $h\in [0,d]$.
\section{Closed subsets of factorial schemes}
In this section, we determine all closed subsets in $\mathbb{S}$. Moreover, we also determine all strongly normal closed subsets in $\mathbb{S}$. We start our discussion with two lemmas.
\begin{lem}\label{L;Lemma3.1}
Assume that $g, h\in [0,d]$. Then $g=h$ if and only if $\mathbb{P}(g)=\mathbb{P}(h)$. In particular, $\leq_2$ is a partial order on the set $[0, d]$.
\end{lem}
\begin{proof}
As $\mathbb{P}(g)$ uniquely determines $\nu(g)$ and $\nu(g)$ uniquely determines $g$, the first statement is proved. The desired lemma thus follows from the definition of $\leq_2$.
\end{proof}
\begin{lem}\label{L;Lemma3.2}
Assume that $g\in [0,d]$. Then there is a unique $R_h\in R_gR_g$ such that $\mathbb{P}(h)=\mathbb{P}_2(h)=\mathbb{P}_2(g)$.
\end{lem}
\begin{proof}
Pick $\mathbf{y}\in\mathbf{x}R_g$. Hence $\mathbf{y}_i\neq\mathbf{x}_i$ and $\mathbf{y}_j\!=\!\mathbf{x}_j$ for any $i\in \mathbb{P}(g)$ and $j\in [1,n]\setminus\mathbb{P}(g)$. Then there is $\mathbf{z}\in \mathbf{x}R_g$ such that $\mathbf{z}_i\notin\{\mathbf{x}_i, \mathbf{y}_i\}$, $\mathbf{z}_j=\mathbf{y}_j\neq\mathbf{x}_j$, and $\mathbf{z}_k=\mathbf{y}_k=\mathbf{x}_k$ for any $i\in \mathbb{P}_2(g)$, $j\in \mathbb{P}(g)\setminus\mathbb{P}_2(g)$, and $k\in [1,n]\setminus\mathbb{P}(g)$. So there is $R_h\in R_gR_g$ such that $\mathbf{z}\in\mathbf{y}R_h$ and $\mathbb{P}(h)=\mathbb{P}_2(h)=\mathbb{P}_2(g)$. As $\mathbb{P}(h)=\mathbb{P}_2(h)=\mathbb{P}_2(g)$, notice that the uniqueness of $h$ can be proved by Lemma \ref{L;Lemma3.1}. The desired lemma thus follows.
\end{proof}
Lemma \ref{L;Lemma3.2} motivates us to introduce the following notation and four lemmas.
\begin{nota}\label{N;Notation3.3}
\em Assume that $g\in[0,d]$. Then $h$ in Lemma \ref{L;Lemma3.2} is denoted by $\widetilde{g}$. Hence $R_{\widetilde{g}}$ is the unique element in $R_gR_g$ that has the property $\mathbb{P}(\widetilde{g})=\mathbb{P}_2(\widetilde{g})=\mathbb{P}_2(g)$. For example, if $n=u_1\!=\!2$ and $u_2\!=\!3$, then $d=3$, $\mathbb{P}(2)=\mathbb{P}_2(2)=\mathbb{P}_2(3)=\{2\}$, $\widetilde{3}=2$.
\end{nota}
\begin{lem}\label{L;Lemma3.4}
Assume that $g, h\in [0,d]$ and $h\leq_2\widetilde{g}$. Then $R_h\in \langle R_{\widetilde{g}}\rangle$.
\end{lem}
\begin{proof}
Pick $\mathbf{y}\in\mathbf{x}R_{\widetilde{g}}$. Hence $\mathbf{y}_i\neq\mathbf{x}_i$ and $\mathbf{y}_j\!=\!\mathbf{x}_j$ for any $i\in \mathbb{P}(\widetilde{g})$ and $j\in [1,n]\setminus\mathbb{P}(\widetilde{g})$. According to Notation \ref{N;Notation3.3} and the assumption $h\leq_2 \widetilde{g}$, there is $\mathbf{z}\in\mathbf{x}R_{\widetilde{g}}$ such that $\mathbf{z}_i\notin\{\mathbf{x}_i, \mathbf{y}_i\}$, $\mathbf{z}_j=\mathbf{y}_j\neq \mathbf{x}_j$, and $\mathbf{z}_k=\mathbf{y}_k=\mathbf{x}_k$ for any $i\in \mathbb{P}(h)$, $j\in \mathbb{P}(\widetilde{g})\setminus\mathbb{P}(h)$, and $k\in[1,n]\setminus\mathbb{P}(\widetilde{g})$. So $\mathbf{z}\in\mathbf{y}R_{h}$ and $R_{h}\in R_{\widetilde{g}}R_{\widetilde{g}}$.
The desired lemma thus follows.
\end{proof}
\begin{lem}\label{L;Lemma3.5}
Assume that $g\!\in\! [0,d]$. Then $\{R_a: a\!\leq_2\! g\}\!\leq\!\mathbb{S}$ and $\langle R_{\widetilde{g}}\rangle\!\!=\!\!\{R_a: a\leq_2 \widetilde{g}\}$.
\end{lem}
\begin{proof}
Pick $R_h, R_i\in \{R_a: a\leq_2 g\}$, $\mathbf{w}\in\mathbb{X}$, and $\mathbf{y}\in \mathbf{w}R_h$. Therefore $\mathbf{y}_j\neq\mathbf{w}_j$ and $\mathbf{y}_k=\mathbf{w}_k$ for any $j\in\mathbb{P}(h)$ and $k\in [1,n]\setminus \mathbb{P}(h)$. Pick $\mathbf{z}\in\mathbf{w}R_i$. Hence $\mathbf{z}_j\!\neq\! \mathbf{w}_j$ and $\mathbf{z}_k\!=\!\mathbf{w}_k$ for any $j\!\in\!\mathbb{P}(i)$ and $k\!\in\! [1,n]\setminus \mathbb{P}(i)$. Hence there is $\ell\in[0,d]$ such that $\mathbf{z}\in\mathbf{y}R_\ell$ and $\mathbb{P}(\ell)\subseteq \mathbb{P}(h)\cup\mathbb{P}(i)$. As $h\leq_2 g$ and $i\leq_2 g$, $\mathbb{P}(\ell)\subseteq\mathbb{P}(h)\cup\mathbb{P}(i)\subseteq \mathbb{P}(g)$ and $\ell\leq_2 g$. As
$\mathbf{w}$ is chosen from $\mathbb{X}$ arbitrarily, notice that $R_hR_i\subseteq\{R_a: a\leq_2 g\}$. Since $R_h$ and $R_i$ are chosen from $\{R_a: a\leq_2 g\}$ arbitrarily, it is obvious that $\{R_a: a\leq_2 g\}\!\leq\!\mathbb{S}$. The first statement is proved. By Lemma \ref{L;Lemma3.4} and the first statement, the second statement is also proved. The desired lemma thus follows.
\end{proof}
\begin{lem}\label{L;Lemma3.6}
Assume that $g, h\in[0,d]$. If $\langle R_{\widetilde{g}}\rangle\!=\!\langle R_h\rangle$, then $R_{\widetilde{g}}\!=\!R_h$. In particular, $\langle R_{\widetilde{g}}\rangle\!=\!\langle R_{\widetilde{h}}\rangle$ if and only if $\mathbb{P}(\widetilde{g})\!=\!\mathbb{P}(\widetilde{h})$.
\end{lem}
\begin{proof}
As $\langle R_{\widetilde{g}}\rangle=\langle R_h\rangle$, it is clear that $h\!=\!\widetilde{h}$ and $h\leq_2 \widetilde{g}\leq_2 h$ by Lemma \ref{L;Lemma3.5}. The first statement thus follows from Lemma \ref{L;Lemma3.1}. The second statement follows from the first one and Lemma \ref{L;Lemma3.1}. The desired lemma thus follows.
\end{proof}
\begin{lem}\label{L;Lemma3.7}
Assume that $\mathbb{U}\leq\mathbb{S}$, $R_g\!\in\!\mathbb{U}$, and $|\mathbb{P}(\widetilde{h})|\!\leq\!|\mathbb{P}(\widetilde{g})|$ for any $R_h\!\in\!\mathbb{U}$. Then $R_{\widetilde{h}}\in\langle R_{\widetilde{g}}\rangle$ for any $R_h\in \mathbb{U}$. In particular, $\langle R_{\widetilde{g}}\rangle$ is independent of the choice of $R_g$.
\end{lem}
\begin{proof}
Assume that there exists $R_i\in \mathbb{U}$ such that $\mathbb{P}(\widetilde{i})\nsubseteq\mathbb{P}(\widetilde{g})$. Then $\mathbb{P}(\widetilde{i})\setminus \mathbb{P}(\widetilde{g})\neq \varnothing$. Pick $j\in \mathbb{P}(\widetilde{i})\setminus \mathbb{P}(\widetilde{g})$ and $\mathbf{v}\in \mathbf{x}R_{\widetilde{i}}$.
Hence $\mathbf{v}_k\neq \mathbf{x}_k$ and $\mathbf{v}_\ell=\mathbf{x}_\ell$ for any $k\in \mathbb{P}(\widetilde{i})$ and $\ell\in [1,n]\setminus \mathbb{P}(\widetilde{i})$.
Then there is $\mathbf{w}\in\mathbf{x}R_{\widetilde{i}}$ such that $\mathbf{w}_j\notin\{\mathbf{x}_j, \mathbf{v}_j\}$, $\mathbf{w}_k=\mathbf{v}_k\neq\mathbf{x}_k$, and
$\mathbf{w}_\ell=\mathbf{v}_\ell=\mathbf{x}_\ell$ for any $k\in\mathbb{P}(\widetilde{i})\setminus\{j\}$ and $\ell\in [1,n]\setminus\mathbb{P}(\widetilde{i})$. So there is $R_m\in R_{\widetilde{i}}R_{\widetilde{i}}$ such that $\mathbf{w}\!\in\!\mathbf{v}R_m$ and $\mathbb{P}(m)=\mathbb{P}_2(m)\!=\!\{j\}$. As $R_i\in \mathbb{U}$, $R_{\widetilde{i}}, R_m\in \mathbb{U}$ by Notation \ref{N;Notation3.3}.

Pick $\mathbf{y}\in\mathbf{x}R_{\widetilde{g}}$. So $\mathbf{y}_k\neq\mathbf{x}_k$ and $\mathbf{y}_\ell\!=\!\mathbf{x}_\ell$ for any $k\in\mathbb{P}(\widetilde{g})$ and $\ell\in[1,n]\setminus\mathbb{P}(\widetilde{g})$. As $j\in \mathbb{P}(m)\setminus \mathbb{P}(\widetilde{g})$, there is $\mathbf{z}\in\mathbf{x}R_m$ such that $\mathbf{z}_j\neq \mathbf{x}_j=\mathbf{y}_j$, $\mathbf{z}_k=\mathbf{x}_k\neq\mathbf{y}_k$, and $\mathbf{z}_\ell=\mathbf{x}_\ell=\mathbf{y}_\ell$ for any $k\in\mathbb{P}(\widetilde{g})$ and $\ell\in [1,n]\setminus(\mathbb{P}(\widetilde{g})\cup\{j\})$. Hence there is $R_q\in R_{\widetilde{g}}R_m$ such that $\mathbf{z}\in\mathbf{y}R_q$ and $\mathbb{P}(q)=\mathbb{P}_2(q)=\mathbb{P}_2(\widetilde{g})\cup \{j\}=\mathbb{P}(\widetilde{g})\cup \{j\}$. Hence $R_{\widetilde{q}}=R_q$. As $R_{\widetilde{g}}, R_m\in\mathbb{U}$,
$R_{\widetilde{q}}\in\mathbb{U}$ by Notation \ref{N;Notation3.3}. Hence $|\mathbb{P}(\widetilde{g})|+1=|\mathbb{P}(\widetilde{q})|\leq|\mathbb{P}(\widetilde{g})|$. This is absurd. So $\widetilde{h}\leq_2\widetilde{g}$ for any $R_h\in \mathbb{U}$. The first statement is thus from Lemma \ref{L;Lemma3.4}. The second statement is thus from the first one.
The desired lemma thus follows.
\end{proof}
Notation \ref{N;Notation3.3} and Lemmas \ref{L;Lemma3.7}, \ref{L;Lemma3.6} motivate us to introduce the following notation.
\begin{nota}\label{N;Notation3.8}
\em Assume that $\mathbb{U}\leq\mathbb{S}$, $R_g\in \mathbb{U}$, and $|\mathbb{P}(\widetilde{h})|\!\leq\! |\mathbb{P}(\widetilde{g})|$ for any $R_h\in \mathbb{U}$. As $R_{\widetilde{g}}\in R_gR_g$ by Notation \ref{N;Notation3.3}, notice that $\langle R_{\widetilde{g}}\rangle\subseteq\mathbb{U}$. As $\langle R_{\widetilde{g}}\rangle$ is uniquely determined by Lemma \ref{L;Lemma3.7}, set $\mathbb{U}_{\max}=\langle R_{\widetilde{g}}\rangle\subseteq\mathbb{U}$. Lemma \ref{L;Lemma3.7} thus implies that $R_{\widetilde{h}}\in \mathbb{U}_{\max}$ for any $R_h\in\mathbb{U}$. As $\mathbb{P}(\widetilde{g})$ is uniquely determined by Lemma \ref{L;Lemma3.6}, put $\mathbb{P}(\mathbb{U})=\mathbb{P}(\widetilde{g})$.
\end{nota}
For further discussion, the following notation and six lemmas are necessary.
\begin{nota}\label{N;Notation3.9}
\em Assume that $\mathbb{U}$ and $\mathbb{V}$ are nonempty subsets of $\mathbb{S}$. Use $\mathbb{U}^\gamma$ to denote $\{\{R_a\}: R_a\in\mathbb{U}\}$. Then $\mathbb{U}=\mathbb{V}$ if and only if $\mathbb{U}^\gamma=\mathbb{V}^\gamma$. If $\mathbb{U}^\gamma$ is a finite group with respect to the complex multiplication and the identity element $\{R_0\}$, then $\mathbb{U}\leq \mathbb{S}$.
\end{nota}
\begin{lem}\label{L;Lemma3.10}
Assume that $\mathbb{U}\leq \mathbb{S}$. Then $\mathrm{O}_\vartheta(\mathbb{U})^\gamma$ is an elementary abelian $2$-group with respect to the complex multiplication and the identity element $\{R_0\}$. Moreover, $\mathrm{O}_\vartheta(\mathbb{U})^\gamma$ is an elementary abelian $2$-subgroup of $\mathrm{O}_\vartheta(\mathbb{S})^\gamma$ with respect to the complex multiplication and the identity element $\{R_0\}$.
\end{lem}
\begin{proof}
By \cite[Preface]{Z}, $\mathrm{O}_\vartheta(\mathbb{U})^\gamma$ is always a finite group with respect to the complex multiplication and the identity element $\{R_0\}$. As $k_g\!=\!1$ for any $R_g\in\mathrm{O}_\vartheta(\mathbb{U})$, Lemma \ref{L;Lemma2.3} thus implies that $R_gR_g\!=\!R_gR_{g'}=\{R_0\}$ for any $R_g\in\mathrm{O}_\vartheta(\mathbb{U})$. The first statement thus follows. As $\mathbb{S}\leq \mathbb{S}$ and $\mathrm{O}_\vartheta(\mathbb{U})^\gamma\subseteq\mathrm{O}_\vartheta(\mathbb{S})^\gamma$, the second statement thus follows from the first one. The desired lemma thus follows from the above discussion.
\end{proof}
\begin{lem}\label{L;Lemma3.11}
Assume that $g\in [0,d]$. Then $k_g=\prod_{h\in\mathbb{P}(g)}(u_h-1)$, where the product over an empty set is equal to one. Moreover, $k_g=1$ if and only if $\widetilde{g}=0$.
\end{lem}
\begin{proof}
The first statement follows as $\mathbf{x}R_g=\{\mathbf{a}: \mathbf{x}=_g\mathbf{a}\}$ and $k_g=|\mathbf{x}R_g|$. According to the first statement, notice that $k_g=1$ if and only if $u_h=2$ for any $h\in\mathbb{P}(g)$. As $\mathbb{P}_2(g)\subseteq\mathbb{P}(g)$, notice that $\mathbb{P}_2(g)\!=\!\varnothing$ if and only if $u_h=2$ for any $h\in\mathbb{P}(g)$. Notation \ref{N;Notation3.3} and Lemma \ref{L;Lemma3.1} imply that $\widetilde{g}=0$ if and only if $\mathbb{P}_2(g)=\varnothing$. The second statement thus follows from the above discussion. The desired lemma thus follows.
\end{proof}
\begin{lem}\label{L;Lemma3.12}
Assume that $g\in [0,d]$. Then there is a unique $R_h\in R_{\widetilde{g}}R_g$ such that $\mathbb{P}(h)=\mathbb{P}(g)\setminus \mathbb{P}_2(g)$ and $k_h=1$.
\end{lem}
\begin{proof}
Pick $\mathbf{y}\in\mathbf{x}R_{\widetilde{g}}$. Hence $\mathbf{y}_i\!\neq\!\mathbf{x}_i$ and $\mathbf{y}_j\!=\!\mathbf{x}_j$ for any $i\in \mathbb{P}(\widetilde{g})$ and $j\in [1,n]\setminus \mathbb{P}(\widetilde{g})$. As $\mathbb{P}(\widetilde{g})=\mathbb{P}_2(g)$ by Notation \ref{N;Notation3.3}, there must be $\mathbf{z}\in\mathbf{x}R_{g}$ such that $\mathbf{z}_i=\mathbf{y}_i\neq\mathbf{x}_i$, $\mathbf{z}_j\neq \mathbf{y}_j=\mathbf{x}_j$, and $\mathbf{z}_k=\mathbf{x}_k$ for any $i\in \mathbb{P}_2(g)$, $j\in \mathbb{P}(g)\setminus \mathbb{P}_2(g)$, and $k\in [1,n]\setminus\mathbb{P}(g)$. Hence there is $R_h\in  R_{\widetilde{g}}R_g$ such that $\mathbf{z}\in\mathbf{y}R_h$ and $\mathbb{P}(h)=\mathbb{P}(g)\setminus \mathbb{P}_2(g)$. Notation \ref{N;Notation3.3} and Lemma \ref{L;Lemma3.11} imply that $k_h=1$. As $\mathbb{P}(h)=\mathbb{P}(g)\setminus \mathbb{P}_2(g)$, the uniqueness of $h$ is thus from Lemma \ref{L;Lemma3.1}. The desired lemma thus follows.
\end{proof}
\begin{lem}\label{L;Lemma3.13}
Assume that $\mathbb{U}\leq \mathbb{S}$. Then $\mathbb{U}\!=\!\mathbb{U}_{\max}\mathrm{O}_\vartheta(\mathbb{U})$ and $\mathbb{U}_{\max}\cap\mathrm{O}_\vartheta(\mathbb{U})=\{R_0\}$.
If $R_g, R_h\in \mathbb{U}_{\max}$, $R_i, R_j\in \mathrm{O}_\vartheta(\mathbb{U})$, $R_gR_i\cap R_hR_j\neq\varnothing$, then $R_g=R_h$ and $R_i=R_j$.
\end{lem}
\begin{proof}
Assume that $R_k\in \mathbb{U}$. As $\mathbb{U}\leq \mathbb{S}$, Lemma \ref{L;Lemma3.12} and Notation \ref{N;Notation3.3} imply that $R_\ell\in R_{\widetilde{k}}R_k$ for $R_\ell\in\mathrm{O}_\vartheta(\mathbb{U})$. Lemma \ref{L;Lemma2.1} thus implies that $R_k\in R_{\widetilde{k}}R_\ell$. Notice that $\mathbb{U}_{\max}\mathrm{O}_\vartheta(\mathbb{U})\!\subseteq\! \mathbb{U}\!\subseteq\!\mathbb{U}_{\max}\mathrm{O}_\vartheta(\mathbb{U})$ as $R_k$ is chosen from $\mathbb{U}$ arbitrarily.
So $\mathbb{U}\!=\!\mathbb{U}_{\max}\mathrm{O}_\vartheta(\mathbb{U})$. Pick $R_m\in\mathbb{U}_{\max}\cap\mathrm{O}_\vartheta(\mathbb{U})$. The combination of Notation \ref{N;Notation3.8}, Lemmas \ref{L;Lemma3.5}, \ref{L;Lemma3.11} shows that $m=\widetilde{m}=0$. As $R_m$ is chosen from $\mathbb{U}_{\max}\cap\mathrm{O}_\vartheta(\mathbb{U})$ arbitrarily, it is obvious that $\mathbb{U}_{\max}\cap\mathrm{O}_\vartheta(\mathbb{U})=\{R_0\}$. As $R_gR_i\cap R_hR_j\neq\varnothing$, $R_gR_h\cap R_iR_j\neq\varnothing$ by Lemma \ref{L;Lemma2.2}. As $R_g, R_h\in \mathbb{U}_{\max}$, $R_i, R_j\in \mathrm{O}_\vartheta(\mathbb{U})$, $\mathbb{U}_{\max}\cap\mathrm{O}_\vartheta(\mathbb{U})=\{R_0\}$, Lemma \ref{L;Lemma2.1} thus implies that $R_g=R_h$ and $R_i=R_j$. The desired lemma thus follows.
\end{proof}
\begin{lem}\label{L;Lemma3.14}
Assume that $\mathbb{U}\!\leq\!\mathbb{S}$. Then $\mathbb{U}_{\max}\!=\!\mathrm{O}^\vartheta(\mathbb{U})$. Moreover, $\mathbb{U}\!=\!\mathrm{O}^\vartheta(\mathbb{U})\mathrm{O}_\vartheta(\mathbb{U})$ and $\mathrm{O}^\vartheta(\mathbb{U})\!\cap\! \mathrm{O}_\vartheta(\mathbb{U})\!=\!\{R_0\}$. In particular, $\mathbb{S}\!=\!\mathrm{O}^\vartheta(\mathbb{S})\mathrm{O}_\vartheta(\mathbb{S})$ and $\mathrm{O}^\vartheta(\mathbb{S})\cap \mathrm{O}_\vartheta(\mathbb{S})=\{R_0\}$.
\end{lem}
\begin{proof}
Pick $R_g\in\mathbb{U}$. As $\mathbb{U}\leq \mathbb{S}$, Lemma \ref{L;Lemma3.13} shows that $R_g\in R_hR_i$ for $R_h\in\mathbb{U}_{\max}$ and $R_i\in \mathrm{O}_\vartheta(\mathbb{U})$.
So $R_gR_g\subseteq R_hR_iR_hR_i=R_hR_h\subseteq\mathbb{U}_{\max}$ by Lemma \ref{L;Lemma3.10}. As $R_g$ is chosen from $\mathbb{U}$ arbitrarily, Lemma \ref{L;Lemma2.4} thus implies that $\mathrm{O}^\vartheta(\mathbb{U})\subseteq\mathbb{U}_{\max}$. Furthermore, the combination of Notations \ref{N;Notation3.8}, \ref{N;Notation3.3}, and Lemma \ref{L;Lemma2.4} shows that $\mathbb{U}_{\max}\subseteq\mathrm{O}^\vartheta(\mathbb{U})$. Hence $\mathbb{U}_{\max}=\mathrm{O}^\vartheta(\mathbb{U})$. The first statement thus follows. The second statement thus follows from Lemma \ref{L;Lemma3.13} and the first one. Notice that $\mathbb{S}\leq\mathbb{S}$. The third statement thus follows from the second one. The desired lemma thus follows.
\end{proof}
\begin{lem}\label{L;Lemma3.15}
Assume that $\mathbb{U}\unlhd\mathbb{S}$. Then $\mathrm{O}^\vartheta(\mathbb{U})\!\!=\!\!\mathrm{O}^\vartheta(\mathbb{S})$. Moreover, $\mathbb{U}\!=\!\mathrm{O}^\vartheta(\mathbb{S})\mathrm{O}_\vartheta(\mathbb{U})$ and $\mathrm{O}^\vartheta(\mathbb{S})\!\cap\! \mathrm{O}_\vartheta(\mathbb{U})\!=\!\{R_0\}$.
\end{lem}
\begin{proof}
Notice that $\mathrm{O}^\vartheta(\mathbb{U})\subseteq\mathrm{O}^\vartheta(\mathbb{S})$ by Lemma \ref{L;Lemma2.4}. As $\mathbb{U}\!\unlhd\!\mathbb{S}$, notice that $\mathrm{O}^\vartheta(\mathbb{S})\subseteq \mathbb{U}$ by Lemma \ref{L;Lemma2.5}. As $\mathbb{S}\leq\mathbb{S}$, Lemma \ref{L;Lemma3.14} shows that $\mathbb{S}_{\max}=\mathrm{O}^\vartheta(\mathbb{S})\subseteq \mathbb{U}$. By combining Notations \ref{N;Notation3.8}, \ref{N;Notation3.3}, and Lemma \ref{L;Lemma3.1}, notice that
$\mathbb{S}_{\max}=\langle R_g\rangle$ for $R_g\in \mathbb{U}$ and $\widetilde{g}=g$. Notation \ref{N;Notation3.8} thus implies that $\mathbb{S}_{\max}\!\!=\!\!\langle R_g\rangle\subseteq \mathbb{U}_{\max}$. Hence $\mathrm{O}^\vartheta(\mathbb{S})\subseteq\mathrm{O}^\vartheta(\mathbb{U})$ by Lemma \ref{L;Lemma3.14}. The first statement thus follows. The second statement thus follows from the first one and Lemma \ref{L;Lemma3.14}. The desired lemma thus follows.
\end{proof}
We now can determine all closed subsets and strongly normal closed subsets in $\mathbb{S}$.
\begin{thm}\label{T;Theorem3.16}
Assume that $\mathbb{U}\subseteq \mathbb{S}$. Then $\mathbb{U}\leq \mathbb{S}$ if and only if there are $g\in [0,d]$ and $\mathbb{V}\leq \mathbb{S}$ such that $\mathbb{P}(g)=\mathbb{P}_2(g)$, $\mathbb{V}\subseteq\mathrm{O}_\vartheta(\mathbb{S})$, $\mathbb{U}=\langle R_g\rangle\mathbb{V}$, and $\langle R_g\rangle\cap \mathbb{V}=\{R_0\}$.
\end{thm}
\begin{proof}
If $g\in [0,d]$, $\mathbb{V}\leq \mathbb{S}$, and $\mathbb{U}=\langle R_g\rangle\mathbb{V}$, notice that $\mathbb{U}=\langle R_g\rangle\mathbb{V}\leq \mathbb{S}$. Hence the desired theorem follows from combining Lemma \ref{L;Lemma3.13}, Notations \ref{N;Notation3.8}, and \ref{N;Notation3.3}.
\end{proof}
\begin{thm}\label{T;Theorem3.17}
Assume that $\mathbb{U}\subseteq\mathbb{S}$. Then $\mathbb{U}\unlhd \mathbb{S}$ if and only if there are $\mathbb{V}\leq \mathbb{S}$ such that $\mathbb{V}\subseteq\mathrm{O}_\vartheta(\mathbb{S})$, $\mathbb{U}=\mathrm{O}^\vartheta(\mathbb{S})\mathbb{V}$, and $\mathrm{O}^\vartheta(\mathbb{S})\cap\mathbb{V}=\{R_0\}$.
\end{thm}
\begin{proof}
The desired theorem follows from combining Lemmas \ref{L;Lemma2.4}, \ref{L;Lemma2.5}, and \ref{L;Lemma3.15}.
\end{proof}
For an additional main result of this section, the following lemmas are required.
\begin{lem}\label{L;Lemma3.18}
Assume that $\mathbb{U}\!\leq\! \mathbb{S}$ and $\mathbb{V}\!\leq\! \mathbb{S}$. Then $\mathbb{U}\!=\!\mathbb{V}$ if and only if $\mathbb{P}(\mathbb{U})=\mathbb{P}(\mathbb{V})$ and $\mathrm{O}_\vartheta(\mathbb{U})=\mathrm{O}_\vartheta(\mathbb{V})$. If $\mathbb{U}\unlhd \mathbb{S}$ and $\mathbb{V}\unlhd\mathbb{S}$, then $\mathbb{U}=\mathbb{V}$ if and only if $\mathrm{O}_\vartheta(\mathbb{U})=\mathrm{O}_\vartheta(\mathbb{V})$.
\end{lem}
\begin{proof}
By Notation \ref{N;Notation3.8} and Lemma \ref{L;Lemma3.13}, notice that $\mathbb{U}\!=\!\mathbb{V}$ if and only if $\mathbb{U}_{\max}\!=\!\mathbb{V}_{\max}$ and $\mathrm{O}_\vartheta(\mathbb{U})=\mathrm{O}_\vartheta(\mathbb{V})$. By Notation \ref{N;Notation3.8} and Lemma \ref{L;Lemma3.6}, notice that $\mathbb{U}_{\max}=\mathbb{V}_{\max}$ if and only if $\mathbb{P}(\mathbb{U})=\mathbb{P}(\mathbb{V})$. The first statement thus follows. The second statement follows from Lemma \ref{L;Lemma3.15}. The desired lemma thus follows.
\end{proof}
\begin{lem}\label{L;Lemma3.19}
Assume that $\mathbb{U}\leq\mathbb{S}$. Then $|\mathbb{U}|=|\mathbb{U}_{\max}||\mathrm{O}_\vartheta(\mathbb{U})|=|\mathrm{O}^\vartheta(\mathbb{U})||\mathrm{O}_\vartheta(\mathbb{U})|$.
\end{lem}
\begin{proof}
By Lemmas \ref{L;Lemma2.3} and \ref{L;Lemma3.13}, notice that $|\mathbb{U}_{\max}R_g|\!=\!|\mathbb{U}_{\max}|$ for any $R_g\in \mathrm{O}_\vartheta(\mathbb{U})$. By Lemma \ref{L;Lemma3.13} again, notice that $\{\mathbb{U}_{\max}R_a\!: R_a\!\in\! \mathrm{O}_\vartheta(\mathbb{U})\}$ forms a partition of $\mathbb{U}$. Hence $|\mathbb{U}|=|\mathbb{U}_{\max}||\mathrm{O}_\vartheta(\mathbb{U})|$. The desired lemma thus follows from Lemma \ref{L;Lemma3.14}.
\end{proof}
\begin{lem}\label{L;Lemma3.20}
Assume that $\mathbb{P}\subseteq\{a: u_a>2\}$, $\mathbb{O}\leq\mathbb{S}$, and $\mathbb{O}\subseteq\mathrm{O}_\vartheta(\mathbb{S})$. Then there exist $\mathbb{U}\leq \mathbb{S}$ and $\mathbb{V}\unlhd \mathbb{S}$ such that $\mathbb{P}(\mathbb{U})=\mathbb{P}$ and $\mathrm{O}_\vartheta(\mathbb{U})=\mathrm{O}_\vartheta(\mathbb{V})=\mathbb{O}$.
\end{lem}
\begin{proof}
According to Notation \ref{N;Notation3.3} and the definition of $\mathbb{P}$, there is $g\in [0,d]$ such that $\mathbb{P}(\widetilde{g})\!=\!\mathbb{P}_2(\widetilde{g})\!=\!\mathbb{P}_2(g)=\mathbb{P}(g)=\mathbb{P}$. Hence $g=\widetilde{g}$ by Lemma \ref{L;Lemma3.1}.
Set $\mathbb{U}=\langle R_g\rangle\mathbb{O}\leq \mathbb{S}$ and $\mathbb{V}=\mathrm{O}^\vartheta(\mathbb{S})\mathbb{O}\unlhd\mathbb{S}$ by Lemma \ref{L;Lemma2.5}. As $R_g\!\in\!\mathbb{U}$ and Notation \ref{N;Notation3.8} holds, notice that $\langle R_g\rangle\!=\!\langle R_{\widetilde{g}}\rangle\!\subseteq\!\mathbb{U}_{\max}$ and $\mathbb{O}\!\!\subseteq\!\!\mathrm{O}_\vartheta(\mathbb{U})\!\!\cap\!\mathrm{O}_\vartheta(\mathbb{V})$. By Lemmas \ref{L;Lemma2.3} and \ref{L;Lemma3.13}, observe that $|\langle R_g\rangle R_h|=|\langle R_g\rangle|$ and $|\mathrm{O}^\vartheta(\mathbb{S})R_h|=|\mathrm{O}^\vartheta(\mathbb{S})|$ for any $R_h\in \mathbb{O}$. These equalities yield $|\mathbb{U}|\leq |\langle R_g\rangle||\mathbb{O}|\leq|\mathbb{U}_{\max}||\mathrm{O}_\vartheta(\mathbb{U})|=|\mathbb{U}|$ and $|\mathbb{V}|\leq |\mathrm{O}^\vartheta(\mathbb{S})||\mathbb{O}|\!\leq\! |\mathrm{O}^\vartheta(\mathbb{S})||\mathrm{O}_\vartheta(\mathbb{V})|\!=\!|\mathbb{V}|$ by Lemmas \ref{L;Lemma3.19} and \ref{L;Lemma3.15}. Therefore $\mathbb{U}_{\max}\!\!=\!\!\langle R_g\rangle$ and $\mathrm{O}_\vartheta(\mathbb{U})\!=\!\mathrm{O}_\vartheta(\mathbb{V})=\mathbb{O}$. As $g=\widetilde{g}$ and Notation \ref{N;Notation3.8} holds, $\mathbb{P}(\mathbb{U})\!=\!\mathbb{P}(\widetilde{g})\!=\!\mathbb{P}(g)\!=\!\mathbb{P}$. The desired lemma thus follows.
\end{proof}
We end this section by the remaining main result of this section and an example.
\begin{nota}\label{N;Notation3.21}
\em Assume that $g\!\in\!\mathbb{N}_0$ and $\mathbb{K}$ is a field of $h$ elements. Then the Galois number $G_h(g)$ is the number of all $\mathbb{K}$-subspaces in a $g$-dimensional $\mathbb{K}$-vector space.
\end{nota}
\begin{thm}\label{T;Theorem3.22}
The number of all closed subsets in $\mathbb{S}$ equals $2^{n_2}G_2(\log_2d_1)$. The number of all strongly normal closed subsets in $\mathbb{S}$ equals $G_2(\log_2d_1)$.
\end{thm}
\begin{proof}
Recall that $d_1\!=\!|\mathrm{O}_\vartheta(\mathbb{S})|$ and $n_2\!=\!|\{a: u_a>2\}|$. For any $g\in \mathbb{N}_0$, notice that every elementary abelian $2$-group of order $2^g$ is a $g$-dimensional vector space over a field of two elements. So $G_2(g)$ also equals the number of all elementary abelian $2$-subgroups of an elementary abelian $2$-group of order $2^g$. So the desired theorem follows from combining Notation \ref{N;Notation3.8}, Lemmas \ref{L;Lemma3.10}, \ref{L;Lemma3.18}, and \ref{L;Lemma3.20}.
\end{proof}
\begin{eg}\label{E;Example3.23}
\em Assume that $n=u_1=2$ and $u_2=3$. So $d=3$, $n_2=k_0=k_1=1$, and $d_1=k_2=k_3=2$ by Lemma \ref{L;Lemma3.11}. So $\mathrm{O}^\vartheta(\mathbb{S})=\{R_0, R_2\}$ and $\mathrm{O}_\vartheta(\mathbb{S})=\{R_0, R_1\}$. Theorems \ref{T;Theorem3.16} and \ref{T;Theorem3.22} imply that $\{R_0\}, \{R_0, R_1\}, \{R_0, R_2\}, \{R_0, R_1, R_2, R_3\}$ are precisely
all closed subsets in $\mathbb{S}$. Furthermore, Theorems \ref{T;Theorem3.17} and \ref{T;Theorem3.22} imply that $\{R_0, R_2\}$ and $\{R_0, R_1, R_2, R_3\}$ are precisely all strongly normal closed subsets in $\mathbb{S}$.
\end{eg}
\section{$\F$-Dimensions of Terwilliger $\F$-algebras of factorial schemes}
In this section, we present an explicit formula for the $\F$-dimension of $\mathbb{T}$. Moreover, we also present two $\F$-bases of $\mathbb{T}$. We recall Notation \ref{N;Notation3.3} and present a needed lemma.
\begin{lem}\label{L;Lemma4.1}
Assume that $g, h\in [0,d]$. Then there are unique $i, j, k\in [0,d]$ such that $\mathbb{P}(i)=\mathbb{P}(g)\setminus\mathbb{P}(h)$, $\mathbb{P}(j)=\mathbb{P}(g)\cup\mathbb{P}(h)$,
and $\mathbb{P}(k)=\mathbb{P}(g)\cap\mathbb{P}(h)$.
\end{lem}
\begin{proof}
The desired lemma follows from an easy observation and Lemma \ref{L;Lemma3.1}.
\end{proof}
Lemma \ref{L;Lemma4.1} motivates us to introduce the following notation and three lemmas.
\begin{nota}\label{N;Notation4.2}
\em Assume that $g, h\in [0,d]$. Then $i, j, k$ in Lemma \ref{L;Lemma4.1} are denoted by $g\setminus h, g\cup h, g\cap h$, respectively. Then $\mathbb{P}(g\setminus h)=\mathbb{P}(g)\setminus\mathbb{P}(h)$, $\mathbb{P}(g\cup h)=\mathbb{P}(g)\cup\mathbb{P}(h)$, and $\mathbb{P}(g\cap h)=\mathbb{P}(g)\cap\mathbb{P}(h)$. Set $g\oplus h=(g\setminus h)\cup(h\setminus g)$. Put $g\odot h=(g\oplus h)\cup (\widetilde{g\cap h})$. For example, if $n=u_1=2$ and $u_2=3$, notice that $d=3$, $2\oplus 3=1$, and $2\odot3=3$. The operation rules of $\setminus$, $\cup$, $\cap$ on $[0,d]$ are clear by Lemma \ref{L;Lemma4.1} and the operation rules of $\setminus$, $\cup$, $\cap$ on the power set of $[1,n]$. In particular, notice that $g\oplus h=(g\cup h)/(g\cap h)$ and $g\cup h=(g\oplus h)\cup (g\cap h)$. By Notation \ref{N;Notation3.3}, notice that $g\oplus h\leq_2g\odot h\leq_2g\cup h$.
\end{nota}
\begin{lem}\label{L;Lemma4.3}
Assume that $g, h, i\in[0,d]$. Then $p_{gh}^i\neq0$ only if $g\oplus h\leq_2 i\leq_2g\odot h$.
\end{lem}
\begin{proof}
As $p_{gh}^i\neq 0$, there are $\mathbf{w}, \mathbf{y}, \mathbf{z}\in \mathbb{X}$ such that $\mathbf{y}\in \mathbf{w}R_g$, $\mathbf{z}\in \mathbf{w}R_h$, and $\mathbf{z}\in\mathbf{y}R_i$.
There is no loss to assume that $\mathbb{P}(g\setminus h)\neq\varnothing$ and $\mathbb{P}(h\setminus g)\!\neq\!\varnothing$. Pick $j\!\in\! \mathbb{P}(g\setminus h)$. So  $j\in\mathbb{P}(g)\setminus \mathbb{P}(h)$, which implies that $\mathbf{y}_j\neq \mathbf{w}_j=\mathbf{z}_j$. So $j\in\mathbb{P}(i)$. As $j$ is chosen from $\mathbb{P}(g\setminus h)$ arbitrarily, $g\setminus h\leq_2 i$. Pick $k\in\mathbb{P}(h\setminus g)$. So $k\in \mathbb{P}(h)\setminus \mathbb{P}(g)$, which implies that $\mathbf{y}_k=\mathbf{w}_k\neq \mathbf{z}_k$ and $k\in\mathbb{P}(i)$. As $k$ is chosen from $\mathbb{P}(h\setminus g)$ arbitrarily, $h\!\setminus\! g\leq_2 i$. So $g\oplus h\leq_2 i$. For any $\ell\in [1,n]\setminus \mathbb{P}(g\cup h)$, notice that $\ell\notin \mathbb{P}(g)\cup\mathbb{P}(h)$, $\mathbf{y}_\ell\!=\!\mathbf{w}_\ell\!=\!\mathbf{z}_\ell$, and $\ell\notin\mathbb{P}(i)$. Hence $i\leq_2 g\cup h=(g\oplus h)\cup (g\cap h)$. For any $\ell\in \mathbb{P}(g\cap h)$ and $u_\ell=2$, notice that $\ell\in\mathbb{P}(g)\cap\mathbb{P}(h)$, $\mathbf{y}_\ell=\mathbf{z}_\ell\neq \mathbf{w}_\ell$, and $\ell\notin \mathbb{P}(i)$. The desired lemma thus follows from Notation \ref{N;Notation3.3}.
\end{proof}
\begin{lem}\label{L;Lemma4.4}
Assume that $g, h, i\in[0,d]$ and $\mathbb{P}(i)=\mathbb{P}(g\oplus h)\cup \mathbb{P}$, where $\mathbb{P}\subseteq \mathbb{P}_2(g\cap h)$ and $\mathbb{Q}=\mathbb{P}_2(g\cap h)\setminus \mathbb{P}$. Then $p_{gh}^i=\prod_{j\in \mathbb{P}}(u_j-2)\prod_{k\in\mathbb{Q}}(u_k-1)$, where the products over empty sets are equal to one. In particular, $p_{gh}^i\neq 0$.
\end{lem}
\begin{proof}
There are $\mathbf{y},\mathbf{z}\in \mathbb{X}$ such that $\mathbf{z}\in\mathbf{y}R_i$. Hence $\mathbf{y}_j\neq\mathbf{z}_j$ and $\mathbf{y}_k=\mathbf{z}_k$ for any $j\in \mathbb{P}(i)$
and $k\in [1,n]\setminus\mathbb{P}(i)$. As $\mathbb{P}\subseteq \mathbb{P}_2(g\cap h)$ and $\mathbb{P}(i)=\mathbb{P}(g\oplus h)\cup \mathbb{P}$, notice that $u_j>2$, $\mathbf{y}_j\neq\mathbf{z}_j$, and $|\mathbb{U}_j\setminus\{\mathbf{y}_j, \mathbf{z}_j\}|\!=\!u_j-2$ for any $j\in\mathbb{P}$. For any $j\in \mathbb{P}(g\cap h)\setminus \mathbb{P}$ and $u_j=2$, notice that $j\notin \mathbb{P}(i)$, $\mathbf{y}_j=\mathbf{z}_j$, and $|\mathbb{U}_j\setminus\{\mathbf{y}_j\}|\!=\!1$. For any $j\in \mathbb{P}(g\cap h)\setminus \mathbb{P}$ and $u_j\!>\!2$, notice that $j\notin \mathbb{P}(i)$, $\mathbf{y}_j\!=\!\mathbf{z}_j$, and $|\mathbb{U}_j\setminus\{\mathbf{y}_j\}|\!=\!u_j-1$. Then $\mathbf{w}\in\mathbf{y}R_g\cap \mathbf{z}R_h$ if and only if $\mathbf{w}_j=\mathbf{z}_j$, $\mathbf{w}_k=\mathbf{y}_k$, $\mathbf{w}_\ell\in \mathbb{U}_\ell\setminus\{\mathbf{y}_\ell,\mathbf{z}_\ell\}$, $\mathbf{w}_m\in\mathbb{U}_m\setminus\{\mathbf{y}_m\}$, and $\mathbf{w}_q=\mathbf{y}_q=\mathbf{z}_q$ for any $j\in \mathbb{P}(g\setminus h)$, $k\in\mathbb{P}(h\setminus g)$,
$\ell\in \mathbb{P}$, $m\!\in\! \mathbb{P}(g\cap h)\setminus \mathbb{P}$, and $q\!\in\! [1,n]\setminus \mathbb{P}(g\cup h)$. As $p_{gh}^i=|\mathbf{y}R_g\cap \mathbf{z}R_h|$, the first statement thus follows. The second statement thus follows from the first one. The desired lemma thus follows.
\end{proof}
\begin{lem}\label{L;Lemma4.5}
$\mathbb{T}$ has an $\F$-basis $\{E_a^*A_bE_c^*: a\oplus b\leq_2 c\leq_2a\odot b\}$ whose cardinality is $|\{(a, b, c):a\oplus b\leq_2 c\leq_2a\odot b\}|$.
\end{lem}
\begin{proof}
The desired lemma follows from combining Lemmas \ref{L;Lemma2.10}, \ref{L;Lemma4.3}, and \ref{L;Lemma4.4}.
\end{proof}
Lemma \ref{L;Lemma4.5} allows us to finish the proof of the first main result of this paper. %deduce the first main result of this paper and a lemma.
\begin{thm}\label{T;Dimension}
The $\F$-dimension of $\mathbb{T}$ equals $$\sum_{g=0}^{n_2}\sum_{h=0}^{n-n_2}\sum_{i=0}^g\sum_{j=0}^h\binom{n_2}{g}\binom{n-n_2}{h}\binom{g}{i}\binom{h}{j}2^{n-g-h+i}.$$
\end{thm}
\begin{proof}
Recall that $n_2\!\!=\!\!|\{a: u_a>2\}|$. For any $g\in [0, n_2]$, $h\!\in\! [0, n-n_2]$, $i\in [0, g]$, and $j\in [0,h]$, notice that the number of all $3$-tuples $(a, b, c)$ that satisfy the conditions $a,b,c\in [0,d]$, $|\mathbb{P}_2(a)|=g$, $|\mathbb{P}(a)\setminus \mathbb{P}_2(a)|=h$, $|\mathbb{P}_2(a\cap b)|=i$, $|\mathbb{P}(a\cap b)\setminus\mathbb{P}_2(a\cap b)|=j$, and $a\oplus b\leq_2 c\leq_2a\odot b$ equals $$\binom{n_2}{g}\binom{n-n_2}{h}\binom{g}{i}\binom{h}{j}2^{n-g-h+i}.$$ The desired theorem thus follows from Lemma \ref{L;Lemma4.5} and the above discussion.
\end{proof}
Our next goal is to find another $\F$-basis of $\mathbb{T}$. We start with the following lemma.
\begin{lem}\label{L;Lemma4.7}
Assume that $g, h, i, j, k\in [0,d]$. Then $$E_g^*A_hE_i^*A_jE_k^*=\sum_{g\oplus k\leq_2\ell\leq_2 g\odot k}c_{ghijk\ell}E_g^*A_\ell E_k^*,$$
where $c_{ghijk\ell}\!=\!\overline{|\mathbf{y}R_h\cap \mathbf{x}R_i\cap\mathbf{z}R_j|}\in \F$ for any $\mathbf{y}\!\in\!\mathbf{x}R_g$ and $\mathbf{z}\!\in\! \mathbf{x}R_k\cap \mathbf{y}R_\ell$. Moreover, the constant $c_{ghijk\ell}$ only depends on $g, h, i, j, k, \ell$ and is independent of the choices of $\mathbf{y}$ and $\mathbf{z}$. In particular, if $c_{ghijk\ell}\neq \overline{0}$, then $R_\ell\in R_gR_k\cap R_hR_j$.
\end{lem}
\begin{proof}
For any $\ell\in [0,d]$, Lemmas \ref{L;Lemma2.9} and \ref{L;Lemma2.1} imply that $E_g^*A_\ell E_k^*\neq O$ if and only if $p_{gk}^\ell\neq 0$. The first statement is thus from combining \eqref{Eq;2}, Lemmas \ref{L;Lemma4.5}, \ref{L;Lemma4.3}, \ref{L;Lemma4.4}, and a direct computation. The second statement and the third statement can be proved by the first one and an easy observation. The desired lemma thus follows.
\end{proof}
For further discussion, the following notation and four lemmas are helpful to us.
\begin{nota}\label{N;Notation4.8}
\em Assume that $g, h, i, j, k\in [0,d]$. Let $m(g, h, i, j, k)$ denote the number $(g\oplus k)\cup ((\widetilde{g\cap k})\setminus i)\cup ((h\cup j)\cap(\widetilde{g\cap k})\cap i)$. Notice that $g\oplus k\leq_2 m(g, h, i, j, k)\leq_2 g\odot k$. For example, if $n=u_1=2$ and $u_2=3$, observe that $d=3$ and $m(2,3,3,2,3)=3$.
\end{nota}
\begin{lem}\label{L;Lemma4.9}
Assume that $g, h, i, j, k, \ell, m\in [0,d]$, $h\leq_2 g\cap k\cap i$, and $j\leq_2 g\cap k\cap i$. If $h\cap j=0$, $\ell=(g\oplus i)\cup h$, and $m=(k\oplus i)\cup j$, then $\ell\oplus m=(g\oplus k)\cup h\cup j$.
\end{lem}
\begin{proof}
By Lemma \ref{L;Lemma3.1}, it suffices to check that $\ell\oplus m\leq_2 (g\oplus k)\cup h\cup j\leq_2\ell\oplus m$. Notice that $\ell\!=\!(g\setminus i)\cup (i\setminus g)\cup h$, $m=(k\setminus i)\cup (i\setminus k)\cup j$, and $\ell\oplus m=(\ell\setminus m)\cup(m\setminus \ell)$. Then $(g\setminus i)\cap(\ell\oplus m)\leq_2 g\setminus k$ and $(i\setminus g)\cap (\ell\oplus m)\leq_2k\setminus g$. Furthermore, notice that $(k\setminus i)\cap(\ell\oplus m)\leq_2k\setminus g$ and $(i\setminus k)\cap(\ell\oplus m)\leq_2g\setminus k$. Hence $\ell\oplus m\leq_2 (g\oplus k)\cup h\cup j$.

As $h\cup j\leq_2 g\cap k\cap i$ and $h\cap j=0$, notice that $h\leq_2 \ell\setminus m$ and $j\leq_2 m\setminus \ell$. It is clear that $g\setminus k=(g\setminus (k\cup i))\cup((g\cap i)\setminus k)\leq_2(\ell\setminus m)\cup(m\setminus \ell)=\ell\oplus m$. Furthermore, $k\setminus g=(k\setminus (g\cup i))\cup((i\cap k)\setminus g)\leq_2 (\ell\setminus m)\cup(m\setminus \ell)=\ell\oplus m$. Therefore $(g\oplus k)\cup h\cup j\leq_2 \ell\oplus m$. The desired lemma thus follows.
\end{proof}
\begin{lem}\label{L;Lemma4.10}
Assume that $g, h, i, j, k,\ell\in [0,d]$. Assume that $g\oplus i\leq_2 h\leq_2 g\odot i$, $k\oplus i\leq_2 j\leq_2 k\odot i$, and $g\oplus k\leq_2 \ell\leq_2 m(g, h, i, j, k)$.
Then there exist $m, q\in [0,d]$ such that $g\oplus i\leq_2 m\leq_2 h$, $k\oplus i\leq_2 q\leq_2 j$, and $R_\ell\in R_gR_k\cap R_mR_q$.
\end{lem}
\begin{proof}
Set $r=(\widetilde{g\cap k})\cap i\leq_2g\cap k\cap i$. Since $g\oplus k\leq_2 \ell\leq_2 m(g, h, i, j, k)\leq_2g\odot k$, notice that $R_\ell\in R_gR_k$ by Lemma \ref{L;Lemma4.4}. There are $s, t\in[0,d]$ such that $s\leq_2 h\cap r\cap \ell$, $t\leq_2 j\cap r\cap\ell$, $s\cup t=(h\cup j)\cap r\cap \ell$, and $s\cap t=0$. Define $m=(g\oplus i)\cup s$ and $q=(k\oplus i)\cup t$. As $g\oplus i\leq_2 h\leq_2 g\odot i$ and $k\oplus i\leq_2 j\leq_2 k\odot i$, it is obvious that $g\oplus i\leq_2 m\leq_2 h$ and $k\oplus i\leq_2 q\leq_2 j$. Since $s\cap t=0$ and $g\oplus k\leq_2 \ell\leq_2 m(g, h, i, j, k)$, notice that $m\oplus q=(g\oplus k)\cup s\cup t=(g\oplus k)\cup((h\cup j)\cap r\cap \ell)\leq_2 \ell$ by Lemma \ref{L;Lemma4.9}. As $s\cap t=0$, $m=(g\setminus i)\cup(i\setminus g)\cup s$, and $q=(k\setminus i)\cup(i\setminus k)\cup t$, it is not difficult to see that $m\cap q=((g\cap k)\setminus i)\cup(i\setminus(g\cup k))$. As $g\oplus k\leq_2 \ell\leq_2 m(g, h, i, j, k)$, notice that $\ell\leq_2(g\oplus k)\cup ((\widetilde{g\cap k})\setminus i)\cup ((h\cup j)\cap r\cap \ell)\leq_2 (m\oplus q)\cup(\widetilde{m\cap q})=m\odot q$. So $m\oplus q\leq_2 \ell\leq_2 m\odot q$ and $R_\ell\in R_mR_q$ by Lemma \ref{L;Lemma4.4}. The desired lemma follows.
\end{proof}
\begin{lem}\label{L;Lemma4.11}
Assume that $g, h, i, j, k,\ell\in [0,d]$. Assume that $g\oplus i\leq_2 h\leq_2 g\odot i$ and $k\oplus i\leq_2 j\leq_2 k\odot i$. If there exist $m, q\in [0,d]$ such that $g\oplus i\leq_2 m\leq_2 h$, $k\oplus i\leq_2 q\leq_2 j$, and $R_\ell\in R_gR_k\cap R_mR_q$, then $g\oplus k\leq_2 \ell\leq_2 m(g, h, i, j, k)$.
\end{lem}
\begin{proof}
Set $r=(\widetilde{g\cap k})\cap i\leq_2 g\cap k\cap i$. Notice that $g\odot k=(g\oplus k)\cup((\widetilde{g\cap k})\setminus i)\cup r$. As $R_\ell\in R_gR_k\cap R_mR_q$, Lemma \ref{L;Lemma4.3} implies that $g\oplus k\leq_2\ell\leq_2 g\odot k$ and $\ell\leq m\odot q$. As $\ell=(g\odot k)\cap \ell$, it suffices to check that $r\cap \ell \leq_2 r\cap (m\odot q)\leq_2m(g, h, i, j, k)$. As $g\oplus i\leq_2 h\leq_2 g\odot i$, $k\oplus i\leq_2 j\leq_2 k\odot i$, $g\oplus i\leq_2 m\leq_2 h$, $k\oplus i\leq_2 q\leq_2 j$, notice that there exist $s, t\in [0,d]$ such that $s\leq_2(\widetilde{g\cap i})\cap h$, $t\leq_2(\widetilde{k\cap i})\cap j$, $m=(g\oplus i)\cup s$, and $q=(k\oplus i)\cup t$. Therefore $m=(g\setminus i)\cup(i\setminus g)\cup s$ and $q=(k\setminus i)\cup(i\setminus k)\cup t$. Notice that $((g\setminus i)\cup(i\setminus g))\cap(m\oplus q)\leq_2 (g\setminus k)\cup (k\setminus g)$. Moreover, it is clear that $((k\setminus i)\cup (i\setminus k))\cap(m\oplus q)\leq_2 (g\setminus k)\cup (k\setminus g)$. So $m\oplus q\leq_2 (g\setminus k)\cup (k\setminus g)\cup s\cup t$.
Hence $r\cap (m\oplus q)\leq_2 r\cap( (g\setminus k)\cup (k\setminus g)\cup s\cup t)=(r\cap s)\cup(r\cap t)\leq_2(h\cup j)\cap r$ and $r\cap m\cap q=r\cap ((g\setminus i)\cup(i\setminus g)\cup s)\cap ((k\setminus i)\cup(i\setminus k)\cup t)=r\cap s\cap t$. Notice that $r\cap (m\odot q)=r\cap ((m\oplus q)\cup (\widetilde{m\cap q}))\leq_2(r\cap s)\cup (r\cap t)\cup (r\cap s\cap t)$. Hence $r\cap\ell\!\leq_2\! r\cap (m\odot q)\!\leq_2\!(h\cup j)\cap r\!\leq_2\! m(g, h, i, j, k)$. The desired lemma thus follows.
\end{proof}
\begin{lem}\label{L;Lemma4.12}
Assume that $g, h, i, j, k\in [0,d]$, $p_{gh}^i\neq0$, and $p_{ij}^k\neq0$. Then $$ (\sum_{g\oplus i\leq_2\ell\leq_2 h}E_g^*A_\ell E_i^*)(\sum_{k\oplus i\leq_2m\leq_2 j}E_i^*A_mE_k^*)=\sum_{g\oplus k\leq_2 q\leq_2 m(g, h, i, j, k)}c_q E_g^*A_q E_k^*,$$ where $c_q\in \F$ for any $g\oplus k\leq_2 q\leq_2 m(g, h, i, j, k)$.
\end{lem}
\begin{proof}
As $p_{gh}^i\neq0$ and $p_{ij}^k\neq 0$, notice that $p_{gi}^h\neq 0$ and $p_{ik}^j\neq 0$ by Lemma \ref{L;Lemma2.1}. Hence $g\oplus i\leq_2 h\leq_2 g\odot i$ and $k\oplus i\leq_2 j\leq_2 k\odot i$ by Lemma \ref{L;Lemma4.3}. By Lemmas \ref{L;Lemma4.10} and \ref{L;Lemma4.11}, $R_gR_k\cap\{R_a: g\oplus i\leq_2 a\leq_2 h\}\{R_a: k\oplus i\leq_2 a\leq_2 j\}$ is precisely $\{R_a: g\oplus k\leq_2 a\leq_2m(g, h, i, j, k)\}$. The desired lemma follows from Lemma \ref{L;Lemma4.7}.
\end{proof}
For further discussion, the following notation and five lemmas are necessary.
\begin{nota}\label{N;Notation4.13}
\em Assume that $\mathbf{y}, \mathbf{z}\in\mathbb{X}$ and $g, h\in [0,d]$. Then $\mathbf{y}(g, \mathbf{z})$ is the element in $\mathbb{X}$ that satisfies the equalities $\mathbf{y}(g, \mathbf{z})_i=\mathbf{y}_i$ and $\mathbf{y}(g, \mathbf{z})_j=\mathbf{z}_j$ for any $i\in\mathbb{P}(g)$ and $j\in [1,n]\setminus \mathbb{P}(g)$. Assume that $\mathbf{z}\in\mathbf{y}R_h$. It is obvious to see that $\mathbf{z}\in\mathbf{y}(g,\mathbf{z})R_{g\cap h}$.
\end{nota}
\begin{lem}\label{L;Lemma4.14}
Assume that $\mathbf{y}, \mathbf{z}\in\mathbb{X}$, $g, h, i, j\in[0,d]$, $p_{hi}^j\neq 0$, $h\!\leq_2\!g$, and $\mathbf{z}\!\in\!\mathbf{y}R_j$. Then $\mathbf{z}R_h\cap \mathbf{y}R_i=\mathbf{z}R_h\cap \mathbf{y}(g,\mathbf{z})R_{g\cap i}$, $\mathbf{z}\in \mathbf{y}(g,\mathbf{z})R_{g\cap j}$, $i\oplus j\leq_2(g\cap i)\oplus (g\cap j)\leq_2 g$.
\end{lem}
\begin{proof}
Since $p_{hi}^j\neq 0$, notice that $\mathbf{z}R_h\cap\mathbf{y}R_i\neq\varnothing$. Pick $\mathbf{v}\in\mathbf{z}R_h\cap\mathbf{y}R_i$. Notice that $\mathbf{v}_k\neq \mathbf{y}_k$
and $\mathbf{v}_\ell=\mathbf{y}_\ell$ for any $k\in \mathbb{P}(i)$ and $\ell\in [1,n]\setminus \mathbb{P}(i)$. Then $\mathbf{v}_k\neq\mathbf{y}_k=\mathbf{y}(g,\mathbf{z})_k$
and $\mathbf{v}_\ell\!=\!\mathbf{y}_\ell\!=\!\mathbf{y}(g,\mathbf{z})_\ell$ for any $k\in \mathbb{P}(g\cap i)$ and $\ell\in\mathbb{P}(g\setminus i)$. As $h\!\leq_2\!g$ and $\mathbf{v}_k=\mathbf{z}_k$ for any $k\in [1,n]\setminus \mathbb{P}(h)$, $\mathbf{v}_k=\mathbf{z}_k=\mathbf{y}(g,\mathbf{z})_k$ for any $k\in [1,n]\setminus\mathbb{P}(g)$.
So $\mathbf{v}\in\mathbf{y}(g,\mathbf{z})R_{g\cap i}$. As $\mathbf{v}$ is chosen from $\mathbf{z}R_h\cap\mathbf{y}R_i$ arbitrarily, notice that $\mathbf{z}R_h\cap\mathbf{y}R_i\!\subseteq\!\mathbf{z}R_h\cap\mathbf{y}(g,\mathbf{z})R_{g\cap i}$.

As $\varnothing\neq \mathbf{z}R_h\cap\mathbf{y}R_i\subseteq\mathbf{z}R_h\cap\mathbf{y}(g,\mathbf{z})R_{g\cap i}$, notice that $\mathbf{z}R_h\cap\mathbf{y}(g,\mathbf{z})R_{g\cap i}\neq\varnothing$. Pick $\mathbf{w}\in\mathbf{z}R_h\cap\mathbf{y}(g,\mathbf{z})R_{g\cap i}$. So $\mathbf{w}_k\neq\mathbf{y}(g,\mathbf{z})_k$ and $\mathbf{w}_\ell=\mathbf{y}(g,\mathbf{z})_\ell$ for any $k\in \mathbb{P}(g\cap i)$ and $\ell\in [1,n]\setminus\mathbb{P}(g\cap i)$. So
$\mathbf{w}_k\neq\mathbf{y}(g,\mathbf{z})_k=\mathbf{y}_k$ and $\mathbf{w}_\ell=\mathbf{y}(g,\mathbf{z})_\ell=\mathbf{z}_\ell$ for any $k\in\mathbb{P}(g\cap i)$ and $\ell\in\mathbb{P}(i\setminus g)$. As $h\!\leq_2\!g$, notice that $i\setminus g\leq_2 i\setminus h$. As $\mathbf{z}\in\mathbf{y}R_j$, Lemma \ref{L;Lemma4.3} implies that $i\setminus g\leq_2 i\setminus h\leq_2 j$ and  $\mathbf{w}_k\!=\!\mathbf{z}_k\!\neq\!\mathbf{y}_k$ for any $k\in\mathbb{P}(i\setminus g)$. For any $k\in \mathbb{P}(g\setminus i)$, notice that $\mathbf{w}_k=\mathbf{y}(g,\mathbf{z})_k=\mathbf{y}_k$. As $h\leq_2 g$ and $\mathbf{z}\in\mathbf{y}R_ j$, Lemma \ref{L;Lemma4.3} implies that $j\leq_2 h\odot i\leq_2 h\cup i\leq_2 g\cup i$ and $\mathbf{w}_k\!=\!\mathbf{y}(g,\mathbf{z})_k\!=\!\mathbf{z}_k\!=\!\mathbf{y}_k$ for any $k\in[1,n]\setminus\mathbb{P}(g\cup i)$. Hence $\mathbf{w}\in \mathbf{z}R_h\cap\mathbf{y}R_i$. As $\mathbf{w}$ is chosen from $\mathbf{z}R_h\cap\mathbf{y}(g,\mathbf{z})R_{g\cap i}$ arbitrarily, the first statement thus follows. Since $\mathbf{z}\!\in\!\mathbf{y}R_j$, $\mathbf{z}_k\neq \mathbf{y}_k=\mathbf{y}(g, \mathbf{z})_k$ and $\mathbf{z}_\ell=\mathbf{y}_\ell=\mathbf{y}(g, \mathbf{z})_\ell$ for any $k\!\in\!\mathbb{P}(g\cap j)$ and $\ell\in \mathbb{P}(g\setminus j)$. Moreover, $\mathbf{z}_k=\mathbf{y}(g,\mathbf{z})_k$ for any $k\in [1,n]\setminus \mathbb{P}(g)$. The second statement thus follows. For the third statement, notice that $g\cap i\leq_2 g$, $g\cap j\leq_2 g$, and $(g\cap i)\oplus (g\cap j)\leq_2g$. As $p_{hi}^j\neq 0$ and $h\leq_2 g$, notice that $i\oplus j\leq_2 h\leq_2 g$ by Lemmas \ref{L;Lemma2.1} and \ref{L;Lemma4.3}. So $i\oplus j\leq_2 (g\cap i)\oplus (g\cap j)$. So $i\oplus j\leq_2(g\cap i)\oplus (g\cap j)\leq_2 g$. The desired lemma thus follows.
\end{proof}
\begin{lem}\label{L;Lemma4.15}
Assume that $\mathbf{y}, \mathbf{z}\in \mathbb{X}$ and $g, h\in[0,d]$. Then $$\sum_{i\leq_2 g}|\mathbf{z}R_i\cap\mathbf{y}R_h|\leq \sum_{i\leq_2 g}|\mathbf{z}R_i\cap\mathbf{y}(g,\mathbf{z})R_{g\cap h}|\leq |\mathbf{y}(g,\mathbf{z})R_{g\cap h}|=k_{g\cap h}.$$
\end{lem}
\begin{proof}
There is no loss to assume that the leftmost side of the desired inequality is not equal to zero. Assume that $j\leq_2 g$ and $|\mathbf{z}R_j\cap\mathbf{y}R_h|\neq 0$. By Lemma \ref{L;Lemma4.14}, notice that $|\mathbf{z}R_j\cap\mathbf{y}R_h|=|\mathbf{z}R_j\cap\mathbf{y}(g,\mathbf{z})R_{g\cap h}|\leq |\mathbf{y}(g,\mathbf{z})R_{g\cap h}|=k_{g\cap h}$. As  $\{\mathbf{z}R_a: a\in[0,d]\}$ forms a partition of $\mathbb{X}$, the desired inequality thus follows.
\end{proof}
\begin{lem}\label{L;Lemma4.16}
Assume that $\mathbf{y}, \mathbf{z}\!\in\! \mathbb{X}$ and $g, h ,i, j, k, \ell\!\in\! [0,d]$. Assume that $g\oplus i\leq_2 h$, $i\setminus j\leq_2 k\leq_2 i\cup j$, and $g\oplus k\leq_2 \ell$. Assume that $\mathbf{y}\in\mathbf{x}R_g\cap\mathbf{z}R_\ell$, $\mathbf{v}=\mathbf{x}(j, \mathbf{z})$, and $\mathbf{w}\in\mathbf{v}(h,\mathbf{y})R_{h\cap i\cap j}$. Then there is $m\!\in\! [0,d]$ such that $g\oplus i\leq_2 m\leq_2h$ and $\mathbf{w}\!\in\!\mathbf{y}R_m$.
\end{lem}
\begin{proof}
There exists $m\in [0,d]$ such that $\mathbf{w}\in\mathbf{y}R_m$. As $\mathbf{w}\in\mathbf{v}(h,\mathbf{y})R_{h\cap i\cap j}$, $\mathbf{w}_q=\mathbf{y}_q$ for any $q\notin\mathbb{P}(h)$. Hence $m\leq_2 h$ as $\mathbf{w}\in\mathbf{y}R_m$. As $g\oplus i\leq_2 h$, $(g\setminus i)\cup(i\setminus g)\leq_2 h$. As $\mathbf{w}\in\mathbf{v}(h,\mathbf{y})R_{h\cap i\cap j}$ and $\mathbf{y}\in \mathbf{x}R_g$, $\mathbf{w}_q=\mathbf{v}_q=\mathbf{x}_q\neq \mathbf{y}_q$ for any $q\in \mathbb{P}(g\cap j)\setminus \mathbb{P}(i)$. As $g\oplus k\leq_2 \ell$, it is obvious to see that $(g\setminus k)\cup(k\setminus g)\leq_2 \ell$. As $i\setminus j\leq_2 k\leq_2 i\cup j$, $(g\setminus k)\cup (k\setminus g)\leq_2\ell$, $\mathbf{w}\in\mathbf{v}(h,\mathbf{y})R_{h\cap i\cap j}$, and $\mathbf{y}\in\mathbf{z}R_\ell$, it is obvious to see that $\mathbf{w}_q\!=\!\mathbf{v}_q\!=\!\mathbf{z}_q\!\neq\!\mathbf{y}_q$ for any $q\!\in\! (\mathbb{P}(g)\setminus \mathbb{P}(i\cup j))\cup(\mathbb{P}(i)\setminus\mathbb{P}(g\cup j))$. For any $q\in \mathbb{P}(i\cap j)\setminus\mathbb{P}(g)$, observe that $\mathbf{w}_q\neq \mathbf{v}_q=\mathbf{x}_q=\mathbf{y}_q$ since $\mathbf{w}\in\mathbf{v}(h,\mathbf{y})R_{h\cap i\cap j}$ and $\mathbf{y}\in \mathbf{x}R_g$. Therefore $g\oplus i=((g\cap j)\setminus i)\cup(g\setminus(i\cup j))\cup((i\cap j)\setminus g)\cup (i\setminus(g\cup j))\leq_2 m$ by the assumption $\mathbf{w}\in\mathbf{y}R_m$. Hence $g\oplus i\leq_2 m\leq_2h$. The desired lemma thus follows.
\end{proof}
\begin{lem}\label{L;Lemma4.17}
Assume that $\mathbf{y}, \mathbf{z}\in \mathbb{X}$, $g, h ,i, j, k\in [0,d]$, $h\oplus i\leq_2 g\leq_2 h\cup i$, and $k\leq_2 h\cup j$. Assume that $\mathbf{y}\in\mathbf{x}R_g\cap\mathbf{z}R_k$ and $\mathbf{v}\!=\!\mathbf{x}(j, \mathbf{z})$. Then $\mathbf{v}(h,\mathbf{y})R_{h\cap i\cap j}\subseteq\mathbf{v}R_{i\cap j}$.
\end{lem}
\begin{proof}
Pick $\mathbf{w}\in\mathbf{v}(h,\mathbf{y})R_{h\cap i\cap j}$. It is obvious that $\mathbf{w}_\ell\neq \mathbf{v}_\ell$ for any $\ell\in \mathbb{P}(h\cap i\cap j)$. As $h\oplus i\leq_2 g$, notice that $(h\setminus i)\cup(i\setminus h)\leq_2 g$. As $\mathbf{w}\in\mathbf{v}(h,\mathbf{y})R_{h\cap i\cap j}$, $\mathbf{y}\in\mathbf{x}R_g$, and $(i\cap j)\setminus h\leq_2g$, $\mathbf{w}_\ell=\mathbf{y}_\ell\neq \mathbf{x}_\ell=\mathbf{v}_\ell$ for any $\ell\in\mathbb{P}(i\cap j)\setminus \mathbb{P}(h)$. As $\mathbf{w}\in\mathbf{v}(h,\mathbf{y})R_{h\cap i\cap j}$, notice that $\mathbf{w}_\ell=\mathbf{v}_\ell$ for any $\ell\in \mathbb{P}(h\cap i)\setminus \mathbb{P}(j)$. As $\mathbf{w}\in\mathbf{v}(h,\mathbf{y})R_{h\cap i\cap j}$, $k\leq_2 h\cup j$, and $\mathbf{y}\in\mathbf{z}R_k$, $\mathbf{w}_\ell=\mathbf{y}_\ell=\mathbf{z}_\ell=\mathbf{v}_\ell$ for any $\ell\in \mathbb{P}(i)\setminus\mathbb{P}(h\cup j)$. As $\mathbf{w}\in\mathbf{v}(h,\mathbf{y})R_{h\cap i\cap j}$, notice that $\mathbf{w}_\ell=\mathbf{v}_\ell$ for any $\ell\in \mathbb{P}(h\cap j)\setminus \mathbb{P}(i)$. As $\mathbf{w}\in\mathbf{v}(h,\mathbf{y})R_{h\cap i\cap j}$, $\mathbf{y}\in\mathbf{x}R_g$, and $g\leq_2 h\cup i$, $\mathbf{w}_\ell=\mathbf{y}_\ell=\mathbf{x}_\ell=\mathbf{v}_\ell$ for any $\ell\in \mathbb{P}(j)\setminus\mathbb{P}(h\cup i)$. As $\mathbf{w}\!\in\!\mathbf{v}(h,\mathbf{y})R_{h\cap i\cap j}$, it is clear that $\mathbf{w}_\ell=\mathbf{v}_\ell$ for any $\ell\in \mathbb{P}(h)\setminus \mathbb{P}(i\cup j)$. As $\mathbf{w}\!\in\!\mathbf{v}(h,\mathbf{y})R_{h\cap i\cap j}$, $\mathbf{y}\in\mathbf{z}R_k$, and $k\leq_2 h\cup j\leq_2 h\cup i\cap j$, notice that $\mathbf{w}_\ell=\mathbf{y}_\ell=\mathbf{z}_\ell=\mathbf{v}_\ell$ for any $\ell\in [1,n]\setminus\mathbb{P}(h\cup i\cup j)$. In conclusion, the above discussion shows that $\mathbf{w}_\ell\neq\mathbf{v}_\ell$ if and only if $\ell\in \mathbb{P}(i\cap j)$. As $\mathbf{w}$ is chosen from $\mathbf{v}(h,\mathbf{y})R_{h\cap i\cap j}$ arbitrarily, the desired lemma thus follows.
\end{proof}
\begin{lem}\label{L;Lemma4.18}
Assume that $\mathbf{y}, \mathbf{z}\in \mathbb{X}$, $g, h, i, j, k, \ell\in [0,d]$, $p_{gh}^i\neq0$, and $p_{ij}^k\neq 0$. Assume that $g\oplus k\leq_2 \ell\leq_2 m(g, h, i, j, k)$, $\mathbf{y}\in \mathbf{x}R_g\cap \mathbf{z}R_\ell$, and $\mathbf{z}\in \mathbf{x}R_k$. Then
$$ \sum_{g\oplus i\leq_2 m\leq_2 h}\sum_{k\oplus i\leq_2 q\leq_2 j}|\mathbf{y}R_m\cap\mathbf{x}R_i\cap\mathbf{z}R_q|=k_{h\cap i\cap j}.$$
\end{lem}
\begin{proof}
As $p_{gh}^i\neq 0$ and $p_{ij}^k\neq 0$, notice that $g\oplus i\leq_2 h\leq_2 g\odot i$ and $k\oplus i\leq_2 j\leq_2 k\odot i$ by Lemmas \ref{L;Lemma2.1} and \ref{L;Lemma4.3}. Hence the left part of the desired equality is defined. Notice that $k\oplus i\leq_2 q\leq_2 j\leq_2 k\odot i$ for any $k\oplus i\leq_2 q\leq_2 j$. So Lemmas \ref{L;Lemma4.4} and \ref{L;Lemma2.1} show that $p_{iq}^k\neq 0$ for any $k\oplus i\leq_2 q\leq_2 j$. Hence $k\oplus i\leq_2 (k\cap j)\oplus (i\cap j)\leq_2 j$ by Lemma \ref{L;Lemma4.14}. Pick $r\in [0,d]$. As $\mathbf{z}\in \mathbf{x}R_k$, the combination of Lemmas \ref{L;Lemma4.14}, \ref{L;Lemma2.1}, \ref{L;Lemma3.5}, and \ref{L;Lemma4.3} implies that $|\mathbf{x}(j,\mathbf{z})R_{i\cap j}\cap\mathbf{z}R_r|\!\neq\! 0$ only if $k\oplus i\leq_2 r\leq_2 j$. As $\{\mathbf{z}R_a: a\in [0,d]\}$ forms a partition of $\mathbb{X}$ and Lemmas \ref{L;Lemma4.14}, \ref{L;Lemma4.15} hold, the following inequality holds:
\begin{align*}
\sum_{g\oplus i\leq_2m\leq_2 h}\sum_{k\oplus i\leq_2q\leq_2 j} |\mathbf{y}R_m\cap\mathbf{x}R_i\cap\mathbf{z}R_q|\!=\!&\sum_{g\oplus i\leq_2m\leq_2 h}\sum_{k\oplus i\leq_2q\leq_2 j}|\mathbf{y}R_m\cap\mathbf{x}(j,\mathbf{z})R_{i\cap j}\!\cap\! \mathbf{z}R_q|\\
=&\sum_{g\oplus i\leq_2m\leq_2 h}|\mathbf{y}R_m\cap\mathbf{x}(j,\mathbf{z})R_{i\cap j}|\leq k_{h\cap i\cap j}.
\end{align*}

As $g\oplus i\leq_2 h\leq_2 g\odot i$, $k\oplus i\leq_2 j\leq_2 k\odot i$, and $g\oplus k\leq_2 \ell\leq_2 m(g, h, i, j, k)$, Lemma \ref{L;Lemma4.10} thus implies that there exist $s, t\in [0,d]$ such that $g\oplus i\leq_2 s\leq_2 h$, $k\oplus i\leq_2 t\leq_2 j$, and $R_\ell\in R_gR_k\cap R_sR_t$. Hence $\ell\leq_2 s\odot t\leq_2 h\cup j$ by Lemma \ref{L;Lemma4.3}. Set $\mathbf{v}=\mathbf{x}(j,\mathbf{z})$. As $p_{gh}^i\neq 0$ and $p_{ij}^k\neq 0$, notice that $h\oplus i\leq_2 g\leq_2 h\odot i\leq_2 h\cup i$ and $i\setminus j\leq_2 i\oplus j\leq_2 k\leq_2 i\odot j\leq_2 i\cup j$ by Lemmas \ref{L;Lemma2.1} and \ref{L;Lemma4.3}. Lemmas \ref{L;Lemma4.16} and \ref{L;Lemma4.17} thus imply the following containment
$$ \mathbf{v}(h,\mathbf{y})R_{h\cap i\cap j}\subseteq\bigcup_{g\oplus i\leq_2 m\leq_2 h}(\mathbf{y}R_m\cap \mathbf{v}R_{i\cap j}).$$
The desired equality thus follows since $\{\mathbf{y}R_a: a\in [0,d]\}$ forms a partition of $\mathbb{X}$.
\end{proof}
The following notation and two lemmas motivate us to give another $\F$-basis of $\mathbb{T}$.
\begin{nota}\label{N;Notation4.19}
\em Assume that $g, h, i\in [0,d]$ and $p_{gh}^i\neq 0$. Then $\sum_{g\oplus i\leq_2j\leq_2h}E_g^*A_jE_i^*$ is denoted by $B_{g, h, i}$. As $p_{gh}^i\!\neq\! 0$, notice that $O\neq B_{g, h, i}\in \mathbb{T}$ by Lemmas \ref{L;Lemma4.3} and \ref{L;Lemma2.10}. Notice that $|\{B_{a,b, c}: p_{ab}^c\neq 0\}|=|\{(a, b, c): p_{ab}^c\neq 0\}|$ by \eqref{Eq;2} and Lemma \ref{L;Lemma2.10}.
\end{nota}
\begin{lem}\label{L;Lemma4.20}
Assume that $g, h, i, j, k, \ell\in [0,d]$, $p_{gh}^i\neq0$, and $p_{\ell j}^k\neq0$. Then
$$\ B_{g, m(g,h,i,j,k), k}\ \text{is defined and}\ B_{g,h,i}B_{\ell,j,k}=\delta_{i\ell}\overline{k_{h\cap i\cap j}}B_{g, m(g,h,i,j,k), k}.$$
\end{lem}
\begin{proof}
Set $\mathbb{U}\!\!=\!\!\{E_g^*A_aE_k^*: g\oplus k\leq_2 a\leq_2 m(g, h,i,j,k)\}$. By combining Lemmas \ref{L;Lemma4.4}, \ref{L;Lemma2.1}, and \ref{L;Lemma2.9}, notice that $O\notin \mathbb{U}$ and $B_{g,m(g, h,i,j,k), k}$ is defined. If $i\neq \ell$, notice that $B_{g,h,i}B_{\ell, j, k}=B_{g,h,i}E_i^*E_\ell^* B_{\ell, j, k}=O$ by \eqref{Eq;2}. Assume that $i=\ell$.
By Lemma \ref{L;Lemma4.12}, $B_{g, h,i}B_{i, j,k}$ is an $\F$-linear combination of the elements in $\mathbb{U}$. If $E_g^*A_mE_k^*\in \mathbb{U}$, let $c_m$ be the coefficient of $E_g^*A_mE_k^*$ in this $\F$-linear combination of $B_{g, h,i}B_{i, j,k}$. It suffices to check that $c_m=\overline{k_{h\cap i\cap j}}$ for any $E_g^*A_mE_k^*\in \mathbb{U}$. Pick $E_g^*A_qE_k^*\in \mathbb{U}$. Notice that there exist $\mathbf{y}, \mathbf{z}\in \mathbb{X}$ such that $\mathbf{y}\in\mathbf{x}R_g\cap\mathbf{z}R_q$ and $\mathbf{z}\in\mathbf{x}R_k$.
By combining the conditions $p_{gh}^i\neq0$, $p_{ij}^k\neq0$, $g\oplus k\leq_2 q\leq_2 m(g, h,i,j,k)$, Lemmas \ref{L;Lemma4.7}, and \ref{L;Lemma4.18},
$$c_q=\sum_{g\oplus i\leq_2 r\leq_2 h}\sum_{k\oplus i\leq_2 s\leq_2 j}\overline{|\mathbf{y}R_r\cap\mathbf{x}R_i\cap\mathbf{z}R_s|}=\overline{k_{h\cap i\cap j}}.$$
The desired lemma thus follows as $E_g^*A_qE_k^*$ is chosen from $\mathbb{U}$ arbitrarily.
\end{proof}
\begin{lem}\label{L;Lemma4.21}
$\mathbb{T}$ has an $\F$-linearly independent subset $\{B_{a,b,c}:a\oplus b\leq_2 c\leq_2 a\odot b\}$.
\end{lem}
\begin{proof}
Set $\mathbb{U}\!=\!\{B_{a,b,c}: a\oplus b\leq_2 c\leq_2 a\odot b\}$. According to Lemma \ref{L;Lemma4.4}, notice that $M\!\neq\! O$ for any $M\!\in\! \mathbb{U}$. For any $B_{g, h, i}, B_{j, k, \ell}\in \mathbb{U}$, Notation \ref{N;Notation4.19} thus implies that $B_{g, h, i}=B_{j, k, \ell}$ if and only if $g\!=\!j$, $h=k$, $i=\ell$. Assume that $\sum_{M\in \mathbb{U}}c_MM=O$ and $c_M\in \F$ for any $M\in \mathbb{U}$. It suffices to check that $c_M=\overline{0}$ for any $M\in\mathbb{U}$. Assume that there is $N\in \mathbb{U}$ such that $c_N\neq \overline{0}$. For any $B_{g, h, i}\in \mathbb{U}$ and $j ,k\in [0,d]$, notice that $E_j^*B_{g, h ,i}E_k^*\!=\!\delta_{gj}\delta_{ik}B_{g,h,i}$ by \eqref{Eq;2}. By \eqref{Eq;3} and \eqref{Eq;2}, $N=INI=E_m^*NE_q^*$ for some $m, q\in [0,d]$. So $\mathbb{V}=\{M: M\in\mathbb{U},\ c_M\neq \overline{0},\ E_m^*ME_q^*=M\}\neq\varnothing$. Therefore there exist $r\in \mathbb{N}_0\setminus\{0\}$ and $s_1, s_2,\ldots, s_r\in [0,d]$ such that $s_1, s_2,\ldots, s_r$ are pairwise distinct and $\mathbb{V}=\{B_{m, s_1, q}, B_{m, s_2, q},\ldots, B_{m, s_r, q}\}$. If $r=1$, notice that $c_NN=O$ and $c_N=\overline{0}$ by \eqref{Eq;2}. It is absurd. So $r>1$. By Lemma \ref{L;Lemma3.1}, there is no loss to assume that $s_1$ is a maximum element of $\{s_1, s_2, \ldots, s_r\}$ with respect to $\leq_2$. By the choices of $s_1, s_2, \ldots, s_r$, observe that $B_{m, s_1, q}$ is an $\F$-linear combination of the elements in $\{B_{m, s_2, q}, B_{m, s_3, q}\ldots, B_{m, s_r, q}\}$. It is absurd by combining the choices of $s_1, s_2, \ldots, s_r$, Notation \ref{N;Notation4.19}, Lemma \ref{L;Lemma4.5}. So $c_M\!=\!\overline{0}$ for any $M\!\in\! \mathbb{U}$. The desired lemma follows.
\end{proof}
We close this section by the other main result of this section and an example.
\begin{thm}\label{T;Theorem4.22}
$\mathbb{T}$ has an $\F$-basis $\{B_{a,b,c}: a\oplus b\leq_2 c\leq_2 a\odot b\}$ whose cardinality is $|\{(a, b, c): a\oplus b\leq_2 c\leq_2 a\odot b\}|$.
\end{thm}
\begin{proof}
The theorem is from combining Lemmas \ref{L;Lemma4.21}, \ref{L;Lemma4.3}, \ref{L;Lemma4.4}, \ref{L;Lemma4.5}, Notation \ref{N;Notation4.19}.
\end{proof}
\begin{eg}\label{E;Example4.23}
\em Assume that $n\!=\!u_1\!=\!2$ and $u_2\!=\!3$. Therefore $d=3$ and $\mathbb{T}$ has an $\F$-basis containing precisely $B_{0,0,0}$, $B_{0,1,1}$, $B_{0,2,2}$, $B_{0,3,3}$, $B_{1,0,1}$, $B_{1,1,0}$, $B_{1,2,3}$, $B_{1,3,2}$, $B_{2,0,2}$, $B_{2,1,3}$, $B_{2,2,0}$, $B_{2,2,2}$, $B_{2,3,1}$, $B_{2,3,3}$, $B_{3,0,3}$, $B_{3,1,2}$, $B_{3,2,1}$, $B_{3,2,3}$, $B_{3,3,0}$, $B_{3,3,2}$ by Theorem \ref{T;Theorem4.22} and a direct computation. Then $B_{2,3,3}B_{3,2,3}=\overline{2}B_{2,3,3}$ by Lemma \ref{L;Lemma4.20}.
\end{eg}
\section{Centers of Terwilliger $\F$-algebras of factorial schemes}
In this section, we give an $\F$-basis of $\mathrm{Z}(\mathbb{T})$ and determine the $\F$-dimension of $\mathrm{Z}(\mathbb{T})$. For our purpose, we recall Notations \ref{N;Notation3.3}, \ref{N;Notation4.2}, \ref{N;Notation4.8}, \ref{N;Notation4.19} and begin with three lemmas.
\begin{lem}\label{L;Lemma5.1}
Assume that $g, h, i, j\!\!\in\!\! [0,d]$ and $g\oplus i\!\leq_2\! h$. Then $k_{j\setminus i}k_{h\cap i\cap j}\!=\!k_{j\setminus g}k_{g\cap h\cap j}$.
\end{lem}
\begin{proof}
As $g\oplus i\leq_2 h$, $(g\setminus i)\cup(i\setminus g)\leq_2 h$. Notice that $j\setminus i=(j\setminus(g\cup i))\cup((g\cap j)\setminus i)$ and $h\cap i\cap j\leq_2\!(j\setminus g)\cup (g\cap h\cap j)$. According to the fact $g\setminus i\leq_2 h$, notice that $j\setminus i=(j\setminus(g\cup i))\cup((g\cap j)\setminus i)\leq_2(j\setminus(g\cup i))\!\cup\!((g\cap h\cap j)\setminus i)\leq_2 (j\setminus g)\cup (g\cap h\cap j)$. Hence $(j\setminus i)\cup(h\cap i\cap j)\leq_2 (j\setminus g)\cup(g\cap h\cap j)$. By exchanging the roles of $g$ and $i$ in the fact $(j\setminus i)\cup(h\cap i\cap j)\leq_2 (j\setminus g)\cup(g\cap h\cap j)$, it is obvious to notice that $(j\setminus g)\cup(g\cap h\cap j)\leq_2 (j\setminus i)\cup(h\cap i\cap j)$.  Therefore Lemma \ref{L;Lemma3.1} implies that $(j\setminus i)\cup(h\cap i\cap j)=(j\setminus g)\cup(g\cap h\cap j)$. Notice that $(j\setminus i)\cap h\cap i\cap j=0$ and $(j\setminus g)\cap g\cap h\cap j=0$. The desired lemma thus follows from Lemma \ref{L;Lemma3.11}.
\end{proof}
\begin{lem}\label{L;Lemma5.2}
Assume that $g, h, i, j\!\!\in\!\! [0,d]$. Then $m(g, h, i, i\cap j, i)=m(g, g\cap j, g, h, i)$.
\end{lem}
\begin{proof}
As $m(g, h, i, i\cap j, i)=m(g, g\cap j, g, h, i)=(g\oplus i)\cup((\widetilde{g\cap i})\cap h)\cup((\widetilde{g\cap i})\cap j)$ by a direct computation, the desired lemma thus follows from the above computation.
\end{proof}
\begin{lem}\label{L;Lemma5.3}
Assume that $g, h ,i, j\!\in\![0,d]$ and $j\!\leq_2\!\widetilde{d}$. Then $B_{g, g\cap j, g}$ and $B_{i, i\cap j, i}$ are defined. Moreover, if $p_{gh}^i\!\neq\! 0$, then $$\overline{k_{j\setminus i}}B_{g,h,i}B_{i, i\cap j, i}=\overline{k_{j\setminus g}}B_{g, g\cap j,g}B_{g, h, i}.$$
\end{lem}
\begin{proof}
Pick $k\in [0,d]$. As $j\!\leq_2\!\widetilde{d}$, notice that $j\cap k\leq_2\widetilde{k}$. Notice that the combination of Lemmas \ref{L;Lemma2.1}, \ref{L;Lemma4.3}, and \ref{L;Lemma4.4} implies that $p_{k\ell}^k\neq 0$ if and only if $\ell\leq_2\widetilde{k}$. In particular, notice that $p_{k(j\cap k)}^k\neq 0$. As $k$ is chosen from $[0,d]$ arbitrarily, notice that $p_{g(g\cap j)}^g\neq0$ and $p_{i(i\cap j)}^i\!\neq\! 0$. So $B_{g, g\cap j, g}$ and $B_{i, i\cap j, i}$ are defined. The first statement thus follows. As $p_{gh}^i\neq 0$, Lemmas \ref{L;Lemma2.1} and \ref{L;Lemma4.3} show that $g\oplus i\leq_2 h$. As the first statement holds,
\begin{align*}
\overline{k_{j\setminus i}}B_{g,h,i}B_{i, i\cap j, i}=&\overline{k_{j\setminus i}}\overline{k_{h\cap i\cap j}}B_{g,m(g, h,i,i\cap j, i), i}\\
=&\overline{k_{j\setminus g}k_{g\cap h\cap j}}B_{g,m(g, g\cap j, g, h, i), i}=\overline{k_{j\setminus g}}B_{g, g\cap j,g}B_{g, h, i}
\end{align*}
by combining Lemmas \ref{L;Lemma4.20}, \ref{L;Lemma5.1}, and \ref{L;Lemma5.2}. The desired lemma thus follows.
\end{proof}
Lemma \ref{L;Lemma5.3} motivates us to introduce the following notation and another lemma.
\begin{nota}\label{N;Notation5.4}
\em Assume that $g, h\in [0,d]$ and $g\cup h\leq_2\widetilde{d}$. Then $\sum_{i=0}^d \overline{k_{g\setminus i}}B_{i, g\cap i, i}$ is defined by Lemma \ref{L;Lemma5.3}. Denote this sum by $C_g$. As $g\setminus d=0$, notice that $C_g\neq O$ by Theorem \ref{T;Theorem4.22}. \eqref{Eq;2} and Theorem \ref{T;Theorem4.22} imply that $C_g=C_h$ if and only if $g=h$.
\end{nota}
\begin{lem}\label{L;Lemma5.5}
$\mathrm{Z}(\mathbb{T})$ has an $\F$-linearly independent subset $\{C_a: a\leq_2\widetilde{d}\}$.
\end{lem}
\begin{proof}
Set $\mathbb{U}\!=\!\{C_a: a\leq_2\widetilde{d}\}$. Let $g, h, i\in [0,d]$, $p_{gh}^i\!\neq\! 0$, and $M\in \mathbb{U}$. Notice that $B_{g, h, i}M=MB_{g, h, i}$ by \eqref{Eq;2} and Lemma \ref{L;Lemma5.3}. As $M$ is chosen from $\mathbb{U}$ arbitrarily and Theorem \ref{T;Theorem4.22} holds, observe that $\mathbb{U}\subseteq \mathrm{Z}(\mathbb{T})$. Let  $\sum_{N\in \mathbb{U}} c_NN=O$ and $c_N\in\F$ for any $N\in \mathbb{U}$. Notice that $\sum_{N\in \mathbb{U}} c_NE_d^*NE_d^*=O$ by \eqref{Eq;2}. Hence Notation \ref{N;Notation5.4} and Theorem \ref{T;Theorem4.22} imply that $c_N\!=\!\overline{0}$ for any $N\in\mathbb{U}$. The desired lemma thus follows.
\end{proof}
The main theorem of this section is proved by the following notation and lemmas.
\begin{nota}\label{N;Notation5.6}
\em Define $\mathbb{B}=\{B_{a, b, c}:a\oplus b\leq_2 c\leq_2 a\odot b\}$. Hence $\mathbb{T}$ has an $\F$-basis $\mathbb{B}$ by Theorem \ref{T;Theorem4.22}. Observe that $\{B_{a, b, a}: b\leq_2\widetilde{a}\}=\{B_{a, b, a}: p_{ab}^a\neq 0\}\subseteq\mathbb{B}$ by combining Lemmas \ref{L;Lemma4.4}, \ref{L;Lemma2.1}, \ref{L;Lemma4.3}. Assume that $M\in \mathbb{T}$. Then $M$ must be a unique $\F$-linear combination of the elements in $\mathbb{B}$. For any $g, h, i\in [0,d]$ and $p_{gh}^i\neq 0$, let $c_{g, h, i}(M)$ be the coefficient of $B_{g, h, i}$ in this $\F$-linear combination of $M$. Let $\mathrm{Supp}_\mathbb{B}(M)$ denote $\{B_{a,b, c}: B_{a,b,c}\in \mathbb{B},\ c_{a,b,c}(M)\neq \overline{0}\}$. Notice that $M=O$ if and only if $\mathrm{Supp}_\mathbb{B}(M)=\varnothing$.
\end{nota}
\begin{lem}\label{L;Lemma5.7}
Assume that $M\in\mathrm{Z}(\mathbb{T})$. Then $\mathrm{Supp}_\mathbb{B}(M)\subseteq\{B_{a, b, a}:b\leq_2\widetilde{a}\}$.
\end{lem}
\begin{proof}
Notice that $\{B_{a, b, a}:b\leq_2\widetilde{a}\}$ is a defined nonempty subset of $\mathbb{B}$. As $M\in \mathrm{Z}(\mathbb{T})$, Theorem \ref{T;Theorem4.22} thus implies that $M$ is an $\F$-linear combination of the elements in $\mathbb{B}$. Assume that $\mathrm{Supp}_\mathbb{B}(M)\not\subseteq\{B_{a, b, a}: b\leq_2\widetilde{a}\}$. By \eqref{Eq;3} and \eqref{Eq;2}, there are distinct $g, h\in [0,d]$ such that $E_g^*ME_h^*\!\neq\! O$. As $E_g^*, E_h^*\!\in\!\mathbb{T}$ and $M\!\in\! \mathrm{Z}(\mathbb{T})$, \eqref{Eq;2} implies that $O\!\neq\! E_g^*ME_h^*\!=\!ME_g^*E_h^*\!=\!O$. This is absurd. The desired lemma thus follows.
\end{proof}
\begin{lem}\label{L;Lemma5.8}
Assume that $g\in [0,d]$ and $g\leq_2 \widetilde{d}$. If $M\in \mathrm{Z}(\mathbb{T})$, $c_{d, g, d}(M)=\overline{1}$, and $\mathrm{Supp}_\mathbb{B}(E_d^*M)=\{B_{d, g, d}\}$, then $M=C_g$.
\end{lem}
\begin{proof}
As $g\leq_2 \widetilde{d}\leq_2 d$, $B_{d, g, d}$ and $c_{d, g, d}(M)$ are defined by Lemma \ref{L;Lemma5.3} and Notation \ref{N;Notation5.6}. Pick $h\!\in\! [0,d]$. Notice that $(d\oplus h)\cap d\cap g\!=\!((d\setminus h)\cup (h\setminus d))\cap d\cap g\!=\!g\setminus h$ and $m(h, d\oplus h, d, g, d)\!=\!(d\oplus h)\cup(g\cap h)$ by a direct computation. By Lemmas \ref{L;Lemma2.1} and \ref{L;Lemma4.3}, notice that $B_{h, d\oplus h, d}$ is defined. As $c_{d, g, d}(M)\!=\!\overline{1}$ and $\mathrm{Supp}_\mathbb{B}(E_d^*M)\!\!=\!\!\{B_{d, g, d}\}$, the combination of \eqref{Eq;2}, Lemmas \ref{L;Lemma5.7}, and \ref{L;Lemma4.20} implies that
\begin{align}\label{Eq;5.1}
B_{h, d\oplus h, d}M=c_{d, g, d}(M)B_{h, d\oplus h, d}B_{d, g,d}=\overline{k_{g\setminus h}}B_{h,(d\oplus h)\cup(g\cap h), d}.
\end{align}

For any $i\in [0,d]$ and $i\leq_2 \widetilde{h}$, notice that $i\cap h\cap (d\oplus h)=i\cap h\cap((d\setminus h)\cup(h\setminus d))=0$ and $m(h, i, h, d\oplus h, d)\!\!=\!\!(d\oplus h)\!\cup\! i$ by a direct computation. As $h\in [0,d]$, notice that $i\!\leq_2\! d\cap h$ for any $i\!\in\! [0,d]$ and $i\!\leq_2\! \widetilde{h}$. By combining \eqref{Eq;2}, Lemmas \ref{L;Lemma5.7}, and \ref{L;Lemma4.20},
\begin{align}\label{Eq;5.2}
MB_{h, d\oplus h, d}=\sum_{i\leq_2 \widetilde{h}}c_{h,i,h}(M)B_{h,i,h}B_{h, d\oplus h, d}=\sum_{i\leq_2 \widetilde{h}}c_{h,i,h}(M)B_{h,(d\oplus h)\cup i ,d}.
\end{align}

For any $i\in [0,d]$ and $i\leq_2 \widetilde{h}$, notice that $g\cap h\leq_2 d\cap h$ and $i\leq_2 d\cap h$. Therefore $(d\oplus h)\cap(g\cap h)=(d\oplus h)\cap i=0$ for any $i\in [0,d]$ and $i\leq_2 \widetilde{h}$. For any $i\in [0,d]$ and $i\leq_2 \widetilde{h}$, Lemma \ref{L;Lemma3.1} implies that $B_{h,(d\oplus h)\cup(g\cap h), d}=B_{h,(d\oplus h)\cup i ,d}$ if and only if $i=g\cap h$. As $M\in\mathrm{Z}(\mathbb{T})$ and $B_{h, d\oplus h, d}\in \mathbb{T}$ by Theorem \ref{T;Theorem4.22}, \eqref{Eq;5.1} and \eqref{Eq;5.2} imply that $c_{h, g\cap h, h}(M)=\overline{k_{g\setminus h}}$ and $c_{h, i, h}(M)=\overline{0}$ for any $i\in [0,d]$, $i\leq_2 \widetilde{h}$, and $i\neq g\cap h$. As $h$ is chosen from $[0,d]$ arbitrarily, the desired lemma thus follows.
\end{proof}
\begin{thm}\label{T;Center}
$\mathrm{Z}(\mathbb{T})$ has an $\F$-basis $\{C_a: a\leq_2\widetilde{d}\}$.
\end{thm}
\begin{proof}
For any $M\!\in\!\mathrm{Z}(\mathbb{T})$, Notation \ref{N;Notation5.4} and Lemma \ref{L;Lemma5.5} imply that there exists an $\F$-linear combination $N$ of the elements in $\{C_a: a\leq_2\widetilde{d}\}$ such that $M\!+\!N\!\in\!\mathrm{Z}(\mathbb{T})$, $c_{d, g, d}(M+N)=\overline{1}$, and $\mathrm{Supp}_\mathbb{B}(E_d^*(M+N))=\{B_{d,g,d}\}$ for some $g\in[0,d]$ and $g\leq_2 \widetilde{d}$. The desired theorem thus follows from Lemmas \ref{L;Lemma5.8} and \ref{L;Lemma5.5}.
\end{proof}
Theorem \ref{T;Center} motivates us to introduce two corollaries that may be interesting.
\begin{cor}\label{C;Corollary5.10}
The $\F$\!-\!dimension of $\mathrm{Z}(\mathbb{T})$ equals $2^{n_2}$. In particular, the $\F$-dimension of $\mathrm{Z}(\mathbb{T})$ is independent of the choice of $\F$.
\end{cor}
\begin{proof}
Recall that $d\!=\!2^n-1$ and $n_2\!=\!|\{a: u_a>2\}|\!=\!|\mathbb{P}_2(d)|=|\mathbb{P}(\widetilde{d})|$. Hence the first statement follows from combining Theorem \ref{T;Center}, Notation \ref{N;Notation5.4}, and Lemma \ref{L;Lemma3.1}. The second statement is from the first one. The desired lemma thus follows.
\end{proof}
\begin{cor}\label{C;Corollary5.11}
The reciprocal of the $\F$-dimension of $\mathrm{Z}(\mathbb{T})$ is equal to the classical probability of choosing a strongly normal closed subset of
$\mathbb{S}$ from the set of all closed subsets of $\mathbb{S}$ randomly.
\end{cor}
\begin{proof}
The desired corollary can be proved by Theorem \ref{T;Theorem3.22} and Corollary \ref{C;Corollary5.10}.
\end{proof}
We next investigate the structure constants of the $\F$-basis $\{C_a: a\leq_2\widetilde{d}\}$ in $\mathrm{Z}(\mathbb{T})$.
\begin{lem}\label{L;Lemma5.12}
Assume that $g, h, i\in [0,d]$. Then $k_{g\setminus i}k_{h\setminus i}k_{g\cap h\cap i}=k_{(g\cup h)\setminus i}k_{g\cap h}$.
\end{lem}
\begin{proof}
As $g\setminus i\!=\!(g\setminus (h\cup i))\cup ((g\cap h)\setminus i)$ and $(g\setminus (h\cup i))\cap ((g\cap h)\setminus i)=0$, notice that $k_{g\setminus i}=k_{g\setminus (h\cup i)}k_{(g\cap h)\setminus i}$ by Lemma \ref{L;Lemma3.11}. As $h\setminus i=(h\setminus(g\cup i))\cup((g\cap h)\setminus i)$ and $(h\setminus(g\cup i))\cap((g\cap h)\setminus i)=0$,
notice that $k_{h\setminus i}=k_{h\setminus(g\cup i)}k_{(g\cap h)\setminus i}$ by Lemma \ref{L;Lemma3.11}. As $g\cap h=(g\cap h\cap i)\cup((g\cap h)\setminus i)$ and $g\cap h\cap i\cap ((g\cap h)\setminus i)=0$, $k_{g\cap h}=k_{g\cap h\cap i}k_{(g\cap h)\setminus i}$ by Lemma \ref{L;Lemma3.11}. Notice that $(g\cup h)\setminus i=(g\setminus(h\cup i))\cup(h\setminus(g\cup i))\cup ((g\cap h)\setminus i)$. Notice that $\mathbb{P}(g\setminus(h\cup i))$, $\mathbb{P}(h\setminus(g\cup i))$, and $\mathbb{P}((g\cap h)\setminus i)$ are pairwise disjoint. So Lemma \ref{L;Lemma3.11} implies that $k_{(g\cup h)\setminus i}\!=\!k_{g\setminus(h\cup i)}k_{h\setminus(g\cup i)}k_{(g\cap h)\setminus i}$. The desired lemma thus follows from the equality $k_{g\setminus i}k_{h\setminus i}k_{g\cap h\cap i}=k_{g\setminus (h\cup i)}k_{(g\cap h)\setminus i}k_{h\setminus(g\cup i)}k_{(g\cap h)\setminus i}k_{g\cap h\cap i}=k_{(g\cup h)\setminus i}k_{g\cap h}$.
\end{proof}
\begin{lem}\label{L;Lemma5.13}
Assume that $g, h\!\in\! [0,d]$ and $g\cup h\leq_2\widetilde{d}$. Then $C_gC_h=\overline{k_{g\cap h}}C_{g\cup h}$.
\end{lem}
\begin{proof}
As $g\cup h\leq_2\widetilde{d}$, notice that $g\!\leq_2\! \widetilde{d}$, $h\!\leq_2\! \widetilde{d}$, and $C_g$, $C_h$, $C_{g\cup h}$ are defined. Pick $i\in [0,d]$. Then $g\cap i=g\cap \widetilde{i}$, $h\cap i=h\cap \widetilde{i}$, $m(i, g\cap i, i, h\cap i, i)=(g\cup h)\cap i$ by a direct computation. By combining \eqref{Eq;2}, Lemmas \ref{L;Lemma4.20}, and \ref{L;Lemma5.12}, observe that
\begin{align*}
&\overline{k_{g\setminus i}}B_{i,g\cap i, i}\overline{k_{h\setminus i}}B_{i,h\cap i, i}=\overline{k_{g\setminus i}}\overline{k_{h\setminus i}}\overline{k_{g\cap h\cap i}}B_{i,m(i, g\cap i, i, h\cap i, i), i}=\overline{k_{(g\cup h)\setminus i}k_{g\cap h}}B_{i,(g\cup h)\cap i ,i}\ \text{and}\\
&\overline{k_{g\setminus i}}B_{i,g\cap i, i}\overline{k_{h\setminus j}}B_{j,h\cap j, j}=\overline{k_{g\setminus i}}\overline{k_{h\setminus j}}B_{i,g\cap i, i}E_i^*E_j^*B_{j,h\cap j, j}=O\ \text{for any $j\in [0,d]\setminus\{i\}$}.
\end{align*}
As $i$ is chosen from $[0,d]$ arbitrarily, the desired lemma thus follows from combining \eqref{Eq;2}, Notation \ref{N;Notation5.4}, and the above computation.
\end{proof}
We conclude this section by giving an example of Theorem \ref{T;Center} and Lemma \ref{L;Lemma5.13}.
\begin{eg}\label{E;Example5.14}
\em Assume that $n=u_1=2$ and $u_2=3$. Notice that $d=3$ and $\widetilde{d}=2$. Hence $n_2=1$ and the $\F$-dimension of $\mathrm{Z}(\mathbb{T})$ is two by Corollary \ref{C;Corollary5.10}.
Theorem \ref{T;Center} implies that $\{C_0, C_2\}$ is an $\F$-basis of $\mathrm{Z}(\mathbb{T})$. By Notation \ref{N;Notation5.4} and \eqref{Eq;3}, notice that
\begin{align*}
C_0=B_{0,0,0}+B_{1,0,1}+B_{2,0,2}+B_{3,0,3}=I\ \text{and}\ C_2=\overline{2}B_{0,0,0}+\overline{2}B_{1,0,1}+B_{2,2,2}+B_{3,2,3}.
\end{align*}
Notice that $C_0C_0=C_0$, $C_0C_2=C_2C_0=C_2$, and $C_2C_2=\overline{2}C_2$ by Lemma \ref{L;Lemma5.13}.
\end{eg}
\section{Semisimplicity of Terwilliger $\F$-algebras of factorial schemes}
In this section, we determine the semisimplicity of $\mathbb{T}$. For our purpose, we recall Notations \ref{N;Notation3.3}, \ref{N;Notation4.2}, \ref{N;Notation4.8}, \ref{N;Notation4.19}, \ref{N;Notation5.4}, \ref{N;Notation5.6}. By Lemma \ref{L;Lemma2.10}, we recall that the subalgebras $E_0^*\mathbb{T}E_0^*, E_1^*\mathbb{T}E_1^*,\ldots, E_d^*\mathbb{T}E_d^*$ of $\mathbb{T}$ are commutative. We first list two needed lemmas.
\begin{lem}\label{L;Lemma6.1}
Assume that $g, h, i\in [0,d]$ and $h\cup i\leq_2 \widetilde{g}$. Then the subalgebra $E_g^*\mathbb{T}E_g^*$ of $\mathbb{T}$ has an $\F$-basis $\{B_{g, a, g}: a\leq_2 \widetilde{g}\}$. Furthermore, $B_{g, h, g}B_{g, i, g}=\overline{k_{h\cap i}}B_{g, h\cup i, g}$.
\end{lem}
\begin{proof}
By combining Theorem \ref{T;Theorem4.22}, \eqref{Eq;2}, Lemmas \ref{L;Lemma4.4}, \ref{L;Lemma2.1}, \ref{L;Lemma4.3}, and Notation \ref{N;Notation5.6}, the subalgebra $E_g^*\mathbb{T}E_g^*$ of $\mathbb{T}$ has an $\F$-basis $\{B_{g, a, g}: a\leq_2 \widetilde{g}\}$. The first statement is proved. As $h\cup i\leq_2 \widetilde{g}$, notice that $m(g, h, g, i, g)\!=\!h\cup i$ by a direct computation. As $h\cap i\leq_2 g$, the second statement thus can be proved by \eqref{Eq;2} and Lemma \ref{L;Lemma4.20}. The desired lemma thus follows.
\end{proof}
\begin{lem}\label{L;Lemma6.2}
Assume that $g\in [0,d]$. Then $\langle\{B_{g, a, g}: a\leq_2 \widetilde{g},\ p\mid k_a\}\rangle_\F$ is a two-sided ideal of the subalgebra $E_g^*\mathbb{T}E_g^*$ of $\mathbb{T}$.
\end{lem}
\begin{proof}
As $E_g^*\mathbb{T}E_g^*$ is commutative and Lemma \ref{L;Lemma6.1} holds, it is enough to check that $B_{g, h, g}B_{g, i, g}\in\langle\{B_{g, a, g}: a\leq_2\widetilde{g},\ p\mid k_a\}\rangle_\F$ for any $h, i\!\in\! [0,d]$, $h\!\leq_2\! \widetilde{g}$, $i\leq_2\widetilde{g}$, and $p\mid k_i$. For any $h, i\in [0,d]$, $h\leq_2\widetilde{g}$, $i\leq_2\widetilde{g}$, and $p\mid k_i$, Lemma \ref{L;Lemma3.11} implies that $p\mid k_{h\cup i}$ as $p\mid k_i$. For any $h, i\in [0,d]$, $h\leq_2\widetilde{g}$, $i\leq_2\widetilde{g}$, and $p\mid k_i$, \eqref{Eq;2} and Lemma \ref{L;Lemma6.1} thus imply that $B_{g, h, g}B_{g, i, g}\!=\!\overline{k_{h\cap i}}B_{g, h\cup i, g}\!\in\! \langle\{B_{g, a, g}: a\leq_2 \widetilde{g},\ p\mid k_a\}\rangle_\F$. The desired lemma thus follows from the above discussion.
\end{proof}
Lemma \ref{L;Lemma6.2} motivates us to introduce the following notation and another lemma.
\begin{nota}\label{N;Notation6.3}
\em Assume that $g\in [0,d]$. Set $\mathbb{I}_g\!=\!\langle\{B_{g, a, g}: a\leq_2 \widetilde{g},\ p\mid k_a\}\rangle_\F$. Hence Lemma \ref{L;Lemma6.2} implies that $\mathbb{I}_g$ is a two-sided ideal of the subalgebra $E_g^*\mathbb{T}E_g^*$ of $\mathbb{T}$.
\end{nota}
\begin{lem}\label{L;Lemma6.4}
Assume that $g\in[0,d]$. Then $\mathbb{I}_g$ is a nilpotent two-sided ideal of the subalgebra $E_g^*\mathbb{T}E_g^*$ of $\mathbb{T}$. Furthermore, $n(\mathbb{I}_g)=|\{a: a\in \mathbb{P}(g),\ u_a\!\equiv\! 1\pmod p\}|+1$.
\end{lem}
\begin{proof}
Set $\mathbb{U}=\{a: a\!\in\! \mathbb{P}(g),\ u_a\!\equiv\! 1\pmod p\}$ and $h\!=\!|\mathbb{U}|+1$. By Notation \ref{N;Notation6.3} and the definition of a nilpotent two-sided ideal of $\mathbb{T}$, it suffices to check that $n(\mathbb{I}_g)=h$. By Notation \ref{N;Notation6.3} and Lemma \ref{L;Lemma3.11}, notice that $\mathbb{U}=\varnothing$ if and only if $\mathbb{I}_g$ is the zero space. Moreover, $\mathbb{I}_g$ is the zero space if and only if $n(\mathbb{I}_g)=1$. So there is no loss to assume that $\mathbb{U}\!\neq\!\varnothing$. Pick $B_{g, i_1, g}, B_{g, i_2, g},\ldots, B_{g, i_h, g}\!\in\! \mathbb{I}_g$. Notation \ref{N;Notation6.3} and Lemma \ref{L;Lemma3.11} imply that $\mathbb{P}(i_j)\cap\mathbb{U}\neq\varnothing$ for any $j\in [1,h]$. By the Pigeonhole Principle, there exist $k, \ell\in [1,h]$ such that $\mathbb{P}(i_k)\cap \mathbb{P}(i_\ell)\cap \mathbb{U}\neq \varnothing$. Hence $p\mid k_{i_k\cap i_\ell}$ by Lemma \ref{L;Lemma3.11}. As $E_g^*\mathbb{T}E_g^*$ is commutative and $B_{g, i_k, g}B_{g, i_\ell, g}=O$ by Lemma \ref{L;Lemma6.1}, it is obvious that $\prod_{j=1}^hB_{g, i_j, g}=O$ by \eqref{Eq;2}. Hence $n(\mathbb{I}_g)\leq h$ as $B_{g, i_1, g}, B_{g, i_2, g},\ldots, B_{g, i_h, g}$ are chosen from $\mathbb{I}_g$ arbitrarily and Notation \ref{N;Notation6.3} holds.
Assume that $B_{g, \ell_1, g}, B_{g, \ell_2, g}, \ldots, B_{g, \ell_{h-1}, g}$ are pairwise distinct elements in $\mathbb{I}_g$, where $\mathbb{P}(\ell_m)\subseteq\mathbb{U}$, $|\mathbb{P}(\ell_m)|=1$, $\mathbb{P}(\ell_m)\neq\mathbb{P}(\ell_q)$ for any $m, q\in [1, h-1]$ and $m\neq q$. Observe that $\prod_{m=1}^{h-1}B_{g, \ell_m, g}=B_{g,\ell_1\cup\ell_2\cup\cdots\cup \ell_{h-1}, g}\neq O$ by combining \eqref{Eq;2}, the choices of $\ell_1, \ell_2,\ldots, \ell_{h-1}$, and Lemma \ref{L;Lemma6.1}. So $n(\mathbb{I}_g)>h-1$, which implies that $h-1<n(\mathbb{I}_g)\leq h$. The desired lemma thus follows.
\end{proof}
For further discussion, the next notation and combinatorial lemmas are required.
\begin{nota}\label{N;Notation6.5}
\em Assume that $g, h, i, j\in [0,d]$, $p_{gh}^i\neq 0$, $p\nmid k_h$, and $h\leq_2 j$. Use $n_{h, j}$ to denote $|\{a:a\in \mathbb{P}(j)\setminus\mathbb{P}(h),\ u_a\not\equiv1\pmod p\}|$. For any $k\in [0, n_{h, j}]$, use $\mathbb{U}_{h, j, k}$ to denote $\{a: h\!\leq_2\! a\!\leq_2\! j,\ p\nmid k_a,\ |\mathbb{P}(a)|\!-\!|\mathbb{P}(h)|\!=\!k\}$. For example, if $p=n=u_1\!=\!2$ and $u_2=3$, notice that $d=3$, $k_0=k_1=1$, and $k_2\!=\!k_3\!=\!2$ by Lemma \ref{L;Lemma3.11}. Furthermore, notice that $n_{0,3}\!=\!1$, $\mathbb{U}_{0,3,0}\!=\!\{0\}$, and $\mathbb{U}_{0,3,1}\!=\!\{1\}$. For any distinct $k, \ell\in [0, n_{h,j}]$, $\mathbb{U}_{h, j, k}\neq\varnothing=\mathbb{U}_{h, j, k}\cap \mathbb{U}_{h, j, \ell}$ by Lemma \ref{L;Lemma3.11}. So $\sum_{k=0}^{n_{h, g\odot i}}\sum_{m\in \mathbb{U}_{h, g\odot i, k}}(\overline{-1})^k\overline{k_{i\cap m}}^{-1}B_{g, m, i}$ is defined
by combining Lemmas \ref{L;Lemma3.11}, \ref{L;Lemma4.3}, \ref{L;Lemma4.4}, \ref{L;Lemma2.1}, Notation \ref{N;Notation4.19}. This listed sum is denoted by $D_{g, h, i}$. So $D_{g, h, i}\!\neq\! O$ by Theorem \ref{T;Theorem4.22}. If $q, r\in [0,d]$, $p\nmid k_r$, $r\leq_2 \widetilde{q}$, Lemmas \ref{L;Lemma4.4} and \ref{L;Lemma2.1} thus imply that $p_{qr}^q\neq 0$, $D_{q, r, q}$ is defined, and $D_{q,r,q}\!\in\! E_q^*\mathbb{T}E_q^*$.
\end{nota}
\begin{lem}\label{L;Lemma6.6}
Assume that $g, h, i, j, k, \ell, m\in [0,d]$, $q=\widetilde{g\cap i}$, $p_{gh}^i\!\neq\! 0$, and $h\leq_2 j$. Assume that $((h\cap k)\cup (j\setminus k))\cap q\leq_2 m\leq_2 ((h\cap k)\cup (j\setminus k)\cup (k\setminus (h\cap k)))\cap q$. Then $h\leq_2 (g\oplus i)\cup m\leq_2 g\odot i$ and $m(\ell, k, g, j, i)=m(\ell, k, g, (g\oplus i)\cup m, i)$. Moreover, if $p\nmid k_hk_jk_k$, then $p\nmid k_{(g\oplus i)\cup m}$.
\end{lem}
\begin{proof}
As $p_{gh}^i\neq 0$, Lemmas \ref{L;Lemma2.1} and \ref{L;Lemma4.3} imply that $h=(g\oplus i)\cup(h\cap q)$. It is obvious that $h=(h\cap k)\cup(h\setminus k)$. As $h\leq_2 j$, notice that   $h\cap q\leq_2((h\cap k)\cup (j\setminus k))\cap q\leq_2 m$ and $h\leq_2 (g\oplus i)\cup m\leq_2 g\odot i$ by a direct computation. As $(h\cap k)\cup(j\setminus k)\cup k=j\cup k$ and $(h\cap k)\cup(j\setminus k)\cup(k\setminus(h\cap k))=j\cup k$, notice that $(j\cup k)\cap q=(m\cup k)\cap q$ by Lemma \ref{L;Lemma3.1}. Therefore $(j\cup k)\cap q\cap \ell=(m\cup k)\cap q\cap\ell$. As $q\cap \ell=g\cap(\widetilde{\ell\cap i})$, it is clear that $m(\ell, k, g, j, i)=m(\ell, k, g, (g\oplus i)\cup m, i)$ by a direct computation. For the remaining statement, Lemma \ref{L;Lemma3.11} implies that $p\nmid k_{g\oplus i}$ as $p\nmid k_h$. As $p\nmid k_jk_k$, Lemma \ref{L;Lemma3.11} also implies that $p\nmid k_m$ and $p\nmid k_{(g\oplus i)\cup m}$. The desired lemma thus follows.
\end{proof}
\begin{lem}\label{L;Lemma6.7}
Assume that $g, h, i,j, k, \ell, m\!\in\! [0,d]$ and $q\!=\!\widetilde{g\cap i}$. Assume that $g\setminus \ell\!\!\leq_2\!\! k$. Then $(j\cup k)\cap q=(m\cup k)\cap q$ if and only if $(j\cup k)\cap q\cap\ell=(m\cup k)\cap q\cap \ell$. Moreover, assume that $h\cap q\leq_2 m\leq_2 q$. Then $m(\ell, k, g, j, i)=m(\ell, k, g, (g\oplus i)\cup m, i)$ only if $((h\cap k)\cup (j\setminus k))\cap q\leq_2 m\leq_2 ((h\cap k)\cup (j\setminus k)\cup (k\setminus (h\cap k)))\cap q$.
\end{lem}
\begin{proof}
As $g\setminus \ell\leq_2k$, notice that $q\setminus \ell=(q\cap k)\setminus \ell$. It is obvious to notice that $(j\cup k)\cap (q\setminus\ell)=q\setminus \ell=(m\cup k)\cap (q\setminus\ell)$. It implies that $(j\cup k)\cap q=(m\cup k)\cap q$ if and only if $(j\cup k)\cap q\cap\ell=(m\cup k)\cap q\cap \ell$. The first statement is proved. By a direct computation and the first statement, $m(\ell, k, g, j, i)=m(\ell, k, g, (g\oplus i)\cup m, i)$ shows that $(j\cup k)\cap q=(m\cup k)\cap q$. As $((h\cap k)\cup (j\setminus k)\cup k)\cap q=(m\cup k)\cap q$ and $h\cap q\leq_2 m\leq_2 q$, notice that $((h\cap k)\cup (j\setminus k))\cap q\leq_2 m$. It is obvious to notice that $j\cup k=(h\cap k)\cup (j\setminus k)\cup (k\setminus (h\cap k))$. The desired lemma thus follows.
\end{proof}
\begin{lem}\label{L;Lemma6.8}
Assume that $g, h, i, j,k, \ell\in [0,d]$, $g\oplus i\leq_2 h$, $k\oplus i\leq_2 j$, and $p\nmid k_hk_j$. Then $p\nmid k_{m(g, h, i, j, k)}$. Moreover, if $h\leq_2 \ell$ and $m\in [0, n_{h,\ell}]$, then $|\mathbb{U}_{h, \ell, m}|={n_{h,\ell}\choose m}$.
\end{lem}
\begin{proof}
As $g\oplus i\leq_2 h$ and $k\oplus i\leq_2 j$, $(g\setminus i)\cup(i\setminus g)\leq_2 h$ and $(k\setminus i)\cup(i\setminus k)\leq_2 j$. Notice that $(\widetilde{g\cap k})\setminus i\leq_2 k\setminus i\leq_2 j$ and $(h\cup j)\cap(\widetilde{g\cap k})\cap i\leq_2 h\cup j$. Notice that $g\setminus k=(g\setminus(i\cup k))\cup((g\cap i)\setminus k)\!\leq_2\!(g\setminus i)\cup(i\setminus k)\!\leq_2\! h\cup j$. Moreover, notice that $k\setminus g\!=\!(k\setminus(g\cup i))\cup((k\cap i)\setminus g)\leq_2 (k\setminus i)\cup(i\setminus g)\leq_2 h\cup j$. As $p\nmid k_hk_j$, the first statement is thus from Notation \ref{N;Notation4.8} and Lemma \ref{L;Lemma3.11}. The second statement is from combining Notation \ref{N;Notation6.5}, Lemmas \ref{L;Lemma3.11}, \ref{L;Lemma3.1}. The desired lemma thus follows.
\end{proof}
\begin{lem}\label{L;Lemma6.9}
Assume that $g, h, i, j, k, \ell\in [0,d]$, $m=(h\cap k)\cup(j\setminus k)$, $q=\widetilde{g\cap i}$, $r\!=\!(g\oplus i)\cup ((j\cup k)\cap q)$, and $s\!=\!|\mathbb{P}((k\cap q)\setminus(h\cap k\cap q))|$. Assume that $p_{gh}^i\!\neq\! 0$, $p\nmid k_hk_jk_k$, and $h\!\leq_2\! j\!\leq_2\! g\odot i$. Then $m\!\in\!\mathbb{U}_{h, g\odot i, n_{h, m}}$, $n_{m,r}$ is defined, $n_{m,r}\!=\!s$, and $\mathbb{U}_{m, r, 0}, \mathbb{U}_{m, r, 1},\ldots, \mathbb{U}_{m, r, s}$ are nonempty. Moreover, if $g\setminus \ell\leq_2 k$, the disjoint union of $\mathbb{U}_{m, r, 0}, \mathbb{U}_{m, r, 1},\ldots, \mathbb{U}_{m, r, s}$ is $\{a: h\leq_2 a\leq_2 g\odot i, \ m(\ell, k, g, j, i)=m(\ell, k, g, a, i)\}$.
\end{lem}
\begin{proof}
As $h\leq_2 j\leq_2 g\odot i$, it is clear that $h=(h\cap k)\cup(h\setminus k)\leq_2 m\leq_2g\odot i$. As $p\nmid k_hk_j$ and $m\leq_2 h\cup j$, Lemma \ref{L;Lemma3.11} and Notation \ref{N;Notation6.5} thus imply that $p\nmid k_m$ and $|\mathbb{P}(m)|-|\mathbb{P}(h)|=n_{h,m}$. Hence $m\in \mathbb{U}_{h, g\odot i, n_{h, m}}$ by Notation \ref{N;Notation6.5}. As $p_{gh}^i\!\neq\!0$, notice that $g\oplus i\leq_2 h$ by Lemmas \ref{L;Lemma2.1} and \ref{L;Lemma4.3}. So $r\leq_2 h\cup j\cup k$. Hence $p\nmid k_r$ as $p\nmid k_hk_jk_k$ and Lemma \ref{L;Lemma3.11} holds. As $g\oplus i\leq_2 h\leq_2 m\leq_2g\odot i$, Lemma \ref{L;Lemma4.4} thus implies that $p_{gm}^i\neq 0$. Therefore $m\!=\!(g\oplus i)\cup (m\cap q)$ by Lemmas \ref{L;Lemma2.1} and \ref{L;Lemma4.3}. Notice that $m\leq_2 r$ as $(j\cup k)\cap q\!=\!(m\cap q)\cup ((k\cap q)\setminus (h\cap k\cap q))$. As $p\nmid k_r$ and $m\leq_2 r$, Lemma \ref{L;Lemma3.11} and Notation \ref{N;Notation6.5} imply that $n_{m,r}=|\mathbb{P}(r)|-|\mathbb{P}(m)|=s$. So $\mathbb{U}_{m, r, 0}, \mathbb{U}_{m, r, 1},\ldots, \mathbb{U}_{m, r, s}$
are nonempty and pairwise disjoint by Notation \ref{N;Notation6.5}. The first statement thus follows.

Set $\mathbb{U}=\bigcup_{t=0}^s\mathbb{U}_{m,r,t}$ and $\mathbb{V}=\{a: h\leq_2 a\leq_2 g\odot i, \ m(\ell, k, g, j, i)=m(\ell, k, g, a, i)\}$. As $h\leq_2 m\leq_2 r\leq_2 g\odot i$, Lemma \ref{L;Lemma6.6} thus implies that $\mathbb{U}\subseteq\mathbb{V}$. Pick $u\in \mathbb{V}$. Therefore $g\oplus i\leq_2 h\leq_2 u\leq_2 g\odot i$. Hence $u=(g\oplus i)\cup (u\cap q)$. As $m(\ell, k, g, j, i)=m(\ell, k, g, u, i)$, notice that $m\cap q\leq_2 u\cap q\leq_2 (j\cup k)\cap q$ by Lemma \ref{L;Lemma6.7}. Therefore $m\leq_2 u\leq_2 r$. As $p\nmid k_r$, Lemma \ref{L;Lemma3.11} thus implies that $p\nmid k_u$. Hence $u\in \mathbb{U}$ by Notation \ref{N;Notation6.5}. As $u$ is chosen from $\mathbb{V}$ arbitrarily, notice that $\mathbb{V}\subseteq \mathbb{U}$. The desired lemma thus follows.
\end{proof}
\begin{lem}\label{L;Lemma6.10}
Assume that $g, h, i, j, k, \ell\in [0,d]$, $m=(h\cap k)\cup(j\setminus k)$, $q=\widetilde{g\cap i}$, $r\!=\!(g\oplus i)\!\cup\!((j\cup k)\cap q)$, and $s\!=\!|\mathbb{P}((k\cap q)\setminus(h\cap k\cap q))|$. Assume that $p_{gh}^i\!\!\neq\!\! 0$, $p\nmid k_hk_jk_k$, and $h\!\leq_2\! j\!\leq_2\! g\odot i$. If $\ell\in [0,s]$, then $\mathbb{U}_{m, r, \ell}$ is defined and $\mathbb{U}_{m, r, \ell}\subseteq\mathbb{U}_{h, g\odot i, \ell+n_{h, m}}$.
\end{lem}
\begin{proof}
Notice that $m\in\mathbb{U}_{h, g\odot i, n_{h, m}}$ and $n_{m,r}=s$ by Lemma \ref{L;Lemma6.9}. As $\ell\in [0,s]$, $\mathbb{U}_{m,r,\ell}$ is defined and nonempty by Notation \ref{N;Notation6.5}. Moreover, notice that $h\leq_2 m\leq_2 r\leq_2 g\odot i$ and $|\mathbb{P}(m)|\!-\!|\mathbb{P}(h)|\!\!=\!\!n_{h,m}$ by Notation \ref{N;Notation6.5}. Pick $t\!\in\!\mathbb{U}_{m, r, \ell}$. Hence $p\nmid k_t$, $m\leq_2 t\leq_2 r$, and $|\mathbb{P}(t)|-|\mathbb{P}(m)|=\ell$ by Notation \ref{N;Notation6.5}. Therefore $|\mathbb{P}(t)|-|\mathbb{P}(h)|=\ell+n_{h,m}$. The desired lemma thus follows as $t$ is chosen from $\mathbb{U}_{m, r, \ell}$ arbitrarily.
\end{proof}
\begin{lem}\label{L;Lemma6.11}
Assume that $g, h, i, j, k, \ell, m\!\in\! [0,d]$, $q\!\!=\!\!\widetilde{g\cap i}$, and $g\setminus \ell\leq_2 k$. Assume that $h\cap q\leq_2 m\leq_2q$ and $m(\ell, k, g, j, i)=m(\ell, k, g, (g\oplus i)\cup m, i)$. Then $$\frac{k_{k\cap g\cap ((g\oplus i)\cup m)}}{k_{i\cap((g\oplus i)\cup m)}}=\frac{k_{(g\cap k)\setminus i}}{k_{i\setminus g}k_{(j\setminus k)\cap q}}.$$
\end{lem}
\begin{proof}
As $m\leq_2 q\leq_2 i$, $i\cap((g\oplus i)\cup m)\!=\!(i\setminus g)\cup m$ by a direct computation. As $(i\setminus g)\cap m=0$, Lemma \ref{L;Lemma3.11} thus implies that $k_{i\cap((g\oplus i)\cup m)}=k_{i\setminus g}k_m$. By Lemma \ref{L;Lemma6.7}, notice that there exists $r\!\in\! [0,d]$ such that $m=(((h\cap k)\cup (j\setminus k))\cap q)\cup r$ and $r\!\leq_2\! (k\setminus (h\cap k))\cap q$. As $\mathbb{P}(h\cap k\cap q)$, $\mathbb{P}((j\setminus k)\cap q)$, and $\mathbb{P}(r)$ are pairwise disjoint, Lemma \ref{L;Lemma3.11} implies that $k_m=k_{h\cap k\cap q}k_{(j\setminus k)\cap q}k_r$. By a direct computation, notice that $k\cap g\cap((g\oplus i)\cup m)=((g\cap k)\setminus i)\cup(k\cap g\cap m)$. As $m=(((h\cap k)\cup (j\setminus k))\cap q)\cup r$, notice that $k\cap g\cap m=(h\cap k\cap q)\cup r$. It is clear that $\mathbb{P}((g\cap k)\setminus i)$, $\mathbb{P}(h\cap k\cap q)$, and $\mathbb{P}(r)$ are pairwise disjoint. Notice that $k_{k\cap g\cap ((g\oplus i)\cup m)}=k_{h\cap k\cap q}k_{(g\cap k)\setminus i}k_r$ by Lemma \ref{L;Lemma3.11}. The desired lemma thus follows from a direct computation.
\end{proof}
We are now ready to present some computational results of the elements in $\mathbb{T}$.
\begin{lem}\label{L;Lemma6.12}
Assume that $g, h, i, j, k, \ell\in [0,d]$, $m=(h\cap k)\cup(j\setminus k)$, $q=\widetilde{g\cap i}$, $r\!=\!(g\oplus i)\cup ((j\cup k)\cap q)$, and $s\!=\!|\mathbb{P}((k\cap q)\setminus(h\cap k\cap q))|\!>\!0$. If $p_{gh}^i\neq 0$, $p_{gj}^i\neq 0$, $p_{\ell k}^g\neq 0$, $p\nmid k_hk_k$, and $B_{g, j,i}\in\mathrm{Supp}_\mathbb{B}(D_{g, h, i})$, then $B_{\ell, m(\ell, k, g, j,i), i}\notin\mathrm{Supp}_\mathbb{B}(B_{\ell, k, g}D_{g, h, i})$. In particular, $B_{\ell, k, g}D_{g, h, i}=O$.
\end{lem}
\begin{proof}
As $p_{gj}^i\!\neq\! 0$, Lemmas \ref{L;Lemma2.1} and \ref{L;Lemma4.3} imply that $j\leq_2 g\odot i$. As $B_{g, j,i}\!\in\!\mathrm{Supp}_\mathbb{B}(D_{g, h, i})$, notice that $h\leq_2 j\leq_2 g\odot i$  and $p\nmid k_j$ by Notation \ref{N;Notation6.5}. As $p_{\ell k}^g\neq 0$, it is clear to see that $g\setminus \ell\leq_2 k$ by Lemmas \ref{L;Lemma2.1} and \ref{L;Lemma4.3}. According to Lemma \ref{L;Lemma6.9}, the disjoint union of $\mathbb{U}_{m, r, 0}, \mathbb{U}_{m, r, 1},\ldots, \mathbb{U}_{m, r, s}$ is $\{a: h\leq_2 a\leq_2 g\odot i, \ m(\ell, k, g, j, i)=m(\ell, k, g, a, i)\}$. As $p_{gh}^i\neq 0$, Lemmas \ref{L;Lemma2.1} and \ref{L;Lemma4.3} imply that $g\oplus i\leq_2 h$. So $t=(g\oplus i)\cup (t\cap q)$ for any $t\in \{a: h\leq_2 a\leq_2 g\odot i, \ m(\ell, k, g, j, i)=m(\ell, k, g, a, i)\}$. As $s>0$, the combination of Lemmas \ref{L;Lemma6.9}, \ref{L;Lemma6.8}, \ref{L;Lemma6.11}, and the Newton's Binomial Theorem thus implies that
\begin{align}\label{Eq;6.1}
\sum_{u=0}^s\sum_{t\in \mathbb{U}_{m,r, u}}(-1)^{u+n_{h, m}}\frac{k_{k\cap g\cap t}}{k_{i\cap t}}=(-1)^{n_{h,m}}\frac{k_{(g\cap k)\setminus i}}{k_{i\setminus g}k_{(j\setminus k)\cap q}}\sum_{u=0}^s(-1)^u{s\choose u}=0.
\end{align}

As $\{a: h\leq_2 a\leq_2 g\odot i, \ m(\ell, k, g, j, i)=m(\ell, k, g, a, i)\}$ is known to be a disjoint union of $\mathbb{U}_{m, r, 0}, \mathbb{U}_{m, r, 1},\ldots, \mathbb{U}_{m, r, s}$, the combination of Notation \ref{N;Notation6.5}, Lemmas \ref{L;Lemma4.20}, \ref{L;Lemma6.9}, \ref{L;Lemma6.10}, and Theorem \ref{T;Theorem4.22} thus implies that $c_{\ell, m(\ell, k, g, j,i),i}(B_{\ell, k, g}D_{g, h, i})$ is equal to
$$\sum_{u=0}^s\sum_{t\in \mathbb{U}_{m,r,u}}(\overline{-1})^{u+n_{h, m}}\overline{k_{i\cap t}}^{-1}\overline{k_{k\cap g\cap t}}.$$

Therefore $c_{\ell, m(\ell, k, g, j,i),i}(B_{\ell, k, g}D_{g, h, i})\!=\!\overline{0}$ by \eqref{Eq;6.1}. The first statement thus follows. As $B_{g, j,i}$ is chosen from $\mathrm{Supp}_\mathbb{B}(D_{g, h, i})$ arbitrarily, the second statement thus follows from Lemma \ref{L;Lemma4.20} and the first one. The desired lemma thus follows.
\end{proof}
\begin{lem}\label{L;Lemma6.13}
Assume that $g, h, i\!\in\! [0,d]$ and $p\nmid k_hk_i$. Assume that $h\cup i\leq_2\widetilde{g}$. Then
\[ B_{g, i, g}D_{g, h, g}=D_{g, h, g}B_{g, i, g}=\begin{cases} \overline{k_i}D_{g, h, g}, & \text{if}\ i\leq_2 h,\\
O, & \text{otherwise}.
\end{cases}\]
\end{lem}
\begin{proof}
As $E_g^*\mathbb{T}E_g^*$ is commutative, it is clear that $B_{g, i, g}D_{g, h, g}\!=\!D_{g, h, g}B_{g, i, g}$. If $i\leq_2 h$, then $B_{g, i, g}D_{g, h, g}\!=\!D_{g, h, g}B_{g, i, g}=\overline{k_i}D_{g, h, g}$ by Notation \ref{N;Notation6.5} and Lemma \ref{L;Lemma6.1}. Otherwise, the assumption $h\cup i\leq_2\widetilde{g}$ implies that $p_{gh}^g\neq 0$ and $p_{gi}^g\neq 0$ by Lemmas \ref{L;Lemma4.4} and \ref{L;Lemma2.1}. Moreover, notice that $|\mathbb{P}(i\setminus (h\cap i))|\!>\!0$. Therefore $B_{g, i, g}D_{g, h, g}=D_{g, h, g}B_{g, i, g}=O$
by Lemma \ref{L;Lemma6.12}. The desired lemma thus follows from the above discussion.
\end{proof}
\begin{lem}\label{L;Lemma6.14}
Assume that $g, h, i\in [0,d]$ and $p\nmid k_hk_i$. Assume that $h\cup i\leq_2\widetilde{g}$. Then $D_{g, i, g}D_{g, h, g}=D_{g, h, g}D_{g, i, g}=\delta_{hi}D_{g, h, g}$.
\end{lem}
\begin{proof}
As $E_g^*\mathbb{T}E_g^*$ is commutative, it is clear that $D_{g, i, g}D_{g, h, g}\!\!=\!\!D_{g, h, g}D_{g, i, g}$. According to
Notation \ref{N;Notation6.5} and Lemma \ref{L;Lemma6.13}, notice that $D_{g, i, g}D_{g, h, g}\!\!=\!\!D_{g, h, g}D_{g, i, g}\!\neq\! O$ only if $i\leq_2 h\leq_2 i$. Lemma \ref{L;Lemma3.1} implies that $h=i$. As $D_{g, h, g}D_{g, h, g}\!=\!D_{g, h, g}$ by Notation \ref{N;Notation6.5} and Lemma \ref{L;Lemma6.13}, the desired lemma thus follows from the above discussion.
\end{proof}
For further discussion, the following notation and another lemma are necessary.
\begin{nota}\label{N;Notation6.15}
\em Assume that $g, h, i, j\!\in\! \mathbb{N}_0\setminus\{0\}$. Let $g\mathrm{M}_h(\F)$ be the direct sum of $g$ copies of $\mathrm{M}_h(\F)$. So $g\mathrm{M}_h(\F)\!\cong\! i\mathrm{M}_j(\F)$ as algebras if and only if $g\!\!=\!\!i$ and $h\!=\!j$.
\end{nota}
\begin{lem}\label{L;Lemma6.16}
Assume that $g\in [0,d]$. Then $E_g^*\mathbb{T}E_g^*/\mathbb{I}_g\!\cong\! 2^{n_{0,\widetilde{g}}}\mathrm{M}_1(\F)$ as algebras and $\mathrm{Rad}(E_g^*\mathbb{T}E_g^*)=\mathbb{I}_g$. The subalgebra $E_g^*\mathbb{T}E_g^*$ of $\mathbb{T}$ is semisimple if and only if $p\nmid k_g$.
\end{lem}
\begin{proof}
Set $\mathbb{U}=\{D_{g, a, g}+\mathbb{I}_g: a\leq_2 \widetilde{g},\ p\nmid k_a\}$ by Notation \ref{N;Notation6.5}. Notice that $O\!+\!\mathbb{I}_g\notin \mathbb{U}$ by Notation \ref{N;Notation6.5} and Theorem \ref{T;Theorem4.22}. So Lemma \ref{L;Lemma6.14} implies that $\mathbb{U}$ is an $\F$-linearly independent subset of the subalgebra $E_g^*\mathbb{T}E_g^*$ of $\mathbb{T}$. So $|\mathbb{U}|\!\!=\!\!2^{n_{0,\widetilde{g}}}$ by Notation \ref{N;Notation6.5} and Lemma \ref{L;Lemma3.11}. By combining Lemmas \ref{L;Lemma6.1}, \ref{L;Lemma3.11}, Notations \ref{N;Notation6.3}, \ref{N;Notation6.5}, Theorem \ref{T;Theorem4.22}, the $\F$-dimension of $E_g^*\mathbb{T}E_g^*/\mathbb{I}_g$ is $2^{n_{0,\widetilde{g}}}$. So $\mathbb{U}$ is an $\F$-basis of $E_g^*\mathbb{T}E_g^*/\mathbb{I}_g$. So first two statements are from Lemmas \ref{L;Lemma6.14} and \ref{L;Lemma6.4}. As $\mathrm{Rad}(E_g^*\mathbb{T}E_g^*)\!=\!\mathbb{I}_g$, Lemma \ref{L;Lemma3.11} and Notation \ref{N;Notation6.3} imply that $\mathbb{I}_g\!\!=\!\!\{O\}$ if and only if $p\nmid k_g$. The desired lemma follows.
\end{proof}
We are now ready to present the main result of this section and an example.
\begin{thm}\label{T;Semisimplicity}
$\mathbb{T}$ is semisimple if and only if $p\nmid k_g$ for any $g\in [0,d]$. In particular, $\mathbb{T}$ is semisimple if and only if its subalgebra $E_g^*\mathbb{T}E_g^*$ is semisimple for any $g\in [0,d]$.
\end{thm}
\begin{proof}
Assume that $p\nmid k_g$ for any $g\in [0,d]$. For any $g\in [0,d]$, Lemma \ref{L;Lemma6.16} implies that $\mathrm{Rad}(E_g^*\mathbb{T}E_g^*)\!=\!\{O\}$. Pick $M\!\in\! \mathrm{Rad}(\mathbb{T})$. Assume that $M\!\neq\! O$. Then Lemma \ref{L;Lemma2.6} implies that $E_g^*ME_g^*\!=\!O$ for any $g\in [0,d]$. As $M\neq O$ and \eqref{Eq;3} holds, there must be distinct $h, i\in [0,d]$ such that $E_h^*ME_i^*\!\!\neq\! \!O$. Then there are $j\!\in\! \mathbb{N}_0\setminus \{0\}$ and pairwise distinct $\ell_1, \ell_2, \ldots, \ell_j\!\in\!\! [0,d]$ such that $\mathrm{Supp}_\mathbb{B}(E_h^*ME_i^*)=\{B_{h, \ell_1, i}, B_{h,\ell_2, i},\ldots, B_{h,\ell_j, i}\}$. So $p_{h\ell_k}^i\!\!\neq\!\! 0$ for any $k\in [1, j]$. Hence $\ell_1\cap h\cap \widetilde{i}, \ell_2\cap h\cap \widetilde{i}, \ldots, \ell_j\cap h\cap \widetilde{i}$ are pairwise distinct by Lemmas \ref{L;Lemma2.1} and \ref{L;Lemma4.3}. For any $k\!\in\! [1, j]$, $m(i, h\oplus i, h, \ell_k, i)\!=\!(\widetilde{i}\setminus h)\cup(\ell_k\cap h\cap \widetilde{i})$ by computation. As $B_{i, h\oplus i, h}$ is defined by Lemmas \ref{L;Lemma4.4} and \ref{L;Lemma2.1}, Lemma \ref{L;Lemma4.20} implies that $c_{i, m(i, h\oplus i, h, \ell_1, i), i}(B_{i, h\oplus i, h}E_h^*ME_i^*)\!=\!c_{h, \ell_1, i}(E_h^*ME_i^*)\overline{k_{(h\oplus i)\cap h\cap\ell_1}}\!\!\neq\!\! \overline{0}$. Hence Theorem \ref{T;Theorem4.22} and Lemma \ref{L;Lemma2.6} imply that $O\neq B_{i, h\oplus i, h}E_h^*ME_i^*\in\mathrm{Rad}(E_i^*\mathbb{T}E_i^*)=\{O\}$. Hence $\mathrm{Rad}(\mathbb{T})\!=\!\{O\}$ by this contradiction. So $\mathbb{T}$ is semisimple.
The first statement thus follows from Lemma \ref{L;Lemma2.8}. The second statement thus follows from Lemma \ref{L;Lemma6.16} and the first one. The desired theorem thus follows.
\end{proof}
\begin{eg}\label{E;Example6.17}
\em Assume that $n\!=\!u_1\!=\!2$ and $u_2=3$. Hence $d=3$, $k_0=k_1=1$, and $k_2=k_3=2$ by Lemma \ref{L;Lemma3.11}. According to Theorem \ref{T;Semisimplicity}, notice that $\mathbb{T}$ is semisimple if and only if $p\neq 2$.
\end{eg}
We end this section with a corollary of Theorem \ref{T;Semisimplicity} that may be interesting.
\begin{cor}\label{L;Lemma6.17}
$\mathbb{T}$ is semisimple if and only if its subalgebra $E_d^*\mathbb{T}E_d^*$ is semisimple. In particular, $\mathbb{T}$ is semisimple if and only if $\mathrm{Z}(\mathbb{T})$ is semisimple.
\end{cor}
\begin{proof}
As $g\leq_2 d$ for any $g\in [0,d]$, Lemma \ref{L;Lemma3.11} thus implies that $p\nmid k_d$ if and only if $p\nmid k_g$ for any $g\in [0,d]$. The first statement thus follows from Lemma \ref{L;Lemma6.16} and Theorem \ref{T;Semisimplicity}. By Theorem \ref{T;Center} and Lemma \ref{L;Lemma6.1}, there is an $\F$-linear bijection from $\mathrm{Z}(\mathbb{T})$ to $E_d^*\mathbb{T}E_d^*$ that sends $C_g$ to $B_{d, g, d}$ for any $g\in [0,d]$ and $g\leq_2\widetilde{d}$. By Lemmas \ref{L;Lemma5.13} and \ref{L;Lemma6.1}, this $\F$-linear bijection is also an algebra isomorphism. The second statement thus follows from the first one. The desired corollary thus follows.
\end{proof}
\section{Jacobson radicals of Terwilliger $\F$-algebras of factorial schemes}
In this section, we determine $\mathrm{Rad}(\mathbb{T})$ and compute the nilpotency of $\mathrm{Rad}(\mathbb{T})$. For this aim, we recall Notations \ref{N;Notation3.3}, \ref{N;Notation4.2}, \ref{N;Notation4.8}, \ref{N;Notation4.19}, \ref{N;Notation5.6}, \ref{N;Notation6.3} and list two needed lemmas.
\begin{lem}\label{L;Lemma7.1}
Assume that $g, h, i\in [0,d]$ and $B_{g, h, i}\in\mathbb{B}$. Then $B_{g, h, i}^T\!=\!B_{i, h, g}\in\mathbb{B}$.
\end{lem}
\begin{proof}
As $B_{g, h, i}\!\in\!\mathbb{B}$, notice that $p_{gh}^i\!\!\neq\!\! 0$ by Notation \ref{N;Notation4.19}. So $p_{ih}^g\!\!\neq\!\! 0$ by Lemma \ref{L;Lemma2.1}. So $B_{i, h, g}$ is defined. So $B_{g, h, i}^T\!\!=\!\!B_{i, h, g}$ by Notation \ref{N;Notation4.19}. The desired lemma follows.
\end{proof}
\begin{lem}\label{L;Lemma7.2}
$\mathbb{T}$ has a two-sided ideal $\langle\{B_{a, b, c}: a\oplus b\leq_2 c\leq_2 a\odot b,\ p\mid k_b\}\rangle_\F$.
\end{lem}
\begin{proof}
The case $\langle\{B_{a, b, c}\!: a\oplus b\leq_2 c\leq_2 a\odot b,\ p\mid k_b\}\rangle_\F\!\!=\!\!\{O\}$ is trivial. For the other case, pick $g, h, i, j, k,\ell\!\in\! [0,d]$, $B_{g ,h, i}\!\!\in\!\! \langle\{B_{a, b, c}\!: a\oplus b\leq_2 c\leq_2 a\odot b,\ p\mid k_b\}\rangle_\F$, and $B_{j,k,\ell}\!\in\! \mathbb{B}$. Notice that $p_{j\ell}^k\!\neq\! 0$, $p_{gi}^h\!\!\neq\!\! 0$, and $p\mid k_h$ by Notation \ref{N;Notation4.19} and Lemma \ref{L;Lemma2.1}. If $g\!\neq\! \ell$, notice that $B_{j, k, \ell}B_{g, h, i}\!\!=\!\!B_{j, k, \ell}E_\ell^*E_g^*B_{g, h, i}\!\!=\!\!O$ by \eqref{Eq;2}. Assume that $g\!=\!\ell$. Then $p_{jg}^k\!\neq\! 0$ and $k\!=\!(g\setminus j)\cup(j\setminus g)\cup((\widetilde{g\cap j})\cap k)$ by Lemma \ref{L;Lemma4.3}. As $p_{gi}^h\!\neq\! 0$, Lemma \ref{L;Lemma4.3} implies that $h\!=\!(g\setminus i)\cup(i\setminus g)\cup((\widetilde{g\cap i})\cap h)$. By a direct computation, notice that $m(j, k, g, h, i)\!=\!(i\setminus j)\cup(j\setminus i)\!\cup\!((\widetilde{i\cap j})\setminus g)\cup((h\cup k)\cap(g\cap(\widetilde{i\cap j})))$. Lemma \ref{L;Lemma4.20} shows that $B_{j ,k ,g}B_{g, h, i}\!=\!\overline{k_{g\cap h\cap k}}B_{j, m(j, k, g, h, i), i}$. As $p\mid k_h$ and Lemma \ref{L;Lemma3.11} holds, there is $m\in\mathbb{P}(g\setminus i)\cup\mathbb{P}(i\setminus g)\cup\mathbb{P}((\widetilde{g\cap i})\cap h)$ such that $u_m\equiv 1\pmod p$. Assume that $m\in\mathbb{P}(g\setminus (i\cup j))$. Then $m\in \mathbb{P}(g\cap h\cap k)$ and $\overline{k_{g\cap h\cap k}}=\overline{0}$ by Lemma \ref{L;Lemma3.11}. Hence $B_{j ,k ,g}B_{g, h, i}=O$. Assume that $m\!\in\!\mathbb{P}((g\cap j)\setminus i)\cup\mathbb{P}((i\cap j)\setminus g)\cup\mathbb{P}(i\setminus(g\cup j))$. As $u_m\equiv 1\pmod p$, notice that $m\in \mathbb{P}((g\cap j)\setminus i)\cup\mathbb{P}((\widetilde{i\cap j})\setminus g)\cup\mathbb{P}(i\setminus(g\cup j))$. So $m\in \mathbb{P}(m(j, k, g, h, i))$ and $p\mid k_{m(j, k, g, h, i)}$ by Lemma \ref{L;Lemma3.11}. The two obtained facts thus imply that $B_{j, m(j, k, g, h, i), i}\in\langle\{B_{a, b, c}: a\oplus b\leq_2 c\leq_2 a\odot b,\ p\mid k_b\}\rangle_\F$. Assume that $m\in \mathbb{P}(((\widetilde{g\cap i})\cap h)\setminus j)\cup\mathbb{P}((\widetilde{g\cap i})\cap h\cap j)$. Then $m\in \mathbb{P}(i\setminus j)\cup\mathbb{P}(g\cap h\cap(\widetilde{i\cap j}))$. So $m\in\mathbb{P}(m(j, k, g, h, i))$ and $p\mid k_{m(j, k, g, h, i)}$ by Lemma \ref{L;Lemma3.11}. The two obtained facts thus imply that $B_{j, m(j, k, g, h, i), i}\in\langle\{B_{a, b, c}: a\oplus b\leq_2 c\leq_2 a\odot b,\ p\mid k_b\}\rangle_\F$. By the above discussion and Lemma \ref{L;Lemma7.1}, $\langle\{B_{a, b, c}: a\oplus b\leq_2 c\leq_2 a\odot b,\ p\mid k_b\}\rangle_\F$ thus contains $B_{j, k,\ell}B_{g, h, i}$ and $B_{g, h, i}B_{j, k, \ell}$. The desired lemma is thus from Theorem \ref{T;Theorem4.22} and the fact that $B_{g, h, i}$ is chosen from $\langle\{B_{a, b, c}: a\oplus b\leq_2 c\leq_2 a\odot b,\ p\mid k_b\}\rangle_\F$ arbitrarily.
\end{proof}
Lemma \ref{L;Lemma7.2} motivates us to introduce the following notation and another lemma.
\begin{nota}\label{N;Notation7.3}
\em The $\F$-linear subspace $\langle\{B_{a, b, c}: a\oplus b\leq_2 c\leq_2 a\odot b,\ p\mid k_b\}\rangle_\F$ of $\mathbb{T}$ is denoted by $\mathbb{I}$. According to Lemma \ref{L;Lemma7.2}, observe that $\mathbb{I}$ is a two-sided ideal of $\mathbb{T}$.
\end{nota}
\begin{lem}\label{L;Lemma7.4}
Assume that $\mathbb{I}\neq\{O\}$. There are $2|\{a: u_a\equiv 1\pmod p\}|$ elements in $\mathbb{I}$ such that a product of all these elements is not the zero matrix.
\end{lem}
\begin{proof}
Set $\mathbb{U}\!=\!\{a: u_a\equiv 1\pmod p\}$ and $g=|\mathbb{U}|$. As $\mathbb{I}\neq \{O\}$, notice that $g>0$ by Notation \ref{N;Notation7.3} and Lemma \ref{L;Lemma3.11}. Then there are pairwise distinct $h_1, h_2,\ldots, h_g\in[0,d]$ such that $\mathbb{P}(h_i)\subseteq\mathbb{U}$ and $|\mathbb{P}(h_i)|=1$ for any $i\!\in\![1,g]$. By Lemma \ref{L;Lemma4.4},
$B_{d, h_i, d\oplus h_i}$ and $B_{d\oplus h_i, h_i, d}$ are defined for any $i\in [1,g]$. By Lemma \ref{L;Lemma3.11}, $B_{d, h_i, d\oplus h_i}, B_{d\oplus h_i, h_i, d}\!\in\!\mathbb{I}$ for any $i\in [1,g]$. By \eqref{Eq;2} and Lemma \ref{L;Lemma4.20}, notice that $\prod_{i=1}^gB_{d, h_i, d}\!\!=\!\!B_{d, h_1\cup h_2\cup\cdots\cup h_g, d}\!\neq\! O$ and $B_{d, h_i, d\oplus h_i}B_{d\oplus h_i, h_i, d}=B_{d, h_i, d}$ for any $i\!\!\in\!\![1,g]$. The desired lemma follows as the choices are $B_{d, h_1, d\oplus h_1}, B_{d\oplus h_1, h_1, d}, B_{d, h_2, d\oplus h_2}, B_{d\oplus h_2, h_2, d},\ldots, B_{d, h_g, d\oplus h_g}, B_{d\oplus h_g, h_g, d}$.
\end{proof}
For further discussion, the following lemmas are necessary for determining $\mathrm{Rad}(\mathbb{T})$.
\begin{lem}\label{L;Lemma7.5}
Assume that $g, h, i, j, k\in [0,d]$, $p_{gh}^i\!\neq\!\! 0$, and $p_{ij}^k\!\neq\! 0$. If $p\nmid k_{h\cap i\cap j}$, then $\{a: a\in \mathbb{P}(h\cup j),\ u_a\equiv 1\pmod p\}\subseteq\mathbb{P}(m(g, h, i, j, k))$.
\end{lem}
\begin{proof}
Recall that $m(g, h, i, j, k)\!=\!(g\setminus k)\cup (k\setminus g)\cup((\widetilde{g\cap k})\setminus i)\cup((h\cup j)\cap (\widetilde{g\cap k})\cap i)$. As $p_{gh}^i\!\neq\!\! 0$ and $p_{ij}^k\!\neq\! 0$, Lemmas \ref{L;Lemma2.1} and \ref{L;Lemma4.3} imply that $h=(g\setminus i)\cup (i\setminus g)\cup((\widetilde{g\cap i})\cap h)$ and $j=(k\setminus i)\cup (i\setminus k)\cup((\widetilde{k\cap i})\cap j)$. Pick $\ell\in \{a: a\in \mathbb{P}(h\cup j),\ u_a\equiv 1\pmod p\}$. Notice that $u_\ell>2$ as $u_\ell\equiv 1\pmod p$. Assume that $\ell\in \mathbb{P}(g\setminus i)\cup\mathbb{P}(k\setminus i)$. It is clear to see that $\ell\in\mathbb{P}((\widetilde{g\cap k})\setminus i)\cup\mathbb{P}(g\setminus(k\cup i))\cup(k\setminus(g\cup i))$. So $\ell\in\mathbb{P}(m(g, h, i, j, k))$. Assume that $\ell\in \mathbb{P}((\widetilde{g\cap i})\cap h)$. So $\ell\in \mathbb{P}(((\widetilde{g\cap i})\cap h)\setminus k)\cup\mathbb{P}((\widetilde{g\cap i})\cap h\cap k)$, which implies that $\ell\in\mathbb{P}(m(g, h, i, j, k))$. Assume that $\ell\in \mathbb{P}((\widetilde{k\cap i})\cap j)$. It is clear that $\ell\in \mathbb{P}(((\widetilde{k\cap i})\cap j)\setminus g)\cup\mathbb{P}((\widetilde{k\cap i})\cap j\cap g)$. Hence $\ell\in\mathbb{P}(m(g, h, i, j, k))$. As $p\nmid k_{h\cap i\cap j}$, notice that $\{a: a\in \mathbb{P}(h\cup j),\ u_a\equiv 1\pmod p\}\cap \mathbb{P}(h\cap i\cap j)=\varnothing$ by Lemma \ref{L;Lemma3.11}. Assume that $\ell\in\mathbb{P}(i\setminus g)\cup \mathbb{P}(i\setminus k)$. So $\ell\in\mathbb{P}((i\cap k)\setminus g)\cup \mathbb{P}(i\setminus (g\cup k))\cup\mathbb{P}((i\cap g)\setminus k)$. As $p_{gh}^i\!\neq\!\! 0$ and $p_{ij}^k\!\neq\! 0$, notice that $\mathbb{P}(i\setminus (g\cup k))\subseteq\mathbb{P}((h\cap i\cap j)\setminus(g\cup k))$ by Lemmas \ref{L;Lemma2.1} and \ref{L;Lemma4.3}. Hence $\ell\in \mathbb{P}((i\cap k)\setminus g)\cup \mathbb{P}((i\cap g)\setminus k)$. So $\ell\in \mathbb{P}(m(g, h, i, j, k))$. As $\ell$ is chosen from $\{a: a\in \mathbb{P}(h\cup j),\ u_a\equiv 1\pmod p\}$ arbitrarily, the desired containment is thus checked. The desired lemma thus follows.
\end{proof}
\begin{lem}\label{L;Lemma7.6}
Assume that $g, h, i, j, k,\ell, m\in [0,d]$ and $B_{g,h,i}, B_{i,j, k}, B_{k, \ell, m}\in \mathbb{B}$. If there is $q\!\in\! [1,n]$ such that $u_q\equiv 1\pmod p$ and $q\in \mathbb{P}(h\cap j\cap\ell)$, $B_{g,h,i}B_{i,j,k}B_{k,\ell, m}=O$.
\end{lem}
\begin{proof}
If $q\in \mathbb{P}(i)$, notice that $q\in \mathbb{P}(h\cap i\cap j)$ and $p\mid k_{h\cap i\cap j}$ by Lemma \ref{L;Lemma3.11}. Hence $B_{g,h,i}B_{i,j,k}=O$ by \eqref{Eq;2} and Lemma \ref{L;Lemma4.20}. If $q\in \mathbb{P}(k)$, notice that $q\in \mathbb{P}(j\cap k\cap \ell)$ and $p\mid k_{j\cap k\cap \ell}$ by Lemma \ref{L;Lemma3.11}. Hence $B_{i,j,k}B_{k, \ell, m}=O$ by \eqref{Eq;2} and Lemma \ref{L;Lemma4.20}. So there is no loss to assume that $q\notin \mathbb{P}(i\cup k)$. Since $p_{ij}^k\neq 0$ by Notation \ref{N;Notation4.19}, notice that $j\leq_2 i\cup k$ by Lemmas \ref{L;Lemma2.1} and \ref{L;Lemma4.3}. Hence $q\notin\mathbb{P}(j)$ as $q\notin \mathbb{P}(i\cup k)$. This is an contradiction as $q\!\in\!\mathbb{P}(h\cap j\cap\ell)$.
So $q\!\in\! \mathbb{P}(i\cup k)$. The desired lemma thus follows.
\end{proof}
For further discussion, the next notation and an additional lemma are required.
\begin{nota}\label{L;Lemma7.7}
\em Assume that $g\in \mathbb{N}_0\setminus\{0\}$ and $h_1, i_1, j_1, h_2, i_2, j_2, \ldots, h_g, i_g, j_g\in [0,d]$. Assume that $B_{h_k, i_k, j_k}\in \mathbb{B}$ for any $k\in [1,g]$.
If $h_k=j_\ell$ for any $k, \ell\in [1,g]$, recall that the subalgebra $E_{h_1}^*\mathbb{T}E_{h_1}^*$ of $\mathbb{T}$ is commutative and $\prod_{k=1}^gB_{h_k, i_k, j_k}$ is defined. For any $\ell\in [1,g]$, set $\prod_{k\!=\!\ell}^gB_{h_k, i_k, j_k}=B_{h_\ell, i_\ell, j_\ell}B_{h_{\ell+1}, i_{\ell+1}, j_{\ell+1}}\cdots B_{h_g, i_g, j_g}$ for the general case.
\end{nota}
\begin{lem}\label{L;Lemma7.8}
Assume that $g\in \mathbb{N}_0\setminus[0,2]$ and $h_1, i_1, j_1, h_2, i_2, j_2, \ldots, h_g, i_g, j_g\in [0,d]$. Assume that $B_{h_k, i_k, j_k}\!\in\! \mathbb{B}$ for any $k\!\in\! [1,g]$. If there are pairwise distinct $\ell, m, q\!\in\! [1,g]$ such that $\{a: u_a\equiv 1\pmod p\}\cap\mathbb{P}(i_\ell\cap i_m\cap i_q)\neq \varnothing$, then $\prod_{k=1}^gB_{h_k, i_k, j_k}=O$.
\end{lem}
\begin{proof}
Assume that $\prod_{k=1}^gB_{h_k, i_k, j_k}\neq O$. As $\ell, m, q$ are pairwise distinct, there is no loss to assume that $\ell<m<q$. Pick $r\in\{a: u_a\equiv 1\pmod p\}\cap\mathbb{P}(i_\ell\cap i_m\cap i_q)$. By combining \eqref{Eq;2}, Lemmas \ref{L;Lemma4.20}, and \ref{L;Lemma7.5}, there are $s, t\in \F\setminus\{\overline{0}\}$ and $u, v\in [0,d]$ such that $r\in \mathbb{P}(u\cap m\cap v)$, $B_{h_1, u, j_{m-1}}, B_{h_{m+1}, v, j_g}\in \mathbb{B}$, $\prod_{k=1}^{m-1}B_{h_k, i_k, j_k}\!\!=\!\!sB_{h_1, u, j_{m-1}}\neq O$, and $\prod_{k=m+1}^{g}B_{h_k, i_k, j_k}=tB_{h_{m+1}, v, j_g}\neq O$. So $\prod_{k=1}^gB_{h_k, i_k, j_k}=O$ by Lemma \ref{L;Lemma7.6}. It is an obvious contradiction. The desired lemma thus follows.
\end{proof}
The following definition and another combinatorial lemma complete preparation.
\begin{defn}\label{D;Definition7.9}
\em Assume that $g, h\!\in\! \mathbb{N}_0\setminus\{0\}$ and $\mathbb{U}$ denotes a nonempty set. Call $\mathbb{U}$ a $(g,h)$-dense set if, for any set sequence $\mathbb{V}_1, \mathbb{V}_2, \ldots, \mathbb{V}_g$ that satisfies the inequality $\mathbb{U}\cap \mathbb{V}_i\neq\varnothing$ for any $i\in [1, g]$, there is $\mathbb{W}\subseteq [1, g]$ such that $\mathbb{U}\cap (\bigcap_{j\in\mathbb{W}}\mathbb{V}_j)\neq\varnothing$ and $|\mathbb{W}|=h$. Notice that $\mathbb{U}$ is a $(g, h)$-dense set only if $g\geq h$.
\end{defn}
\begin{lem}\label{L;Lemma7.10}
Assume that $g\in\mathbb{N}_0\!\setminus\!\{0\}$. Then every set of cardinality $g$ is always a $(2g\!+\!1, 3)$-dense set.
\end{lem}
\begin{proof}
We work by induction on $g$. If $g=1$, it is obvious that every set containing exactly a single element is a $(3,3)$-dense set. The base case is thus checked. Assume that $g>1$ and every set of cardinality $g-1$ is a $(2g-1, 3)$-dense set. Assume that $\mathbb{U}$ is a set of cardinality $g$ and a set sequence $\mathbb{V}_1, \mathbb{V}_2, \ldots, \mathbb{V}_{2g+1}$ satisfies the inequality $\mathbb{U}\cap \mathbb{V}_h\!\!\neq\!\!\varnothing$ for any $h\!\in\! [1, 2g+1]$. As $|\mathbb{U}|\!=\!g$, the Pigeonhole Principle says that there are distinct $i, j\in [1, 2g+1]$ such that $\mathbb{U}\cap \mathbb{V}_i\cap\mathbb{V}_j\!\!\neq\!\!\varnothing$. Pick $k\!\in\!\mathbb{U}\cap \mathbb{V}_i\cap\mathbb{V}_j$. If there is $\ell\in [1,2g+1]\setminus\{i, j\}$ such that $k\in \mathbb{V}_\ell$, $\mathbb{U}\cap\mathbb{V}_i\cap\mathbb{V}_j\cap\mathbb{V}_\ell\!\neq\! \varnothing$. Otherwise, $\mathbb{U}\setminus\{k\}$ is a $(2g-1, 3)$-dense set by the inductive hypothesis. So $(\mathbb{U}\setminus\{k\})\cap \mathbb{V}_{m}\cap\mathbb{V}_{q}\cap\mathbb{V}_{r}\neq\varnothing$ for some pairwise distinct $m, q, r\in [1,2g+1]\setminus \{i, j\}$. So $\mathbb{U}$ is a $(2g+1, 3)$-dense set by the above discussion. The desired lemma follows as $\mathbb{U}$ is chosen arbitrarily.
\end{proof}
\begin{lem}\label{L;Lemma7.11}
The product of any $2|\{a: u_a\equiv 1\pmod p\}|\!+\!1$ elements in $\mathbb{I}$ is the zero matrix. Furthermore, $n(\mathbb{I})=2|\{a: u_a\equiv 1\pmod p\}|+1$ and $\mathbb{I}\subseteq\mathrm{Rad}(\mathbb{T})$.
\end{lem}
\begin{proof}
Set $\mathbb{U}\!=\!\{a: u_a\equiv 1\pmod p\}$ and $g=|\mathbb{U}|$. If $g=0$, then the combination of Theorem \ref{T;Semisimplicity}, Lemma \ref{L;Lemma3.11}, Notation \ref{N;Notation7.3} implies that $\mathrm{Rad}(\mathbb{T})\!\!=\!\!\{O\}=\mathbb{I}$. Assume that $g>0$. Set $h\!=\!2g+1$. Pick $B_{i_1, j_1, \ell_1}, B_{i_2, j_2, \ell_2},\ldots, B_{i_h, j_h, \ell_h}\!\in\! \mathbb{I}$. Lemma \ref{L;Lemma3.11} and Notation \ref{N;Notation7.3} imply that $\mathbb{U}\cap \mathbb{P}(j_m)\neq \varnothing$ for any $m\in [1,h]$. As $g\!>\!0$ and Lemma \ref{L;Lemma7.10} holds, there are pairwise distinct $j_q, j_r, j_s$ such that $\mathbb{U}\cap\mathbb{P}(j_q\cap j_r\cap j_s)\neq\varnothing$. Lemma \ref{L;Lemma7.8} thus implies that $\prod_{m=1}^hB_{i_m, j_m, \ell_m}\!\!=\!\!O$. The first statement is thus from Notation \ref{N;Notation7.3}. The desired formula of $n(\mathbb{I})$ thus follows from the first statement and Lemma \ref{L;Lemma7.4}. So $\mathbb{I}$ is a nilpotent two-sided ideal of $\mathbb{T}$. The desired lemma thus follows.
\end{proof}
We are now ready to close this section by presenting the main result of this section.
\begin{thm}\label{T;Jacobson}
Assume that $M\in \mathrm{Rad}(\mathbb{T})$. Then $M\in \mathbb{I}$. In particular, $\mathrm{Rad}(\mathbb{T})\!=\!\mathbb{I}$.
\end{thm}
\begin{proof}
Assume that $M\in\mathrm{Rad}(\mathbb{T})\setminus\mathbb{I}$. As $M\notin \mathbb{I}$ and \eqref{Eq;3} holds, the combination of Lemmas \ref{L;Lemma2.6}, \ref{L;Lemma6.16}, Notations \ref{N;Notation6.3}, \ref{N;Notation7.3} implies that there exist distinct $g, h\in [0,d]$ such that $E_g^*ME_h^*\!\notin\mathbb{I}$. As Lemma \ref{L;Lemma7.11} and Theorem \ref{T;Theorem4.22} hold, there is no loss to assume that $\mathrm{Supp}_\mathbb{B}(E_g^*ME_h^*)\!=\!\{B_{g, i_1, h}, B_{g, i_2, h},\ldots, B_{g, i_j, h}\}$, $i_1, i_2,\ldots, i_j$ are pairwise distinct in $[0,d]$, $p_{gi_k}^h\!\neq\! 0$, $p\nmid k_{i_k}$ for any $k\in [1,j]$. So $i_1\cap\widetilde{g}\cap h, i_2\cap\widetilde{g}\cap h, \ldots,i_j\cap\widetilde{g}\cap h$ are pairwise distinct by Lemmas \ref{L;Lemma2.1} and \ref{L;Lemma4.3}. As $m(h, g\oplus h, g, i_k, h)=(\widetilde{h}\setminus g)\cup (i_k\cap\widetilde{g}\cap h)$ for any $k\in [1,j]$ and $B_{h, g\oplus h, g}$ is defined by Lemmas \ref{L;Lemma4.4} and \ref{L;Lemma2.1}, Lemmas \ref{L;Lemma4.20} and \ref{L;Lemma3.11} imply that $c_{h, m(h,g\oplus h, g, i_1, h), h}(B_{h, g\oplus h, g}E_g^*ME_h^*)\!\!=\!\!c_{g, i_1, h}(E_g^*ME_h^*)\overline{k_{(g\oplus h)\cap g\cap i_1}}\!\neq\! \overline{0}$. So $O\neq B_{h, g\oplus h, g}E_g^*ME_h^*\in\mathrm{Rad}(E_h^*\mathbb{T}E_h^*)$ by Lemma \ref{L;Lemma2.6}. The combination of Lemmas \ref{L;Lemma6.16}, \ref{L;Lemma3.11}, Notation \ref{N;Notation6.3}, Theorem \ref{T;Theorem4.22} implies that $B_{h, g\oplus h, g}E_g^*ME_h^*\notin\mathrm{Rad}(E_h^*\mathbb{T}E_h^*)$. This is a contradiction. The desired theorem thus follows from Lemma \ref{L;Lemma7.11}.
\end{proof}
\section{Structures of Terwilliger $\F$-algebras of factorial schemes I}
In this section, we study the algebraic structure of $\mathbb{T}$ by investigating the objects in Notation \ref{N;Notation6.5}. Our aim is to generalize Lemmas \ref{L;Lemma6.13} and \ref{L;Lemma6.14}. For our purpose, we recall Notations \ref{N;Notation3.3}, \ref{N;Notation4.2}, \ref{N;Notation4.8}, \ref{N;Notation4.19}, \ref{N;Notation5.6}, \ref{N;Notation6.5} and first present two required lemmas.
\begin{lem}\label{L;Lemma8.1}
Assume that $g, h, i, j\in [0,d]$. Then $(\widetilde{g\cap j})\setminus(\widetilde{g\cap i})\leq_2 (\widetilde{g\cap j})\cap h$ if $g\setminus h\leq_2 i$.
\end{lem}
\begin{proof}
Notice that $(\widetilde{g\cap j})\setminus(\widetilde{g\cap i})=(((\widetilde{g\cap j})\cap h)\setminus(\widetilde{g\cap i}))\cup((\widetilde{g\cap j})\setminus((\widetilde{g\cap i})\cup h))$. As $g\setminus h\leq_2 i$, $(\widetilde{g\cap j})\setminus((\widetilde{g\cap i})\cup h)\!=\!((\widetilde{g\cap j})\cap i)\setminus((\widetilde{g\cap i})\cup h)\!=\!((\widetilde{g\cap i})\cap j)\setminus((\widetilde{g\cap i})\cup h)\!=\!0$. The desired lemma thus follows from the above discussion.
\end{proof}
\begin{lem}\label{L;Lemma8.2}
Assume that $g, h, i, j,k \in[0,d]$, $g\setminus h\leq_2 i$, $k=(g\oplus j)\cup ((\widetilde{g\cap j})\cap k)$. Then $(\widetilde{g\cap i})\cap k\!\leq_2\! h$ if and only if $(\widetilde{g\cap i})\setminus h\leq_2 j$ and $k\leq_2(g\oplus j)\cup ((\widetilde{g\cap j})\cap h)$.
\end{lem}
\begin{proof}
Assume that $(\widetilde{g\cap i})\cap k\!\leq_2\! h$. It is clear that $(\widetilde{g\cap i})\cap k\cap h=(\widetilde{g\cap i})\cap k$. As $(\widetilde{g\cap i})\setminus j=(((\widetilde{g\cap i})\cap h)\setminus j)\cup((\widetilde{g\cap i})\setminus(h\cup j))$ and $k=(g\oplus j)\cup ((\widetilde{g\cap j})\cap k)$, the equality $(\widetilde{g\cap i})\cap k\cap h=(\widetilde{g\cap i})\cap k$ thus implies that $(\widetilde{g\cap i})\setminus h\leq_2 j$. As $g\setminus h\leq_2 i$, notice that $((\widetilde{g\cap j})\cap k)\setminus (\widetilde{g\cap i})\leq_2 ((\widetilde{g\cap j})\cap h)$ by Lemma \ref{L;Lemma8.1}. As $(\widetilde{g\cap i})\cap k\!\leq_2\! h$, notice that $(\widetilde{g\cap j})\cap k\cap((\widetilde{g\cap i})\setminus h)\leq_2((\widetilde{g\cap i})\cap k)\setminus h=0$. This fact implies that $(\widetilde{g\cap j})\cap k\leq_2 (((\widetilde{g\cap j})\cap k)\setminus (\widetilde{g\cap i}))\cup((\widetilde{g\cap j})\cap k\cap(\widetilde{g\cap i})\cap h)\leq_2 (\widetilde{g\cap j})\cap h$. So $k\leq_2(g\oplus j)\cup ((\widetilde{g\cap j})\cap h)$. For the other direction, $k\leq_2(g\oplus j)\cup ((\widetilde{g\cap j})\cap h)$ shows that $(\widetilde{g\cap i})\cap k\leq_2 ((\widetilde{g\cap i})\setminus j)\cup h$ by a direct computation.
As $(\widetilde{g\cap i})\setminus h\leq_2 j$ and $(\widetilde{g\cap i})\setminus j\!=\!(((\widetilde{g\cap i})\cap h)\setminus j)\cup ((\widetilde{g\cap i})\setminus (h\cup j))$, the desired lemma thus follows.
\end{proof}
The following four lemmas continue to investigate the objects in Notation \ref{N;Notation6.5}.
\begin{lem}\label{L;Lemma8.3}
Assume that $g, h, i, j, k\in [0,d]$ and $\ell=(i\oplus j)\cup((\widetilde{i\cap j})\cap h)$. Assume that $p_{gh}^i\neq 0$, $p_{jk}^g\neq 0$, $p\nmid k_hk_k$, and $(\widetilde{g\cap i})\cap k\!\leq_2\! h$. Then $p\nmid k_\ell$ and $n_{h, g\odot i}, n_{\ell, i\odot j}$ are defined. Moreover, $n_{h, g\odot i}=n_{\ell, i\odot j}$.
\end{lem}
\begin{proof}
As $(\widetilde{g\cap i})\cap k\!\leq_2\! h$, notice that $(k\cup h)\cap g\cap(\widetilde{i\cap j})\!=\!h\cap g\cap(\widetilde{i\cap j})$. As $p_{gh}^i\neq 0$, $i\setminus g\leq_2 g\oplus i\leq_2 h\leq_2 g\odot i$ by Lemmas \ref{L;Lemma2.1} and \ref{L;Lemma4.3}. Hence $(\widetilde{i\cap j})\setminus g=((\widetilde{i\cap j})\cap h)\setminus g$. Hence $m(j,k,g,h,i)=(i\oplus j)\cup((\widetilde{i\cap j})\setminus g)\cup((k\cup h)\cap g\cap(\widetilde{i\cap j}))=\ell$ by a direct computation. As $p_{jk}^g\neq 0$, $k=(g\oplus j)\cup ((\widetilde{g\cap j})\cap k)$ by Lemmas \ref{L;Lemma2.1} and \ref{L;Lemma4.3}. Lemma \ref{L;Lemma6.8} thus implies that $p\nmid k_\ell$. Set $\mathbb{U}\!=\!\{a: a\in \mathbb{P}(g\odot i)\setminus \mathbb{P}(h),\ u_a\not\equiv 1\pmod p\}$. Put $\mathbb{V}\!=\!\{a: a\in \mathbb{P}(i\odot j)\setminus \mathbb{P}(\ell),\ u_a\not\equiv 1\pmod p\}$. As $p\nmid k_hk_\ell$, notice that both $n_{h, g\odot i}$ and $n_{\ell, i\odot j}$ are defined by Notation \ref{N;Notation6.5}. Moreover, notice that $|\mathbb{U}|=n_{h, g\odot i}$ and $|\mathbb{V}|=n_{\ell, i\odot j}$. As $p_{gh}^i\neq 0$, $h=(g\oplus i)\cup((\widetilde{g\cap i})\cap h)$, $g\setminus h\leq_2 i$, and $i\setminus h\leq_2 g$ by Lemmas \ref{L;Lemma2.1} and \ref{L;Lemma4.3}. If $m\in \mathbb{U}$, notice that $m\!\in\!\mathbb{P}((\widetilde{g\cap i})\setminus h)$. Therefore Lemma \ref{L;Lemma8.2} implies that $m\in\mathbb{P}((\widetilde{i\cap j})\setminus h)$ and $m\in \mathbb{V}$. If $m\!\in\! \mathbb{V}$, notice that $m\!\in\!\mathbb{P}((\widetilde{i\cap j})\setminus h)$. As $i\setminus h\leq_2 g$, Lemma \ref{L;Lemma8.2} also implies that $m\!\in\!\mathbb{P}((\widetilde{g\cap i})\setminus h)$ and $m\in\mathbb{U}$. So $\mathbb{U}=\mathbb{V}$, which implies that $n_{h, g\odot i}=|\mathbb{U}|=|\mathbb{V}|=n_{\ell, i\odot j}$. The desired lemma thus follows.
\end{proof}
The conclusion of Lemma \ref{L;Lemma8.3} allows us to formulate the remaining three lemmas.
\begin{lem}\label{L;Lemma8.4}
Assume that $g, h, i, j, k\in [0,d]$ and $\ell=(i\oplus j)\cup((\widetilde{i\cap j})\cap h)$. Assume that $p_{gh}^i\neq 0$, $p_{jk}^g\neq 0$, $p\nmid k_hk_k$, and $(\widetilde{g\cap i})\cap k\!\leq_2\! h$. Assume that $m\in [0, n_{h, g\odot i}]$ and $q\!\in\!\mathbb{U}_{h, g\odot i, m}$. Then $m(j, k, g, q, i)\!\!=\!\!(i\oplus j)\cup ((\widetilde{i\cap j})\cap q)$. Furthermore, if $r\in\mathbb{U}_{h, g\odot i, m}$ and $q\neq r$, then $(i\oplus j)\cup((\widetilde{i\cap j})\cap q)\neq (i\oplus j)\cup((\widetilde{i\cap j})\cap r)$.
\end{lem}
\begin{proof}
As $p_{gh}^i\neq 0$, Lemmas \ref{L;Lemma2.1} and \ref{L;Lemma4.3} thus imply that $h\!=\!(g\oplus i)\cup((\widetilde{g\cap i})\cap h)$. Notice that $\mathbb{U}_{h, g\odot i, m}$ is defined by Lemma \ref{L;Lemma8.3}. Notice that $g\oplus i\leq_2 h\leq_2 q\leq_2 g\odot i$ as $q\in \mathbb{U}_{h, g\odot i, m}$. In particular, notice that $i\setminus g\!\leq_2\! q$ and $(\widetilde{i\cap j})\setminus g=((\widetilde{i\cap j})\cap q)\setminus g$. As $(\widetilde{g\cap i})\cap k\!\leq_2\! h$, $(k\cup q)\cap g\cap(\widetilde{i\cap j})=(k\cup q)\cap(\widetilde{g\cap i})\cap j=g\cap(\widetilde{i\cap j})\cap q$. So $m(j, k, g, q, i)=(i\oplus j)\cup((\widetilde{i\cap j})\setminus g)\cup((k\cup q)\cap g\cap (\widetilde{i\cap j}))=(i\oplus j)\cup ((\widetilde{i\cap j})\cap q)$. The first statement thus follows. For the second statement, the inequality $p_{gh}^i\!\neq\! 0$ shows that $g\setminus h\leq_2 i$ and $i\setminus h\leq_2 g$ by Lemmas \ref{L;Lemma2.1} and \ref{L;Lemma4.3}. As $r\in \mathbb{U}_{h, g\odot i, m}$, notice that $g\oplus i\leq_2 h\leq_2 r\leq_2 g\odot i$. It is obvious to see that $q=(g\oplus i)\cup((\widetilde{g\cap i})\cap q)$ and $r\!=\!(g\oplus i)\cup((\widetilde{g\cap i})\cap r)$. As $q\neq r$, notice that $(\widetilde{g\cap i})\cap q\neq(\widetilde{g\cap i})\cap r$. Furthermore, notice that $(\widetilde{g\cap i})\cap h\cap q=(\widetilde{g\cap i})\cap h\cap r$ and $(\widetilde{i\cap j})\cap h\cap q=(\widetilde{i\cap j})\cap h\cap r$.
As $p_{jk}^g\neq 0$, notice that $k=(g\oplus j)\cup ((\widetilde{g\cap j})\cap k)$ by Lemmas \ref{L;Lemma2.1} and \ref{L;Lemma4.3}. By Lemma \ref{L;Lemma8.2} and the fact $i\setminus h\leq_2 g$, $((\widetilde{i\cap j})\cap q)\setminus h=((\widetilde{g\cap i})\cap q)\setminus h\neq((\widetilde{g\cap i})\cap r)\setminus h=((\widetilde{i\cap j})\cap r)\setminus h$. So $(i\oplus j)\cup((\widetilde{i\cap j})\cap q)\!\!\neq\!\!(i\oplus j)\cup((\widetilde{i\cap j})\cap r)$. The desired lemma thus follows.
\end{proof}
\begin{lem}\label{L;Lemma8.5}
Assume that $g, h, i, j, k\in [0,d]$ and $\ell=(i\oplus j)\cup((\widetilde{i\cap j})\cap h)$. Assume that $p_{gh}^i\neq 0$, $p_{jk}^g\neq 0$, $p\nmid k_hk_k$, and $(\widetilde{g\cap i})\cap k\!\leq_2\! h$. Assume that $m\in [0, n_{h, g\odot i}]$. Then $\{(i\oplus j)\cup ((\widetilde{i\cap j})\cap a): a\in\mathbb{U}_{h, g\odot i, m}\}\subseteq\mathbb{U}_{\ell, i\odot j, m}$.
\end{lem}
\begin{proof}
Set $\mathbb{U}\!\!=\!\!\{(i\oplus j)\cup ((\widetilde{i\cap j})\cap a): a\in\mathbb{U}_{h, g\odot i, m}\}$. As $p_{gh}^i\neq 0$, $g\setminus h\leq_2 i$, $i\setminus h\leq_2 g$, and $h=(g\oplus i)\cup((\widetilde{g\cap i})\cap h)$ by Lemmas \ref{L;Lemma2.1} and \ref{L;Lemma4.3}. Hence $\mathbb{U}\neq\varnothing$ by Lemma \ref{L;Lemma8.3}. Pick $q\!\in\!\mathbb{U}_{h, g\odot i, m}$. Set $r\!=\!(i\oplus j)\cup ((\widetilde{i\cap j})\cap q)\!\in\!\mathbb{U}$. Notice that $p\nmid k_\ell$ by Lemma \ref{L;Lemma8.3}. Hence $p\nmid k_{i\oplus j}$ by Lemma \ref{L;Lemma3.11}. As $q\in\mathbb{U}_{h, g\odot i, m}$, notice that $p\nmid k_r$ by Lemma \ref{L;Lemma3.11}. Moreover, $h\leq_2 q\leq_2 g\odot i$, $q=(g\oplus i)\cup((\widetilde{g\cap i})\cap q)$, and $|\mathbb{P}(q)|-|\mathbb{P}(h)|=m$. As $p_{jk}^g\neq 0$, notice that $k=(g\oplus j)\cup ((\widetilde{g\cap j})\cap k)$ by Lemmas \ref{L;Lemma2.1} and \ref{L;Lemma4.3}.
Lemma \ref{L;Lemma8.2} thus implies that $m=|\mathbb{P}(q)|-|\mathbb{P}(h)|=|\mathbb{P}((\widetilde{g\cap i})\cap(q\setminus h))|=|\mathbb{P}((\widetilde{i\cap j})\cap(q\setminus h))|$.
So $|\mathbb{P}(r)|\!-\!|\mathbb{P}(\ell)|\!\!=\!\!|\mathbb{P}((\widetilde{i\cap j})\cap(q\setminus h))|\!\!=\!\!m$. As $h\leq_2 q$, notice that $\ell\leq_2 r\leq_2 i\odot j$. Hence $r\in\mathbb{U}_{\ell, i\odot j, m}$. The desired lemma follows as $q$ is chosen from $\mathbb{U}_{h, g\odot i, m}$ arbitrarily.
\end{proof}
\begin{lem}\label{L;Lemma8.6}
Assume that $g, h, i, j, k\in [0,d]$ and $\ell=(i\oplus j)\cup((\widetilde{i\cap j})\cap h)$. Assume that $p_{gh}^i\neq 0$, $p_{jk}^g\neq 0$, $p\nmid k_hk_k$, and $(\widetilde{g\cap i})\cap k\!\leq_2\! h$. Assume that $m\in [0, n_{h, g\odot i}]$. Then $\{m(j,k,g, a, i): a\in\mathbb{U}_{h, g\odot i, m}\}\!=\!\{(i\oplus j)\cup ((\widetilde{i\cap j})\cap a): a\in\mathbb{U}_{h, g\odot i, m}\}\!=\!\mathbb{U}_{\ell, i\odot j, m}$.
\end{lem}
\begin{proof}
Set $\mathbb{U}=\{(i\oplus j)\cup ((\widetilde{i\cap j})\cap a): a\in\mathbb{U}_{h, g\odot i, m}\}$. As $p_{gh}^i\neq 0$, Lemmas \ref{L;Lemma2.1} and \ref{L;Lemma4.3} imply that $h=(g\oplus i)\cup((\widetilde{g\cap i})\cap h)$. Notice that $\mathbb{U}\neq\varnothing$ by Lemma \ref{L;Lemma8.3}. Lemma \ref{L;Lemma8.4} thus implies that $|\mathbb{U}|=|\mathbb{U}_{h, g\odot i, m}|$. Notice that $|\mathbb{U}_{h, g\odot i, m}|=|\mathbb{U}_{\ell, i\odot j, m}|$ by Lemmas \ref{L;Lemma6.8} and \ref{L;Lemma8.3}. The desired lemma thus follows from Lemmas \ref{L;Lemma8.4} and \ref{L;Lemma8.5}.
\end{proof}
The next lemmas give two equalities that relate to the valencies of elements in $\mathbb{S}$.
\begin{lem}\label{L;Lemma8.7}
Assume that $g, h, i, j, k, \ell\!\in\![0,d]$. Assume that $i\setminus g\leq_2 h$, $g\oplus j\leq_2 k$, $\ell\!=\!(g\oplus i)\cup((\widetilde{g\cap i})\cap \ell)$, and $(\widetilde{g\cap i})\cap k\leq_2 h\leq_2 \ell$. Then $$\frac{k_{k\cap g\cap \ell}}{k_{(\widetilde{g\cap i})\cap\ell}}=\frac{k_{(g\cap k)\setminus i}k_{h\cap k\cap (\widetilde{i\cap j}) }}{k_{(\widetilde{i\cap j})\cap \ell}}.$$
\end{lem}
\begin{proof}
As $\ell\!=\!(g\oplus i)\cup((\widetilde{g\cap i})\cap \ell)$, $k\cap g\cap \ell=((g\cap k)\setminus i)\cup((k\cap \ell)\cap (\widetilde{g\cap i}))$ by a direct computation. Lemma \ref{L;Lemma3.11} thus implies that $k_{k\cap g\cap \ell}=k_{(g\cap k)\setminus i}k_{k\cap \ell\cap (\widetilde{g\cap i})}$. As $g\setminus j\leq_2 g\oplus j\leq_2 k$ and $(\widetilde{g\cap i})\cap k\leq_2 h\leq_2 \ell$, $(k\cap\ell\cap(\widetilde{g\cap i}))\setminus j=(\ell\cap(\widetilde{g\cap i}))\setminus j$ and $k\cap\ell\cap(\widetilde{g\cap i})\cap j=h\cap k\cap (\widetilde{g\cap i})\cap j$. So $k_{k\cap\ell\cap(\widetilde{g\cap i})}=k_{(\ell\cap(\widetilde{g\cap i}))\setminus j}k_{h\cap k\cap (\widetilde{g\cap i})\cap j}$ by Lemma \ref{L;Lemma3.11}. Notice that $(\widetilde{g\cap i})\cap\ell=((\ell\cap(\widetilde{g\cap i}))\setminus j)\cup(g\cap\ell\cap(\widetilde{i\cap j}))$ by a direct computation. Lemma \ref{L;Lemma3.11} thus implies that $k_{(\widetilde{g\cap i})\cap\ell}=k_{(\ell\cap(\widetilde{g\cap i}))\setminus j}k_{g\cap\ell\cap(\widetilde{i\cap j})}$. Notice that $(h\cap k\cap (\widetilde{i\cap j}))\setminus g=(\widetilde{i\cap j})\setminus g$ as $j\setminus g\leq_2 g\oplus j\leq_2k$ and $i\setminus g\leq_2 h$. Moreover, notice that $h\cap k\cap (\widetilde{i\cap j})=((\widetilde{i\cap j})\setminus g)\cup(h\cap k\cap (\widetilde{g\cap i})\cap j)$. Lemma \ref{L;Lemma3.11} thus implies that $k_{h\cap k\cap (\widetilde{i\cap j})}=k_{(\widetilde{i\cap j})\setminus g}k_{h\cap k\cap (\widetilde{g\cap i})\cap j}$. As $i\setminus g\leq_2\ell$, $((\widetilde{i\cap j})\cap \ell)\setminus g=(\widetilde{i\cap j})\setminus g$. As $(\widetilde{i\cap j})\cap \ell=((\widetilde{i\cap j})\setminus g)\cup(g\cap\ell\cap(\widetilde{i\cap j}))$ by a direct computation, Lemma \ref{L;Lemma3.11} thus implies that $k_{(\widetilde{i\cap j})\cap \ell}=k_{(\widetilde{i\cap j})\setminus g}k_{g\cap\ell\cap(\widetilde{i\cap j})}$. The desired lemma thus follows from combining all displayed equalities of the valencies of elements in $\mathbb{S}$.
\end{proof}
\begin{lem}\label{L;Lemma8.8}
Assume that $g, h, i,j, k, \ell\in [0,d]$ and $m\!=\!(i\oplus j)\cup((\widetilde{i\cap j})\cap\ell)$. Assume that $i\setminus g\leq_2 h$, $g\oplus j\leq_2 k$, $\ell\!=\!(g\oplus i)\cup((\widetilde{g\cap i})\cap \ell)$, and $(\widetilde{g\cap i})\cap k\leq_2 h\leq_2 \ell$. Then $$\frac{k_{k\cap g\cap \ell}}{k_{i\cap\ell}}=\frac{k_{g\cap k}}{k_{i\cap m}}.$$
\end{lem}
\begin{proof}
As $\ell\!=\!(g\oplus i)\cup((\widetilde{g\cap i})\cap \ell)$ and $m\!=\!(i\oplus j)\cup((\widetilde{i\cap j})\cap\ell)$, a direct computation shows that $i\cap \ell=(i\setminus g)\cup ((\widetilde{g\cap i})\cap\ell)$ and $i\cap m=(i\setminus j)\cup((\widetilde{i\cap j})\cap\ell)$. Lemma \ref{L;Lemma3.11} thus implies that $k_{i\cap\ell}=k_{i\setminus g}k_{(\widetilde{g\cap i})\cap\ell}$ and $k_{i\cap m}\!=\!k_{i\setminus j}k_{(\widetilde{i\cap j})\cap\ell}$. According to Lemma \ref{L;Lemma8.7}, it suffices to check that $k_{i\setminus j}k_{(g\cap k)\setminus i}k_{h\cap k\cap (\widetilde{i\cap j})}=k_{i\setminus g}k_{g\cap k}$. According to Lemma \ref{L;Lemma3.11} and Notation \ref{N;Notation3.3}, it is not very difficult to notice that $k_q=k_{\widetilde{q}}$ for any $q\in [0,d]$.

As $g\setminus j\leq_2 g\oplus j\leq_2 k$, notice that $i\setminus j=((g\cap k\cap i)\setminus j)\cup(i\setminus(g\cup j))$. Lemma \ref{L;Lemma3.11} thus implies that $k_{i\setminus j}=k_{(g\cap k\cap i)\setminus j}k_{i\setminus(g\cup j)}$. As $i\setminus g\leq_2 h$ and $j\setminus g\leq_2 g\oplus j\leq_2 k$, notice that $(h\cap k\cap (\widetilde{i\cap j}))\setminus g=(\widetilde{i\cap j})\setminus g$. As $(\widetilde{g\cap i})\cap k\leq_2 h$, it is obvious that $h\cap k\cap g\cap(\widetilde{i\cap j})=h\cap k\cap (\widetilde{g\cap i})\cap j=g\cap k\cap(\widetilde{i\cap j})$. These equalities thus imply that $h\cap k\cap (\widetilde{i\cap j})=((\widetilde{i\cap j})\setminus g)\cup (g\cap k\cap(\widetilde{i\cap j}))$. Hence Lemma \ref{L;Lemma3.11} implies that $k_{h\cap k\cap (\widetilde{i\cap j})}=k_{(\widetilde{i\cap j})\setminus g}k_{g\cap k\cap(\widetilde{i\cap j})}=k_{(i\cap j)\setminus g}k_{g\cap k\cap i\cap j}$. The above discussion thus implies that $k_{i\setminus j}k_{(g\cap k)\setminus i}k_{h\cap k\cap (\widetilde{i\cap j})}\!=\!k_{(g\cap k\cap i)\setminus j}k_{i\setminus(g\cup j)}k_{(g\cap k)\setminus i}k_{(i\cap j)\setminus g}k_{g\cap k\cap i\cap j}$. It is obvious that $\mathbb{P}((g\cap k\cap i)\setminus j)$, $\mathbb{P}((g\cap k)\setminus i)$, and $\mathbb{P}(g\cap k\cap i\cap j)$ are pairwise disjoint. Notice that $g\cap k=((g\cap k\cap i)\setminus j)\cup ((g\cap k)\setminus i)\cup(g\cap k\cap i\cap j)$. Furthermore, notice that  $i\setminus g=((i\cap j)\setminus g)\cup(i\setminus(g\cup j))$. Hence $k_{i\setminus j}k_{(g\cap k)\setminus i}k_{h\cap k\cap (\widetilde{i\cap j})}=k_{i\setminus g}k_{g\cap k}$ by Lemma \ref{L;Lemma3.11}. The desired lemma thus follows.
\end{proof}
We end this section by presenting the generalizations of Lemmas \ref{L;Lemma6.13} and \ref{L;Lemma6.14}.
\begin{thm}\label{T;Theorem8.9}
Assume that $g, h, i,j, k,\ell\!\in\! [0,d]$ and $m\!=\!(i\oplus j)\cup((\widetilde{i\cap j})\cap h)$. Assume that $p_{gh}^i\!\neq\! 0$, $p_{jk}^\ell\!\neq\! 0$, and $p\nmid k_hk_k$. Then $B_{j, k, \ell}D_{g, h, i}\!\neq\! O$ if and only if $g=\ell$ and $(\widetilde{g\cap i})\cap k\leq_2 h$. Moreover, if $g=\ell$ and $(\widetilde{g\cap i})\cap k\leq_2 h$, then $D_{j, m, i}$ is defined and $B_{j, k, g}D_{g, h, i}=\overline{k_{g\cap k}}D_{j, m, i}$.
\end{thm}
\begin{proof}
If $g\!\neq\! \ell$, then $B_{j, k, \ell}D_{g, h, i}\!\!=\!\!B_{j, k, \ell}E_\ell^*E_g^*D_{g, h, i}\!=\!O$ by \eqref{Eq;2}. If $g=\ell$, then Lemma \ref{L;Lemma6.12} implies that
$B_{j, k, g}D_{g, h, i}\neq O$ only if $(\widetilde{g\cap i})\cap k\leq_2 h$. For the other direction, assume that $g\!=\!\ell$ and $(\widetilde{g\cap i})\cap k\leq_2 h$. Notice that $p_{jm}^i\!\neq\! 0$ by Lemmas \ref{L;Lemma4.4} and \ref{L;Lemma2.1}. Moreover, notice that $p\nmid k_m$ by Lemma \ref{L;Lemma8.3}. Therefore $D_{j, m, i}$ is defined by Notation \ref{N;Notation6.5}. As $p_{gh}^i\neq 0$ and $p_{jk}^g\neq 0$, Lemmas \ref{L;Lemma2.1} and \ref{L;Lemma4.3} thus imply that $i\setminus g\leq_2 g\oplus i\leq_2h$ and $g\oplus j\leq_2 k$.
Set $q=n_{h, g\odot i}=n_{\ell, i\odot j}$ by Lemma \ref{L;Lemma8.3}. If $r\in [0,q]$ and $s\in \mathbb{U}_{h, g\odot i, r}$, notice that $g\oplus i\leq_2 h\leq_2 s\leq_2 g\odot i$ and
$s=(g\oplus i)\cup((\widetilde{g\cap i})\cap s)$. By combining Lemmas \ref{L;Lemma4.20}, \ref{L;Lemma8.4}, \ref{L;Lemma8.6}, \ref{L;Lemma8.8}, and Notation \ref{N;Notation6.5}, the following computation holds:
\begin{align*}
B_{j, k, g}D_{g, h,i}=&\sum_{r=0}^q\sum_{s\in \mathbb{U}_{h, g\odot i, r}}(\overline{-1})^r\overline{k_{i\cap s}}^{-1}B_{j, k, g}B_{g, s, i}\\
=&\sum_{r=0}^q\sum_{s\in \mathbb{U}_{h, g\odot i, r}}(\overline{-1})^r\overline{k_{i\cap s}}^{-1}\overline{k_{k\cap g\cap s}}B_{j, m(j, k, g, s, i), i}\\
=&\overline{k_{g\cap k}}\sum_{r=0}^q\sum_{s\in \mathbb{U}_{m, i\odot j, r}}(\overline{-1})^r\overline{k_{i\cap s}}^{-1}B_{j,s,i}=\overline{k_{g\cap k}}D_{j,m,i}.
\end{align*}
The desired theorem thus follows from the above discussion and computation.
\end{proof}
\begin{thm}\label{T;Theorem8.10}
Assume that $g, h, i, j, k, \ell\in [0,d]$ and $m=(i\oplus j)\cup((\widetilde{i\cap j})\cap h)$. Assume that $p_{gh}^i\!\neq\! 0$, $p_{jk}^\ell\!\neq\!0$, and $p\nmid k_hk_k$. Then $D_{j, k, \ell}D_{g, h, i}\!\neq\! O$ if and only if $g=\ell$, $(\widetilde{g\cap i})\setminus h\leq_2 j$, and $k=(g\oplus j)\cup((\widetilde{g\cap j})\cap h)$. Moreover, if $g=\ell$, $(\widetilde{g\cap i})\setminus h\leq_2 j$, and $k=(g\oplus j)\cup((\widetilde{g\cap j})\cap h)$, then $D_{j, m, i}$ is defined and $D_{j, k, g}D_{g, h, i}=D_{j, m, i}$.
\end{thm}
\begin{proof}
Pick $B_{j, q, \ell}\in \mathrm{Supp}_\mathbb{B}(D_{j, k,\ell})$. As $p_{gh}^i\neq 0$ and $p_{jk}^\ell\!\neq\!0$, notice that $g\setminus h\leq_2 i$ and $j\oplus \ell\leq_2 k\leq_2 q\leq_2 j\odot\ell$ by Lemmas \ref{L;Lemma2.1} and \ref{L;Lemma4.3}. Assume that $g=\ell$, $(\widetilde{g\cap i})\setminus h\leq_2 j$, and $k=(g\oplus j)\cup((\widetilde{g\cap j})\cap h)$. So
$q=(g\oplus j)\cup((\widetilde{g\cap j})\cap q)$. So Lemma \ref{L;Lemma8.2} shows that $(\widetilde{g\cap i})\cap q\leq_2 h$ if and only if $q\!=\!k$. Hence $D_{j, k, g}D_{g, h, i}\!=\!D_{j, m, i}\neq O$ as $B_{j, q, g}$ is chosen from $\mathrm{Supp}_\mathbb{B}(D_{j, k,g})$ arbitrarily and Notation \ref{N;Notation6.5}, Theorem \ref{T;Theorem8.9} hold. For the other direction, assume that $D_{j, k, \ell}D_{g, h, i}\!\neq\! O$. Notation \ref{N;Notation6.5} and Theorem \ref{T;Theorem8.9} thus imply that $g\!=\!\ell$ and $(\widetilde{g\cap i})\cap k\leq_2(\widetilde{g\cap i})\cap r\!\leq_2\! h$ for some $B_{j, r, g}\!\in\!\mathrm{Supp}_\mathbb{B}(D_{j, k,g})$. Set $s=(g\oplus j)\cup((\widetilde{g\cap j})\cap h)$. As $k\!=\!(g\oplus j)\cup ((\widetilde{g\cap j})\cap k)$, Lemma \ref{L;Lemma8.2} thus shows that $(\widetilde{g\cap i})\setminus h\leq_2 j$ and $k\leq_2 s$. As $p\nmid k_hk_k$, $p\nmid k_s$ by Lemma \ref{L;Lemma3.11}. Assume that $k\neq s$. So $n_{k,s}>0$ by Notation \ref{N;Notation6.5}. Moreover, $\mathbb{U}_{k, s, 0}, \mathbb{U}_{k, s, 1},\ldots, \mathbb{U}_{k, s, n_{k,s}}$ are nonempty and pairwise disjoint by Notation \ref{N;Notation6.5}. Lemma \ref{L;Lemma8.2} says that $(\widetilde{g\cap i})\cap q\leq_2 h$ if and only if $k\leq_2 q\leq_2 s$. By combining Notation \ref{N;Notation6.5}, Lemma \ref{L;Lemma6.8}, Theorem \ref{T;Theorem8.9}, the Newton's Binomial Theorem, notice that $O\neq D_{j, k, g}D_{g, h, i}\!=\!(\sum_{t=0}^{n_{k,s}}\sum_{u\in \mathbb{U}_{k, s, t}}(\overline{-1})^t)D_{j,m,i}=O$. This is a contradiction. Therefore $k=s$. The desired theorem thus follows.
\end{proof}
\section{Structures of Terwilliger $\F$-algebras of factorial schemes II}
In this section, we get the algebraic structure of the semisimple algebra $\mathbb{T}/\mathrm{Rad}(\mathbb{T})$. We recall Notations \ref{N;Notation3.3}, \ref{N;Notation4.2}, \ref{N;Notation4.8}, \ref{N;Notation4.19}, \ref{N;Notation5.6}, \ref{N;Notation6.5}, \ref{N;Notation6.15}, \ref{N;Notation7.3} and present three lemmas.
\begin{lem}\label{L;Lemma9.1}
$\mathbb{T}$ has an $\F$-linearly independent subset $\{D_{a, b, c}\!\!: a\oplus b\!\leq_2\!\! c\!\!\leq_2\! a\odot b,\ p\!\nmid\! k_b\}$.
\end{lem}
\begin{proof}
Set $\mathbb{U}=\{D_{a, b, c}: a\oplus b\leq_2 c\leq_2 a\odot b,\ p\nmid k_b\}$. By Lemma \ref{L;Lemma4.4} and Notation \ref{N;Notation6.5}, notice that $M\!\neq\! O$ for any $M\in \mathbb{U}$. For any $D_{g, h, i}, D_{j, k, \ell}\!\in\!\mathbb{U}$, \eqref{Eq;2} and Theorem \ref{T;Theorem4.22} can imply that $D_{g, h, i}=D_{j, k, \ell}$ if and only if $g=j$, $h=k$, $i=\ell$. Assume that $\sum_{M\in\mathbb{U}}c_MM=O$ and $c_M\in \F$ for any $M\in \mathbb{U}$. Assume that $N\in \mathbb{U}$ and $c_N\neq \overline{0}$. So \eqref{Eq;2} and \eqref{Eq;3} can imply that $N=INI=E_m^*NE_q^*$ for some $m, q\in [0,d]$. Therefore $\mathbb{V}=\{M: M\in \mathbb{U},\ c_M\neq \overline{0},\ E_m^*ME_q^*\neq O\}\neq\varnothing$. Therefore there are $r\in\mathbb{N}_0\setminus\{0\}$ and $s_1, s_2,\ldots, s_r\in [0,d]$ such that the numbers $s_1, s_2,\ldots, s_r$ are pairwise distinct and $\mathbb{V}=\{D_{m, s_1, q}, D_{m, s_2, q}, \ldots, D_{m, s_r, q}\}$. If $r=1$, notice that $c_NN=O$ and $c_N=\overline{0}$ by \eqref{Eq;2}. This is a contradiction. Hence $r>1$. By Lemma \ref{L;Lemma3.1}, there is no loss to assume that $s_1$ is a minimal element of $\{s_1, s_2, \ldots, s_r\}$ with respect to $\leq_2$. By the choices of $s_1, s_2, \ldots, s_r$, notice that $D_{m, s_1, q}$ is an $\F$-linear combination of the elements in $\{D_{m, s_2, q}, D_{m, s_3, q},\ldots, D_{m, s_r, q}\}$. This is a contradiction by combining the choices of $s_1, s_2, \ldots, s_r$, Notation \ref{N;Notation6.5}, and Theorem \ref{T;Theorem4.22}. Therefore $c_M=\overline{0}$ for any $M\in\mathbb{U}$. The desired lemma thus follows.
\end{proof}
\begin{lem}\label{L;Lemma9.2}
$\mathbb{T}/\mathrm{Rad}(\mathbb{T})$ has an $\F$-basis $\{D_{a, b, c}+\mathrm{Rad}(\mathbb{T}): a\oplus b\leq_2 c\leq_2 a\odot b,\ p\nmid k_b\}$ whose cardinality is $|\{(a, b, c):a\oplus b\leq_2 c\leq_2 a\odot b,\ p\nmid k_b\}|$.
\end{lem}
\begin{proof}
Set $\mathbb{U}\!=\!\{D_{a, b, c}: a\oplus b\leq_2 c\leq_2 a\odot b,\ p\nmid k_b\}\cup\{B_{a, b, c}: a\oplus b\leq_2 c\leq_2 a\odot b,\ p\mid k_b\}$. The combination of Notation \ref{N;Notation6.5}, Theorem \ref{T;Theorem4.22}, and Lemma \ref{L;Lemma9.1} implies that $\mathbb{T}$ has an $\F$-basis $\mathbb{U}$. The desired lemma follows from Theorem \ref{T;Jacobson} and Notation \ref{N;Notation7.3}.
\end{proof}
\begin{lem}\label{L;Lemma9.3}
Assume that $g, h, i, j, k, \ell\in [0,d]$. If $p_{gh}^i\neq 0$, $p_{jk}^g\neq 0$, $p_{g\ell}^i\neq 0$, and $p\nmid k_hk_kk_\ell$, then $D_{j, k, g}D_{g, h, i}=D_{j,\ell, i}$ only if $(\widetilde{g\cap i})\setminus h=(\widetilde{g\cap j})\setminus k=(\widetilde{i\cap j})\setminus \ell$.
\end{lem}
\begin{proof}
As $D_{j, k, g}D_{g, h, i}=D_{j,\ell, i}\neq O$, Theorem \ref{T;Theorem8.10} says that $k=(g\oplus j)\cup((\widetilde{g\cap j})\cap h)$, $\ell=(i\oplus j)\cup((\widetilde{i\cap j})\cap h)$, and $(\widetilde{g\cap i})\setminus h\leq_2 j$. Notice that $(\widetilde{g\cap j})\setminus k=(\widetilde{g\cap j})\setminus h$ and $(\widetilde{i\cap j})\setminus \ell\!=\!(\widetilde{i\cap j})\setminus h$ by a direct computation. As $p_{gh}^i\neq 0$, Lemmas \ref{L;Lemma2.1} and \ref{L;Lemma4.3} thus imply that $g\setminus h\leq_2 i$ and $i\setminus h\leq_2 g$. As $(\widetilde{g\cap i})\setminus h\leq_2 j$, it is obvious that $(\widetilde{g\cap i})\setminus h\leq_2(\widetilde{g\cap j})\setminus h\leq_2 (\widetilde{i\cap j})\setminus h\leq_2(\widetilde{g\cap i})\setminus h$. Lemma \ref{L;Lemma3.1} thus implies that $(\widetilde{g\cap i})\setminus h=(\widetilde{g\cap j})\setminus k=(\widetilde{i\cap j})\setminus \ell$. The desired lemma thus follows.
\end{proof}
Lemma \ref{L;Lemma9.3} motivates us to introduce the following notation and another lemma.
\begin{nota}\label{N;Notation9.4}
\em Set $\mathbb{D}\!=\!\{(a, b, c): a\oplus b\leq_2 c\leq_2 a\odot b,\ p\nmid k_b\}$. Notice that $\mathbb{D}\neq\varnothing$ as $(0,0,0)\in \mathbb{D}$. Lemma \ref{L;Lemma9.1} thus implies that $D_{g, h, i}$ is defined for any $(g, h, i)\in \mathbb{D}$. If $(g, h, i), (j, k, \ell)\in \mathbb{D}$, write $(g, h, i)\approx(j, k, \ell)$ if and only if $(\widetilde{g\cap i})\setminus h=(\widetilde{j\cap \ell})\setminus k$. So $\approx$ is an equivalence relation on $\mathbb{D}$. There is $n_\approx\!\in\! \mathbb{N}_0\setminus\{0\}$ such that $\mathbb{D}_1, \mathbb{D}_2,\ldots, \mathbb{D}_{n_\approx}$ are exactly all equivalence classes of $\mathbb{D}$ with respect to $\approx$. Assume that $m\in [1, n_\approx]$.
Define $\mathbb{D}(m)=\{a: (a, b, a)\in \mathbb{D}_m\}$ and $\mathbb{I}(m)=\langle\{D_{a, b, c}+\mathrm{Rad}(\mathbb{T}): (a, b, c)\in \mathbb{D}_m\}\rangle_\F$.
\end{nota}
\begin{lem}\label{L;Lemma9.5}
Assume that $g\in [1, n_\approx]$. Then $\mathbb{I}(g)$ is a two-sided ideal of $\mathbb{T}/\mathrm{Rad}(\mathbb{T})$. Moreover, $\mathbb{T}/\mathrm{Rad}(\mathbb{T})=\bigoplus_{h=1}^{n_\approx}\mathbb{I}(h)$ as $\F$-linear spaces.
\end{lem}
\begin{proof}
As $\{\mathbb{D}_1, \mathbb{D}_2,\ldots, \mathbb{D}_{n_\approx}\}$ forms a partition of $\mathbb{D}$, the desired lemma thus follows from combining Theorem \ref{T;Theorem8.10}, Lemmas \ref{L;Lemma9.2}, \ref{L;Lemma9.3}, and Notation \ref{N;Notation9.4}.
\end{proof}
For further discussion, the following three combinatorial lemmas are necessary.
\begin{lem}\label{L;Lemma9.6}
Assume that $g\in [1, n_\approx]$ and $h, i, j, k\in [0,d]$. If $(h, i, j), (h, k, j)\in\mathbb{D}_g$, then $i=k$.
\end{lem}
\begin{proof}
As $(h, i, j), (h, k, j)\in\mathbb{D}_g$, Lemmas \ref{L;Lemma4.4} and \ref{L;Lemma2.1} imply that $p_{hj}^i\neq 0$ and $p_{hj}^k\neq 0$. So $i=(h\oplus j)\cup((\widetilde{h\cap j})\cap i)$ and $k=(h\oplus j)\cup((\widetilde{h\cap j})\cap k)$ by Lemma \ref{L;Lemma4.3}. Notice that $(\widetilde{h\cap j})\setminus i=(\widetilde{h\cap j})\setminus k$ and $(\widetilde{h\cap j})\cap i=(\widetilde{h\cap j})\cap k$. The desired lemma follows.
\end{proof}
\begin{lem}\label{L;Lemma9.7}
Assume that $g\in [1, n_\approx]$ and $h, i, j\in [0,d]$. If $(h, i, j)\in\mathbb{D}_g$,  then there exist $k, \ell\in [0,d]$ such that $(h, k, h), (j, \ell, j)\in \mathbb{D}_g$. In particular, $h, j\in \mathbb{D}(g)$.
\end{lem}
\begin{proof}
Set $k=\widetilde{h\cap i}$ and $\ell=\widetilde{i\cap j}$. As $(h, i, j)\in\mathbb{D}_g$, notice that $p_{hi}^j\neq 0$ and $p\nmid k_i$ by Lemma \ref{L;Lemma4.4}. Hence $h\setminus i\leq_2 j$ and $j\setminus i\leq_2 h$ by Lemmas \ref{L;Lemma2.1} and \ref{L;Lemma4.3}. Moreover, $p\nmid k_kk_\ell$ by Lemma \ref{L;Lemma3.11}. As $k\leq_2 \widetilde{h}$ and
$\ell\leq_2 \widetilde{j}$, The combination of Lemmas \ref{L;Lemma4.4}, \ref{L;Lemma2.1}, and \ref{L;Lemma4.3} thus implies that $(h, k, h), (j, \ell, j)\in \mathbb{D}$. Then $\widetilde{h}\setminus k=\widetilde{h}\setminus i$ and $\widetilde{j}\setminus \ell=\widetilde{j}\setminus i$ by a direct computation. Hence $\widetilde{h}\setminus k=\widetilde{h}\setminus i=(\widetilde{h\cap j})\setminus i=\widetilde{j}\setminus i=\widetilde{j}\setminus \ell$, which implies that $(h, k, h), (j, \ell, j)\in \mathbb{D}_g$.
So $h, j\in\mathbb{D}(g)$. The desired lemma follows.
\end{proof}
\begin{lem}\label{L;Lemma9.8}
Assume that $g\in [1, n_\approx]$ and $h, i, j, k\in [0,d]$. If $(h, i, h), (j, k, j)\in\mathbb{D}_g$, then there exists $\ell\in [0,d]$ such that $(h, \ell, j)\in\mathbb{D}_g$.
\end{lem}
\begin{proof}
As $(h, i, h), (j, k, j)\in\mathbb{D}_g$, notice that $\widetilde{h}\setminus i=\widetilde{j}\setminus k$ and $p\nmid k_ik_k$. Hence there is not $m\in\mathbb{P}(h\setminus(i\cup j))\cup\mathbb{P}(j\setminus(h\cup k))$ such that $u_m>2$ and $u_m\equiv 1\pmod p$. As $h\setminus j=((h\cap i)\setminus j)\cup(h\setminus(i\cup j))$ and $j\setminus h=((j\cap k)\setminus h)\cup(j\setminus(h\cup k))$. Lemma \ref{L;Lemma3.11} thus implies that $p\nmid k_{h\oplus j}$. Set $\ell=(h\oplus j)\cup((\widetilde{h\cap j})\cap k)$. The combination of Lemmas \ref{L;Lemma4.4}, \ref{L;Lemma2.1}, and \ref{L;Lemma4.3} thus implies that $h\oplus \ell\leq_2 j\leq_2 h\odot\ell$. Moreover, $p\nmid k_\ell$ by Lemma \ref{L;Lemma3.11}. Hence $(h, \ell, j)\in\mathbb{D}$. Notice that $(\widetilde{h\cap j})\setminus \ell=(\widetilde{h\cap j})\setminus k$. Notice that $(\widetilde{h\cap j})\setminus k\leq_2\widetilde{j}\setminus k\leq_2 (\widetilde{h\cap j})\setminus(i\cup k)\leq_2(\widetilde{h\cap j})\setminus k$. Lemma \ref{L;Lemma3.1} thus implies that $(\widetilde{h\cap j})\setminus \ell=(\widetilde{h\cap j})\setminus k=\widetilde{j}\setminus k$. Hence $(h, \ell, j)\in\mathbb{D}_g$. The desired lemma follows.
\end{proof}
Lemmas \ref{L;Lemma9.6}, \ref{L;Lemma9.7}, \ref{L;Lemma9.8} motivate us to present the following lemma and a notation.
\begin{lem}\label{L;Lemma9.9}
Assume that $g\in[1, n_\approx]$. Then the cartesian product $\mathbb{D}(g)\times \mathbb{D}(g)\neq\varnothing$. Furthermore, the map that sends $(h, i, j)$ to $(h, j)$ for any $(h, i, j)\!\in\!\mathbb{D}_g$ is a bijection from $\mathbb{D}_g$ to $\mathbb{D}(g)\times \mathbb{D}(g)$. In particular, the $\F$-dimension of $\mathbb{I}(g)$ equals $|\mathbb{D}(g)|^2$.
\end{lem}
\begin{proof}
As $\mathbb{D}_g\neq\varnothing$, notice that $\mathbb{D}(g)\times \mathbb{D}(g)\neq\varnothing$ by Lemma \ref{L;Lemma9.7}. Hence the mentioned map is defined. Moreover, notice that this mentioned map is a bijection by combining Lemmas \ref{L;Lemma9.6}, \ref{L;Lemma9.7}, \ref{L;Lemma9.8}. In particular, $|\mathbb{D}_g|=|\mathbb{D}(g)|^2$. The desired lemma thus follows as
$\mathbb{I}(g)=\langle\{D_{a, b, c}+\mathrm{Rad}(\mathbb{T}): (a, b, c)\in \mathbb{D}_g\}\rangle_\F$ and Lemma \ref{L;Lemma9.2} holds.
\end{proof}
\begin{nota}\label{N;Notation9.10}
\em Assume that $g\in[1, n_\approx]$. By combining Lemmas \ref{L;Lemma9.9}, \ref{L;Lemma9.2}, Notation \ref{N;Notation9.4}, there is a unique $D_{h, i, j}+\mathrm{Rad}(\mathbb{T})\in \{D_{a, b, c}+\mathrm{Rad}(\mathbb{T}): (a, b, c)\in \mathbb{D}_g\}$ for any $h, j\in \mathbb{D}(g)$. Denote this unique element by $D_{h, j}(g)$. In particular, notice that $\mathbb{I}(g)$ has an $\F$-basis $\{D_{a,b}(g): a, b\in\mathbb{D}(g)\}$ by combining Lemmas \ref{L;Lemma9.9}, \ref{L;Lemma9.2}, Notation \ref{N;Notation9.4}.
\end{nota}
The next lemma displays a computational result of the objects in Notation \ref{N;Notation9.10}.
\begin{lem}\label{L;Lemma9.11}
Assume that $g\in [1,n_\approx]$ and $h, i, j, k\in [0,d]$. If $h, i, j, k\in \mathbb{D}(g)$, then $D_{h, i}(g)D_{j, k}(g)=\delta_{ij}D_{h, k}(g)$.
\end{lem}
\begin{proof}
As $h, i, j, k\in \mathbb{D}(g)$ and Notation \ref{N;Notation9.10} holds, there exist $\ell, m, q\in [0,d]$ such that $D_{h, i}(g)=D_{h, \ell, i}+\mathrm{Rad}(\mathbb{T})$, $D_{j, k}(g)=D_{j, m, k}+\mathrm{Rad}(\mathbb{T})$, $D_{h, k}(g)=D_{h, q, k}+\mathrm{Rad}(\mathbb{T})$, $(h, \ell, i), (j, m, k), (h, q, k)\in\mathbb{D}_g$. If $i\neq j$, $D_{h,i}(g)D_{j, k}(g)\!=\!D_{h,i}(g)E_i^*E_j^*D_{j, k}(g)\!=\!O$ by \eqref{Eq;2}. Assume that $i=j$. As $(h, \ell, i), (i, m, k)\in \mathbb{D}(g)$, $(\widetilde{i\cap k})\setminus m=(\widetilde{h\cap i})\setminus \ell$. So $(\widetilde{i\cap k})\setminus m\leq_2 h$. As $(h, \ell, i), (i, m, k)\in \mathbb{D}(g)$, Lemmas \ref{L;Lemma4.4} and \ref{L;Lemma2.1} thus imply that $p_{i\ell}^h\neq 0$ and $p_{im}^k\neq0$. Lemmas \ref{L;Lemma2.1} and \ref{L;Lemma4.3} thus imply that $i\setminus \ell\leq_2 h$, $i\setminus m\leq_2 k$, and $\ell=(h\oplus i)\cup((\widetilde{h\cap i})\cap \ell)$. Notice that $\widetilde{i}\setminus \ell=(\widetilde{h\cap i})\setminus \ell=(\widetilde{i\cap k})\setminus m=\widetilde{i}\setminus m$. So  $(\widetilde{h\cap i})\cap \ell=(\widetilde{h\cap i})\cap m$. So $\ell=(h\oplus i)\cup((\widetilde{h\cap i})\cap m)$. Theorem \ref{T;Theorem8.10} thus implies that $( D_{h, \ell, i}+\mathrm{Rad}(\mathbb{T}))(D_{i, m, k}+\mathrm{Rad}(\mathbb{T}))\neq O+\mathrm{Rad}(\mathbb{T})$. and $D_{h, \ell, i}D_{i, m, k}=D_{h, r, k}$ for some $r\in [0,d]$. Notice that $(h, r, k)\in \mathbb{D}_g$ by combining Notation \ref{N;Notation6.5}, Lemmas \ref{L;Lemma4.3}, and \ref{L;Lemma9.3}. Therefore $q=r$ by Lemma \ref{L;Lemma9.6}. The desired lemma thus follows.
\end{proof}
The following lemma lets us get the algebraic structure of $\mathbb{I}(g)$ for any $g\in [1, n_\approx]$.
\begin{lem}\label{L;Lemma9.12}
Assume that $g\in [1, n_\approx]$. Then $\mathbb{I}(g)\cong\mathrm{M}_{|\mathbb{D}(g)|}(\F)$ as algebras.
\end{lem}
\begin{proof}
It suffices to check that $\mathbb{I}(g)\cong\mathrm{M}_{\mathbb{D}(g)}(\F)$ as algebras. For any $h, i\in\mathbb{D}(g)$, let $E_{h,i}$ be the $\{\overline{0}, \overline{1}\}$-matrix in $\mathrm{M}_{\mathbb{D}(g)}(\F)$ whose unique nonzero entry is the $(h, i)$-entry. Hence $E_{h,i}E_{j, k}=\delta_{ij}E_{h, k}$ for any $h, i, j, k\in \mathbb{D}(g)$. By Notation \ref{N;Notation9.10}, there exists an obvious $\F$-linear bijection from $\mathbb{I}(g)$ to $\mathrm{M}_{\mathbb{D}(g)}(\F)$ that sends $D_{h,i}(g)$ to $E_{h,i}$ for any $h, j\in\mathbb{D}(g)$. Notice that this $\F$-linear bijection is also an algebra isomorphism by Notation \ref{N;Notation9.10} and Lemma \ref{L;Lemma9.11}. The desired lemma thus follows.
\end{proof}
We conclude the whole paper by the main result of this section and an example.
\begin{thm}\label{T;Structure}
$\mathbb{T}/\mathrm{Rad}(\mathbb{T})\!\cong\! \bigoplus_{g=1}^{n_\approx}\mathrm{M}_{|\mathbb{D}(g)|}(\F)$ as algebras. Moreover, the number of all isomorphic classes of irreducible $\mathbb{T}$-modules equals $n_\approx$. In particular, the number of all isomorphic classes of irreducible $\mathbb{T}$-modules is independent of the choice of $\F$.
\end{thm}
\begin{proof}
The desired theorem follows from combining Lemmas \ref{L;Lemma9.5}, \ref{L;Lemma9.12}, and \ref{L;Lemma2.7}.
\end{proof}
\begin{eg}\label{E;Example9.13}
\em Assume that $p=n=u_1=2$ and $u_2=3$. Then $d=3$, $k_0=k_1=1$, and $k_2\!=\!k_3\!=\!2$ by Lemma \ref{L;Lemma3.11}. By combining Theorem \ref{T;Jacobson}, Notation \ref{N;Notation7.3}, Example \ref{E;Example4.23}, $\mathrm{Rad}(\mathbb{T})$ is an $\F$-linear space spanned by $B_{0,2,2}$, $B_{0,3,3}$, $B_{1,2,3}$, $B_{1,3,2}$, $B_{2,2,0}$, $B_{2,2,2}$, $B_{2,3,1}$, $B_{2,3,3}$, $B_{3, 2,1}$, $B_{3,2,3}$, $B_{3,3,0}$, $B_{3,3,2}$. Observe that $\mathbb{D}$ contains precisely $(0, 0, 0)$, $(0,1,1)$, $(1,0,1)$, $(1,1,0)$, $(2,0,2)$, $(2,1,3)$, $(3,0,3)$, $(3,1,2)$. Therefore $n_{\approx}=2$ and the two equivalence classes with respect to $\approx$ are $\{(0, 0, 0), (0,1,1), (1,0,1), (1,1,0)\}$ and $\{(2, 0, 2), (2,1,3), (3,0,3), (3,1,2)\}$. It is also obvious that $|\mathbb{D}(1)|=|\mathbb{D}(2)|=2$. According to Theorem \ref{T;Structure}, it is obvious that $\mathbb{T}/\mathrm{Rad}(\mathbb{T})\cong 2\mathrm{M}_2(\F)$ as algebras.
\end{eg}
\subsection*{Acknowledgements}
This present research was supported by the National Natural Science Foundation of China (Youth Program, No. 12301017) and Anhui Provincial Natural Science Foundation (Youth Program, No. \!2308085QA01). The author would like to gratefully thank Professor Tatsuro Ito for constant encouragements. He also gratefully thank Professor Gang Chen for some comments on this present research.

%\begin{thm}\label{T;Dimension}
%\end{thm}
%\begin{thm}\label{T;Center}
%\end{thm}
%\begin{thm}\label{T;Semisimplicity}
%\end{thm}
%\begin{thm}\label{T;Jacobson}
%\end{thm}

%\subsection*{Disclosure statement} There are no relevant financial or nonfinancial interests to report.
%\subsection*{Acknowledgements}
%The author thanks a referee for his or her helpful comments. The present work is supported by the Mathematical Center in Akademgorodok under Agreement No. 075-15-2019-1613 with the Ministry of Science and Higher Education of the Russian Federation.

\end{document}